\documentclass[a4paper,11pt]{amsart}
\usepackage{amsmath,amsthm,amssymb}
\usepackage[mathscr]{eucal}
\usepackage{upgreek}

\usepackage[bookmarks,bookmarksnumbered]{hyperref}

\usepackage{enumerate}
\usepackage{amsmath}
\usepackage{mathrsfs}
\usepackage{blindtext}
\usepackage{scrextend}
\usepackage{xcolor}

\usepackage{enumitem}
\usepackage{bm} 
\addtokomafont{labelinglabel}{\sffamily}
\usepackage{color}
\usepackage[at]{easylist}
\usepackage{leftindex}

\numberwithin{equation}{section}

\theoremstyle{plain}
\newtheorem{main}{Theorem}
\newtheorem{mcor}[main]{Corollary}

\newtheorem{theorem}{Theorem}[section]
\newtheorem{claim}[theorem]{Claim}
\newtheorem{conjecture}[theorem]{Conjecture}

\newtheorem{lemma}[theorem]{Lemma}

\newtheorem{corollary}[theorem]{Corollary}
\theoremstyle{definition}
\newtheorem{definition}[theorem]{Definition}
\newtheorem*{definition*}{Definition}

\newtheorem{remark}[theorem]{Remark}

\begin{document}

\title[A class of II$_1$ factors without non-trivial crossed product decompositions]
{A class of II$_1$ factors without non-trivial \\ crossed product decompositions}

\author[A. Fernández Quero]{Adriana Fernández Quero}

\address{Department of Mathematics, KU Leuven, Celestijnenlaan 200B, 3001 Leuven, Belgium}\email{adriana.fernandeziquero@kuleuven.be}

\author[A. Ioana]{Adrian Ioana}
\address{Department of Mathematics, University of California San Diego, 9500 Gilman Drive, La Jolla, CA 92093, USA}\email{aioana@ucsd.edu}

\author[H. Tan]{Hui Tan}

\address{Department of Mathematics, University of California, Los Angeles, 520 Portola Plaza, Los Angeles, CA 90095, USA}\email{tanhui@math.ucla.edu}

{\thanks{A.F.Q. was partially supported by FWO research project G016325N of the Research Foundation Flanders.}}
{\thanks{A.I. was supported by NSF grants DMS-2153805 and DMS-2451697, and the Presidential Chair in Mathematics}}

\begin{abstract}

We introduce a class of separable II$_1$ factors $M$ admitting no non-trivial crossed product decompositions: $M\not\cong B\rtimes_\sigma G$, for any trace preserving action $G\curvearrowright^\sigma (B,\tau)$ of an infinite countable group $G$ on a tracial von Neumann algebra $(B,\tau)$. These provide the first examples of II$_1$ factors that do not arise as crossed products of noncommutative dynamical systems. Our approach relies on a novel construction of separable II$_1$ factors $M$ whose embeddings into their tensor product square $M\overline{\otimes}M$ all
arise from the canonical embeddings $x\mapsto x\otimes 1$ and $x\mapsto 1\otimes x$. 

\end{abstract}
\maketitle

\section{Introduction}

\subsection{Background and motivation}
The crossed product construction is a central tool in the theory of von Neumann algebras,  providing a powerful connection between operator algebras and   dynamics. Every trace preserving action $G\curvearrowright^\sigma (B,\tau)$ of a countable discrete group $G$ on a tracial von Neumann algebra $(B,\tau)$ gives rise to the {\it crossed product von Neumann algebra} $B\rtimes_\sigma G$. This construction was introduced by Murray and von Neumann in their seminal articles \cite{MvN36,MvN43} in the cases $B=\mathbb C$ and, more generally,  $B=\text{L}^{\infty}(X)$, for a probability space $(X,\mu)$. 
In these cases, one obtains the group von Neumann algebra $\text{L}(G)$ and the group measure space von Neumann algebra $\text{L}^\infty(X)\rtimes_\sigma G$, respectively. 
The construction was later  extended in \cite{Tur58,NT58,Suz59} to arbitrary trace preserving actions $G\curvearrowright^\sigma(B,\tau)$. 
The resulting algebra $B\rtimes_\sigma G$ is again tracial and is often a II$_1$ factor, e.g., if $G$ is icc and  $\sigma$ is ergodic.
A fundamental problem is to understand to what extent general II$_1$ factors 
can be realized through the crossed product construction. 

The main goal of this paper is to give a definitive negative answer to this problem by constructing the first class of separable II$_1$ factors admitting no non-trivial crossed product decompositions.

Over the past five decades, this problem has received considerable attention and led to major advances. In the mid 1970s, Connes obtained the first examples of type II$_1$ factors \cite{Con75} (see also \cite{Jon80}) 
that are not isomorphic to group or group measure space von Neumann algebras. 
His argument shows that the factors $M$ studied, unlike group and group measure space von Neumann algebras, have no anti-automorphisms, i.e., $M\not\cong M^{\text{op}}$.  Two decades later, Voiculescu \cite{Voi95} used free entropy theory to prove that the free group factors $\text{L}(\mathbb F_n)$, for $2\leq n\leq\infty$, have no Cartan subalgebras (see \cite{Pop83b} for the case of uncountable free groups). Moreover, he showed that $\text{L}(\mathbb F_n)$ has no regular diffuse hyperfinite von Neumann subalgebras (equivalently, no regular diffuse amenable subalgebras by \cite{Con76}). Consequently, $\text{L}(\mathbb F_n)$ cannot be realized as a crossed product of the form $B\rtimes_\sigma G$, with $B$ diffuse and amenable. 

Since the early 2000s, Popa's deformation/rigidity theory has led to remarkable progress in the study of crossed product II$_1$ factors  (see the surveys \cite{Pop07,Vae10a,Ioa18}). We highlight several advances most relevant to the present paper, 
which dramatically expanded the list of examples of II$_1$ factors lacking various types of  crossed product decompositions. 
 Early results produced new examples of II$_1$ factors without anti-automorphisms, and hence without group measure space decompositions \cite{IPP05,PV06}. 
 Subsequently, Ozawa and Popa \cite{OP07}  showed that the free group factors $\text{L}(\mathbb F_n)$ satisfy a powerful structural property called strong solidity. This property is now known for numerous classes of factors $M$ (see, e.g., \cite{OP08,Hou09, Sin10,CS11,PV11,PV12,Ioa12,Iso12,BHV15,Cas18,DP23,DT25}). It requires that the normalizer of every diffuse amenable von Neumann subalgebra of $M$ is amenable, thereby precluding the existence of crossed product decompositions of the form $B\rtimes_\sigma G$ with $B$  diffuse amenable,  not only for $M$ but also for all of its nonamenable subfactors. 
 In a different direction, starting with \cite{Pet09,PV09,Ioa10} (see also \cite{CP10,Vae10b}),  several works exhibited classes of II$_1$ factors admitting at most one group measure space decomposition of the form $\text{L}^\infty(X)\rtimes_\sigma G$, where $G\curvearrowright^\sigma (X,\mu)$ is a free ergodic probability measure preserving action. 
 In parallel, beginning with \cite{Ioa10,IPV10} (see \cite{CIOS21,DV24} for recent work),  several families of II$_1$ factors were shown to not decompose as group von Neumann algebras.

Despite this progress, it remained open whether every separable II$_1$ factor can be realized as a crossed product $B\rtimes_{\sigma}G$, where an infinite group $G$ acts on a tracial von Neumann algebra $B$. Previous examples only ruled out such decompositions when $B$ is amenable. For non-amenable $B$, decompositions involving the trivial action $\sigma$ (yielding the tensor product $B\overline{\otimes}\text{L}(G)$) can be excluded for prime II$_1$ factors (see, e.g., \cite{Ge96,Oza03,Pet06,Pop06}). However, existing methods fail to handle arbitrary actions $\sigma$, and could not even rule out the possibility that every separable non-amenable II$_1$ factor can be written as $B\rtimes_\sigma G$, with $B=\text{L}(\mathbb F_n)$ a free group factor.

\subsection{Absence of crossed product decompositions}
We settle this problem by introducing a class of II$_1$ factors possessing only  trivial crossed product decompositions.
If $P,Q$ are II$_1$ factors, then $P$ {\it stably embeds} into $Q$ \cite{PV21} if there exists an embedding (i.e., a unital $*$-homomorphism) $\theta:P\rightarrow Q^t$, for some $t>0$, where $Q^t$ is the amplification of $Q$ by $t$.
If this holds when $t=1$, we simply say that $P$ {\it embeds} into $Q$. 

\begin{main}\label{no_crossed1} There exists a separable II$_1$ factor $M$ such that $M\not\cong B\rtimes_{\sigma} G$, for any trace preserving action $G\curvearrowright^\sigma(B,\tau)$ of any countable infinite group $G$ on any tracial von Neumann algebra $(B,\tau)$.

Moreover, there exist $2^{\aleph_0}$  pairwise non-stably embeddable separable II$_1$ factors  with this property.

\end{main}

Before discussing the  proof of Theorem \ref{no_crossed1}, we make several remarks on its statement.

\begin{remark}
\begin{enumerate}
\item The conclusion of Theorem \ref{no_crossed1} does not hold without restricting to infinite groups. 
For a finite group $G$, we have $\mathbb M_{|G|}(\mathbb C)\cong\mathbb B(\ell^2(G))=\ell^\infty(G)\rtimes_\lambda G $, where $\lambda$ denotes the left regular representation of $G$. Hence, every II$_1$ factor $M$ decomposes trivially as a crossed product by a finite group $G$: $(\star$) $M\cong \mathbb M_{|G|}(\mathbb C)\overline{\otimes} M^{1/|G|}\cong(\ell^\infty(G)\overline{\otimes} M^{1/|G|})\rtimes_{\lambda\otimes\text{Id}}G$.

\item  The II$_1$ factors $M$ from Theorem \ref{no_crossed1}  satisfy that any embedding $\theta: M\rightarrow M^t$, $t>0$, is trivial, see Corollary \ref{D}(1). 
This allows us to improve Theorem \ref{no_crossed1} and additionally show the following: if $M=B\rtimes_\sigma G$, for a trace preserving action $\sigma$ of a finite group $G$, then  $\sigma$ is of a specific form which generalizes ($\star$) and only depends on $G$ (see Theorem \ref{finite_cross}). 
This yields the first class of II$_1$ factors for which crossed product decompositions can be completely classified, exhibiting a new form of superrigidity for II$_1$ factors.

\item The crossed product construction 
was extended in \cite{Tak73} to 
 continuous actions $G\curvearrowright^\sigma B$ of locally compact groups $G$ on arbitrary von Neumann algebras $B$. The crossed product von Neumann algebra $B\rtimes_\sigma G$ can be tracial for non-discrete locally compact groups $G$, e.g., if $G$ is a  non-discrete SIN group acting trivially on a tracial von Neumann algebra $B$. However, if $B\rtimes_\sigma G$ is a II$_1$ factor, then $G$ must  be discrete, see Lemma \ref{non-discrete}.
Hence, if $M$ is a II$_1$ factor  as in Theorem \ref{no_crossed1}, then  $M\not\cong B\rtimes_{\sigma} G$, for any continuous action $G\curvearrowright^\sigma B$ of an infinite, locally compact group $G$ on a  von Neumann algebra $B$.
The II$_1$ factors provided by Theorem \ref{no_crossed1} 
are therefore not built from dynamical data 
in the strongest possible sense.

\item Let $N$ be a II$_1$ factor and $\mathcal U$ be a free ultrafilter on $\mathbb N$. Then any two diffuse separable abelian von Neumann subalgebras of $N^{\mathcal U}$ are unitarily conjugate \cite{Pop83b}. This endows $N^{\mathcal U}$ with various indecomposability properties, see for instance \cite{Pop83b, Pop13, GKEPT24}. In particular, $N^{\mathcal U}\not\cong B\rtimes_\sigma G$, for every infinite discrete group $G$ acting on a diffuse von Neumann algebra $B$ (see \cite[Proposition 3.17]{GKEPT24}). However, showing that $N^{\mathcal U}$ does not decompose as a group von Neumann algebra is open; see \cite[Remark 3.18]{GKEPT24} for positive evidence in this direction. Theorem \ref{no_crossed1} gives the first (separable or not) II$_1$ factors that are not crossed products by infinite groups.
\end{enumerate}
\end{remark}

\subsection{Rigidity for embeddings}
The proof of Theorem \ref{no_crossed1} follows the general strategy for studying crossed product decompositions of II$_1$ factors introduced in \cite{Ioa10,IPV10}. 
For a II$_1$ factor $M$, any decomposition $M=B\rtimes_\sigma G$ gives rise to an embedding $\Delta: M\rightarrow M\overline{\otimes}M$ \cite{PV09}, defined by $$\text{$\Delta(bu_g)=bu_g\otimes u_g$, for all $b\in B$ and $g\in G$.}$$ Note that $\Delta$ is a {\it comultiplication}, i.e., an embedding which is co-associative in the sense that $(\Delta\otimes\text{Id})\circ\Delta=(\text{Id}\otimes\Delta)\circ\Delta$.
One then aims to classify such embeddings $M\hookrightarrow M\overline{\otimes}M$. Ideally, this classification is precise enough to determine all triples $(B,G,\sigma)$ such that $M=B\rtimes_\sigma G$.

This strategy has been applied to great effect to study group and group measure space von Neumann algebra decompositions of II$_1$ factors.
It was used in \cite{Ioa10} to prove that Bernoulli actions  of icc property (T) groups  are W$^*$-superrigid.  
The same strategy was employed in \cite{IPV10} to give the first examples of countable icc groups which are W$^*$-superrigid 
and in all subsequent works providing new classes of W$^*$-superrigid groups (see, e.g., \cite{BV12,CI17,CDD21,CDD20,CIOS21,DV24}).

These results were obtained by classifying the comultiplication embeddings $\Delta:M\rightarrow M\overline{\otimes}M$  associated to the group or group measure space von Neumann decompositions $M=B\rtimes_\sigma G$ of the II$_1$ factors $M$ studied. Such classifications are possible when $B$ is trivial or abelian, in which case $\Delta$ behaves  analogously to the diagonal embedding, being {\it coarse} \cite{DV24} in the following sense: $\text{L}^2(M\overline{\otimes}M)$ is coarse as both $\Delta(M)$-$(M\overline{\otimes}1)$ and $\Delta(M)$-$(1\overline{\otimes}M)$ bimodules.

 However, $\Delta$ is coarse if and only if $B$ is amenable, and therefore this approach does not apply when $B$ is nonamenable. Moreover, it is generally believed that classifying arbitrary embeddings $M\hookrightarrow M\overline{\otimes}M$ is out of reach (see, e.g.,  \cite{DV25}). As such, while general embeddings $M\hookrightarrow M$ have been classified for several classes of II$_1$ factors $M$ \cite{Ioa10,Dep13,PV21,CIOS24}, not a single example of a II$_1$ factor $M$ for which all embeddings $M\hookrightarrow M\overline{\otimes}M$ are understood is currently known.

Our second main result, Theorem \ref{no_comult1}, provides the first class of II$_1$ factors $M$ for which all the embeddings $M\hookrightarrow M\overline{\otimes}M$ can be classified. The canonical embeddings $x\mapsto x\otimes 1$ and $x\mapsto 1\otimes x$ can be combined to build additional embeddings  $M\hookrightarrow M\overline{\otimes}M$ defined by $$x\mapsto u_1(x\otimes q_1)u_1^*+u_2(q_2\otimes x)u_2^*,$$
for given projections $q_1,q_2\in M$   and unitaries $u_1,u_2\in M\overline{\otimes}M$ with $u_1(1\otimes q_1)u_1^*+u_2(q_2\otimes 1)u_2^*=1$.

Theorem \ref{no_comult1} shows that any embedding $M\hookrightarrow M\overline{\otimes}M$ arises this way, thereby  establishing a new strong rigidity phenomenon for II$_1$ factors. Moreover, we prove that an analogous classification holds for all embeddings of $M$ into $M^{\mathbb N}=\overline{\otimes}_{n\in\mathbb N}M$ and furthermore into $M^{\mathbb N}\overline{\otimes}(M^{\text{op}})^{\mathbb N}$.

\begin{main}\label{no_comult1} There exists a separable II$_1$ factor $M$ such that, denoting $\mathcal M=M^{\mathbb N}\overline{\otimes}(M^{\emph{op}})^{\mathbb N}$ and by $\Delta_n:M\rightarrow\mathcal M$ the natural embedding into the $n$-th copy of $M$, for $n\in\mathbb N$, the following holds. 

For every embedding $\Delta:M\rightarrow\mathcal M$ we can find projections $q_n\in\Delta_n(M)'\cap \mathcal M$ and unitaries $u_n\in\mathcal M$, for every $n\in\mathbb N$, such that $\sum_{n\in\mathbb N}u_nq_nu_n^*=1$ and $$\text{$\Delta(x)=\sum_{n\in\mathbb N}u_n\Delta_n(x)q_nu_n^*$, $\forall x\in M$.}$$
Moreover, there exist $2^{\aleph_0}$  pairwise non-stably embeddable separable II$_1$ factors  with this property.

\end{main}
Specifically, we note that $\Delta_n:M\rightarrow\mathcal M$ is  defined by $$\Delta_n(x)=(1\otimes\cdots\otimes 1\otimes x\otimes 1\otimes\cdots)\otimes 1_{(M^{\emph{op}})^{\mathbb N}},$$ where $x\in M$ is placed in the $n$-th tensor coordinate of $M^{\mathbb N}$.

 The  comultiplication $\Delta:M\rightarrow M\overline{\otimes}M$ associated to a crossed product decomposition $M=B\rtimes_\sigma G$ with $G$ infinite 
 has no nonzero intertwiner with either of the embeddings  $x\mapsto x\otimes 1$ and $x\mapsto 1\otimes x$.
 Consequently, the II$_1$ factors from Theorem \ref{no_comult1} satisfy the conclusion of Theorem \ref{no_crossed1}. 
Moreover, such an $M$ is not even virtually isomorphic to any $\mathrm{II}_1$ factor $N$ that decomposes as a crossed product $N = B \rtimes_\sigma G$ with $G$ infinite.
Indeed, otherwise we could find a finite index inclusion $M \subset N^s$ for some $s > 0$. Combining this fact with the natural comultiplication on $N$ yields an embedding $\Delta \colon M \to (M \overline{\otimes} M)^t$ for some $t > 0$, which has no nonzero intertwiner with either of the embeddings $x \mapsto x \otimes 1$ or $x \mapsto 1 \otimes x$, thus contradicting Theorem \ref{no_comult1}.

For an outline of the proof of Theorem \ref{no_comult1}, we refer the reader to Subsection \ref{comments}. For now, we simply note that the II$_1$ factors $M$
are constructed as amalgamated free products $M=M_1*_RM_2$, where $M_1$ and $M_2$ are property (T) factors with a ``highly coarse" hyperfinite common subfactor $R$. The proof  that these factors satisfy the conclusion of Theorem \ref{no_comult1} relies crucially on results from \cite{IPP05}.

\subsection{Applications}
The proof of Theorem \ref{no_comult1}  can be used to strengthen Theorem \ref{no_crossed1} by allowing amplifications, twisting by scalar $2$-cocycles, and groupoid actions. 
Given a trace preserving action $G\curvearrowright^\sigma(B,\tau)$ of a countable group $G$ on a tracial von Neumann algebra $(B,\tau)$ and a $2$-cocycle $c\in Z^2(G,\mathbb T)$ one obtains the {twisted crossed product von Neumann algebra} $B\rtimes_{\sigma,c}G$. More generally, any trace preserving action $\mathcal G\curvearrowright^{\sigma}\mathcal B$ of a discrete pmp groupoid $\mathcal G$ on a measurable field of tracial von Neumann algebras $\mathcal B=(B_x,\tau_x)_{x\in \mathcal G^{(0)}}$ together with a $2$-cocycle $c\in Z^2(\mathcal G,\mathbb T)$ yields the twisted crossed product von Neumann algebra  $\mathcal B\rtimes_{\sigma,c}\mathcal G$ (see, e.g., \cite{Yam93,Cha24}).  When $B=\mathbb C$ and $\mathcal B=(\mathbb C)_{x\in \mathcal G^{(0)}}$, these  reduce to the twisted group von Neumann algebra $\text{L}_c(G)$ and the twisted groupoid von Neumann algebra $\text{L}_c(\mathcal G)$, respectively.

Examples of II$_1$ factors $N$  not isomorphic to twisted group von Neumann algebras were found in \cite[Corollary G]{Ioa10}, which further showed that no corner $pNp$ has this form. 
Only recently, \cite[Theorem B]{DV24} gave examples of II$_1$ factors $N$ none of whose amplifications $N^t$ is isomorphic to a group von Neumann algebra  or, more generally,  to a twisted groupoid von Neumann algebra.

We strengthen these results by providing II$_1$ factors whose amplifications cannot be realized as a twisted crossed product 
$\mathcal B\rtimes_{\sigma,c}\mathcal G$, for any action of an infinite groupoid  and any scalar $2$-cocycle.

\begin{mcor}\label{C}

There exists a separable II$_1$ factor $M$ such that $M^t\not\cong\mathcal B\rtimes_{\sigma,c}\mathcal G$, for any $t>0$,
any trace preserving action $\mathcal G\curvearrowright^\sigma\mathcal B$ of an infinite discrete pmp groupoid $\mathcal G$ on a measurable field $\mathcal B=(B_x,\tau_x)_{x\in\mathcal G^{(0)}}$ of tracial von Neumann algebras and any $2$-cocycle $c\in Z^2(\mathcal G,\mathbb T)$.

Moreover, there exist $2^{\aleph_0}$  pairwise non-stably embeddable separable II$_1$ factors  with this property.
\end{mcor}
In particular, Corollary \ref{C} implies that $$M^t\not\cong B\rtimes_{\sigma,c}G,$$   for any $t>0$, any trace preserving action $G\curvearrowright^\sigma(B,\tau)$ of a countable infinite group $G$ on a tracial von Neumann algebra $(B,\tau)$ and any $2$-cocycle $c\in Z^2(G,\mathbb T)$. This holds for any II$_1$ factor $M$ as in Theorem \ref{no_comult1}.
Indeed, a twisted crossed product decomposition $M=B\rtimes_{\sigma,c}G$, where $c\in Z^2(G,\mathbb T)$,
gives rise to an embedding $\Delta:M\rightarrow M\overline{\otimes}M\overline{\otimes}M^{\text{op}}$ \cite{Ioa10} defined by $$\text{$\Delta(bu_g)=bu_g\otimes u_g\otimes (u_{g^{-1}})^{\text{op}}$, for all $b\in B$ and $g\in G$.}$$
If $G$ is infinite, then $\Delta$ has no  nonzero intertwiner with either 
$x\mapsto x\otimes 1\otimes 1$ or $x\mapsto 1\otimes x\otimes 1$, 
hence Theorem \ref{no_comult1} forces $G$ to be finite.

To prove the full strength of Corollary \ref{C}, allowing for groupoid actions, we  additionally use the special amalgamated free product decomposition of our II$_1$ factors mentioned above. 

\begin{remark}
Corollary \ref{C} suggests even stronger indecomposability phenomena which we formalize in the following conjecture. The statement can be interpreted as a von Neumann algebraic analogue of Margulis' normal subgroup theorem.  

\begin{conjecture}\label{conj:margulisII_1}
There exists a separable II$_1$ factor $M$ such that every regular von Neumann subalgebra $P\subset M$ is either finite dimensional or of 
finite index.
\end{conjecture}

Corollary \ref{C} provides some evidence towards this conjecture. If $M^t$, for some $t>0$, had an infinite-dimensional regular subalgebra of infinite index, then $M^s$ would have an irreducible, regular subfactor of infinite index, for some $s\in (0,t]$ (see Corollary \ref{CC}(2)). Consequently,  $M^s$ would decompose as a cocycle crossed product $B\rtimes_{\sigma,c}G$, for a cocycle action $(\sigma,c)$ of an infinite group $G$. By Corollary \ref{C}, such decompositions cannot exist when the $2$-cocycle is scalar; more generally, Corollary \ref{DD}(2) rules out finite dimensional $2$-cocycles. However, our methods do not exclude infinite dimensional $2$-cocycles. More broadly, it remains open whether there exist any II$_1$ factors that admit no cocycle crossed product decompositions $B\rtimes_{\sigma,c}G$ with $G$ infinite.

\begin{conjecture}\label{conj:no-cocycle-crossed-product}
There exists a separable II$_1$ factor $M$ such that $M\not\cong B\rtimes_{\sigma,c}G$, for every countably infinite group $G$, every tracial von Neumann algebra $(B,\tau)$, and every trace preserving cocycle action $(\sigma,c)$ of $G$ on $B$.
\end{conjecture}

\end{remark}

Our next result records further embedding rigidity properties for the II$_1$ factors from Theorem \ref{no_comult1}.

\begin{mcor}\label{D}
Any II$_1$ factor $M$ satisfying the conclusion of Theorem \ref{no_comult1} has the following properties:
\begin{enumerate}

\item If $\Delta:M\rightarrow M^t$ is any embedding, for some $t>0$, then $t\in\mathbb N$ and there exists a unitary $u\in M^t=M\overline{\otimes}\mathbb M_t(\mathbb C)$ such that $\Delta(x)=u(x\otimes 1)u^*$, for every $x\in M$.
\item If $\Delta:M\rightarrow M\overline{\otimes}M$ is any embedding, then there exist projections $q_1,q_2\in M$  and unitaries $u_1,u_2\in M\overline{\otimes}M$ such that  $\Delta(x)=u_1(x\otimes q_1)u_1^*+u_2(q_2\otimes x)u_2^*$, for every $x\in M$.  

 \item If $\Delta:M\rightarrow M\overline{\otimes}M$ is any comultiplication, then there exists a unitary $u\in M\overline{\otimes}M$ such that either $\Delta(x)=u(x\otimes 1)u^*$, for every $x\in M$, or $\Delta(x)=u(1\otimes x)u^*$, for every $x\in M$.

\item $M^{\emph{op}}$ does not embed into $M^{\mathbb N}$.

\end{enumerate}
\end{mcor}

Examples of II$_1$ factors $M$ satisfying (1) were found in the unpublished preprint \cite{Dep13} via  an involved construction, while simpler examples were exhibited recently in \cite[Theorem D]{PV21}.  Corollary \ref{D}(1) gives a new class of II$_1$ factors with this property.

Corollary \ref{D} further provides the first  II$_1$ factors exhibiting any of the rigidity properties (2)-(4). 
Since (2) was discussed above, we highlight the novelty of (3) and (4). Recall that a {\it Hopf von Neumann algebra} is a pair $(M,\Delta)$ consisting of a von Neumann algebra $M$ 
 and a comultiplication $\Delta:M\rightarrow M\overline{\otimes}M$ \cite{Ern70,ES92}. Hopf von Neumann algebras generalize compact quantum groups, which  additionally require an 
 invariant faithful normal state, 
 see, e.g., \cite{DV25}. If $(M,\Delta)$ is a compact quantum group, then $M\cong M^{\text{op}}$. Thus, II$_1$ factors $M$ with $M\not\cong M^{\text{op}}$ admit no compact quantum group structure. 
 
 In contrast, every II$_1$ factor $M$ has trivial Hopf von Neumann algebra structures arising from the comultiplications $\Delta_1(x)=x\otimes 1$ and $\Delta_2(x)=1\otimes x$. Corollary \ref{D}(3) shows that there exist II$_1$ factors $M$ for which every comultiplication  is unitarily conjugate to $\Delta_1$ or $\Delta_2$. Furthermore,   any comultiplication $\Delta$ agrees with $\Delta_1$ or $\Delta_2$ on a subalgebra $N\subset M$  of essentially finite index in the sense of \cite{Vae07}, see  Corollary \ref{DD}(5).

While II$_1$ factors $M$ with $M^{\text{op}}\not\cong M$ have been known since \cite{Con75}, Corollary \ref{D} provides examples of II$_1$ factors $M$ such that $M^{\text{op}}$ does not even embed into $M$ and also not into $M^{\mathbb N}$. Although recent methods in \cite{DV24} might plausibly yield examples with $M^{\text{op}}\not\hookrightarrow M$, the existence of II$_1$ factors $M$ with the stronger non-embeddability $M^{\text{op}} \not\hookrightarrow M\overline{\otimes}M$ appears out of reach with existing techniques.

We next discuss applications to $C^*$-algebras. The construction of \cite{Con75} was used in \cite{Phi02} to produce a simple separable non-nuclear $C^*$-algebra $A$ with $A\not\cong A^{\text{op}}$. By \cite{PV10}, $A$ can be taken exact. 
While non-separable nuclear examples were given in \cite{FH17}, the existence of separable simple nuclear examples remains open. Since groupoid $C^*$-algebras are isomorphic to their opposites \cite{BS17}, any $C^*$-algebra $A$ with $A\not\cong A^{\text{op}}$ is not a groupoid $C^*$-algebra. Very recently, \cite{BGP26,Gao26} showed that $\mathbb B(H)$ is not  a twisted groupoid $C^*$-algebra for any infinite-dimensional $H$.

As a consequence of our results, we  obtain the first  simple separable $C^*$-algebras not isomorphic to twisted groupoid $C^*$-algebras, and also the first known simple $C^*$-algebras not embedding into their opposite algebras.
 More precisely:

\begin{mcor}\label{applicationsC}
There exists a simple  separable non-nuclear unital $C^*$-algebra $A$ which is stably finite, has a unique tracial state, stable rank one, and real rank zero,  and satisfies the following:
\begin{enumerate}
\item $A^{\emph{op}}$ does not embed into the minimal tensor product $\otimes_{n\in\mathbb N}A$.
\item $A$ is not isomorphic to the reduced crossed product $C^*$-algebra $B\rtimes_{\sigma,c,r}G$, for any action $G\curvearrowright^\sigma B$ of a countable infinite group $G$ on a $C^*$-algebra $B$ and any $2$-cocycle $c\in Z^2(G,\mathbb T)$.
\item $A$ is not isomorphic to the reduced $C^*$-algebra $C^*_r(\mathcal G,\Sigma)$, for any locally compact Hausdorff \'{e}tale groupoid $\mathcal G$ and any twist $\Sigma$ over $\mathcal G$.

\end{enumerate}
\end{mcor}

\subsection{A strengthening of Popa's coarseness theorem}
We now turn to the principal technical contribution of the paper. 
Although different in flavor from the previous results, we highlight it here because of its independent interest. 
Recall that a subfactor $P$ of a II$_1$ factor $M$ is  {\it coarse} if the $P$-bimodule $\text{L}^2(M)\ominus\text{L}^2(P)$ is isomorphic to a subbimodule of a  multiple of the coarse $P$-bimodule, $\text{L}^2(P)\otimes\text{L}^2(P)$.
A striking theorem of Popa \cite{Pop18b}  shows that every separable II$_1$ factor $M$ contains a coarse copy of the hyperfinite II$_1$ factor $R$. 

The proof of Theorem \ref{no_comult1} relies on the following strengthening of Popa's coarseness theorem.

\begin{main}\label{popa_extension}
Let $M$ be a separable II$_1$ factor and $(N,\tau)$ a separable tracial von Neumann algebra. Let $\mathcal H$ be a separable $M$-$M\overline{\otimes}N$-bimodule such that the $M$-$N$-bimodule $_M\mathcal H_{1\overline{\otimes}N}$ is left weakly mixing. 
Let $\mathcal K\subset\mathcal H$ be the closed subspace of vectors $\xi\in\mathcal H$ such that $x\xi=\xi (x\otimes 1)$,    for every $x\in M$.

Then there is a hyperfinite subfactor $R\subset M$ such that the   $R$-$R\overline{\otimes}N$-bimodule $\mathcal H\ominus \overline{\emph{sp}}(R\mathcal K)$ is coarse. \end{main}

Here,  $\overline{\text{sp}}(R\mathcal K)$  denotes the $\|\cdot\|_2$-closure of the linear span of $R\mathcal K$. We note that  $\overline{\text{sp}}(R\mathcal K)$ is naturally an 
$R$-$R\overline{\otimes}N$-bimodule, as used implicitly in Theorem \ref{popa_extension}. In fact, one can show that $\overline{\text{sp}}(R\mathcal K)$ is isomorphic, as an $R$-$R\overline{\otimes}N$-bimodule, to $\text{L}^2(R)\otimes\mathcal L$, for some right $N$-module $\mathcal L$.

For a more general version of Theorem \ref{popa_extension}, which treats additional types of bimodules needed in the proof of Theorem \ref{no_comult1}, see Theorem \ref{bimodulenonintertwining}. Their proof %of Theorem \ref{popa_extension}, and that of Theorem \ref{bimodulenonintertwining}, 
follows Popa's influential idea from \cite{Pop81,Pop83a} of recursively constructing a copy of the hyperfinite II$_1$ factor $R$ inside a given factor as the inductive limit of finite dimensional subalgebras that increasingly satisfy certain desired properties. This strategy has since found many applications, including in the proof of the coarseness theorem in \cite{Pop18b} and \cite{Pop19}.

The assumptions in Theorem \ref{popa_extension} are necessary. First, the separability assumption on $\mathcal H$  cannot be dropped. Indeed, the conclusion fails if $\mathcal H$ is the direct sum,  over all hyperfinite subfactors $R\subset M$, of the $M$-$M\overline{\otimes}N$-bimodules  $\text{L}^2(\langle M,e_R\rangle)\otimes \text{L}^2(N)$. Second, since $\mathcal H$ has to be assumed separable,  $M$ and $N$ must also be separable  whenever $\mathcal H$ is a faithful right $M\overline{\otimes}N$-module. Finally, the conclusion of Theorem \ref{popa_extension} fails if $\mathcal H$ not left weakly mixing as an $M$-$N$-bimodule. This can be seen by taking $\mathcal H=_{\alpha(M)}\text{L}^2(M\overline{\otimes}N)_{M\overline{\otimes}N}$, for any embedding $\alpha:M\rightarrow 1\overline{\otimes}N\subset M\overline{\otimes}N$.

The main novelty of Theorem \ref{popa_extension} lies in the case when $N\not=\mathbb C1$. However, the statement appears to be new even when  $N=\mathbb C1$. In that case,  Theorem \ref{popa_extension} asserts that for any separable $M$-bimodule $\mathcal H$, there exists a hyperfinite subfactor $R\subset M$ such that the $R$-bimodule $\mathcal H\ominus \overline{\text{sp}}(R\mathcal K)$ is coarse, where $\mathcal K\subset\mathcal H$ denotes the subspace of $M$-central vectors. For $\mathcal H=\text{L}^2(M)$,  this recovers Popa's coarseness theorem \cite{Pop18b}. The result is new if $\mathcal H$ has no nonzero $M$-central vectors, in which case it shows that $\mathcal H$ is coarse as a bimodule over some hyperfinite subfactor $R\subset M$. As a corollary, if $\theta_n:M\rightarrow M$ is any sequence of properly outer endomorphisms (e.g., outer automorphisms) of $M$, then there exists a hyperfinite subfactor $R\subset M$ such that the $R$-bimodule $_{\theta_n(R)}\text{L}^2(M)_R$ is coarse for every $n\in\mathbb N$. In particular, $\theta_n(R)\nprec_{M}R$, in the sense of Popa \cite{Pop03}, for every $n\in\mathbb N$.

Similarly, the general case of Theorem \ref{popa_extension} has the following consequence, which will be needed to prove Theorem \ref{no_comult1}. 
\begin{mcor}\label{corG}
Let $M$ be a separable II$_1$ factor. For $n\in\mathbb N$, let $\theta_n:M\rightarrow M\overline{\otimes}N_n$ be an embedding, with $(N_n,\tau_n)$ a separable tracial von Neumann algebra. 
Assume that there is no nonzero $v\in M\overline{\otimes}N_n$ such that $\theta_n(x)v=v(x\otimes 1)$, for every $x\in M$, and that $\theta_n(M)\nprec_{M\overline{\otimes}N_n}1\overline{\otimes}N_n$, for every $n\in\mathbb N$.

 Then there exists a hyperfinite subfactor $R\subset M$ such that the $R$-$R\overline{\otimes}N_n$-bimodule $_{\theta_n(R)}\emph{L}^2(M\overline{\otimes}N_n)_{R\overline{\otimes}N_n}$ is coarse, for every $n\in\mathbb N$. In particular, $\theta_n(R)\nprec_{M\overline{\otimes}N_n}R\overline{\otimes}N_n$, for every $n\in\mathbb N$.
\end{mcor}

For a more general version of Corollary \ref{corG}, see Theorem \ref{homomorphisms}.

\subsection{Comments on the proof of Theorem \ref{no_comult1}}\label{comments} We end the introduction with an outline of the proof of Theorem \ref{no_comult1}. For simplicity, we explain the construction in the case $\mathcal M=M^{\mathbb N}:=\overline{\otimes}_{n\in\mathbb N}M$.

The factors $M$ arise  as amalgamated free products $M=M_1*_RM_2$, where $M_1, M_2$ are II$_1$ factors with property (T) (for instance, $M_1=\text{L}(\text{PSL}_m(\mathbb Z))$ and  $M_2=\text{L}(\text{PSL}_n(\mathbb Z))$, for  $m,n\geq 3$) and $R$ is the hyperfinite II$_1$ factor.

An amalgamated free product of this form does not, in general, satisfy Theorem \ref{no_comult1}. The key point is that the conclusion does hold for suitably chosen embeddings of $R$ into $M_1$ and $M_2$.
More precisely, we construct  coarse copies of the hyperfinite II$_1$ factor $R_1\subset M_1$ and $R_2\subset M_2$ for which Theorem \ref{no_comult1} holds for $M_1*_{R_1=R_2}M_2$, whenever $R_1, R_2$ are identified via an isomorphism $R_1\cong R_2$.

To explain the construction, we introduce a new notion for embeddings of the form 
$$\text{$\theta:M_i\rightarrow (M_{j_1}\overline{\otimes}\cdots\overline{\otimes}M_{j_k})^t,$\;\;\; with $k\geq 1,i,j_1,\dots, j_k\in\{1,2\}$ and $t>0$.}$$ We say that $\theta$ is {\it non-degenerate} if no corner $\theta(M_i)$ embeds into a proper tensor subproduct, i.e.,
$$\text{$\theta(M_i)\nprec (M_{j_1}\overline{\otimes}\cdots\overline{\otimes}M_{j_{\ell-1}}\overline{\otimes}1\overline{\otimes}M_{j_{\ell+1}}\overline{\otimes}\cdots\overline{\otimes}M_{j_k})^t$, \;\;\; for all $1\leq \ell\leq k$.}$$
This condition is automatic for $k=1$.
If $k=1$ and $j_1=i$, we also call $\theta:M_i\rightarrow M_i^t$ {\it properly outer} if it has no nonzero intertwiner with the embedding $M_i\ni x\mapsto x\otimes 1\in M_i\overline{\otimes}\mathbb M_p(\mathbb C)$, for any  $p\geq t$.

The coarse hyperfinite subfactors $R_i\subset M_i$ are built such that the following condition holds.  Let $\theta:M_i\rightarrow (M_{j_1}\overline{\otimes}\cdots\overline{\otimes}M_{j_k})^t$ be a non-degenerate embedding, assumed moreover to be properly outer if $k=1$ and $j_1=i$. Then
  $$(\diamond)\;\;\;\;\;\text{$\theta(R_i)\nprec(M_{j_1}\overline{\otimes}\cdots\overline{\otimes}M_{j_{\ell-1}}\overline{\otimes}R_{j_\ell}\overline{\otimes}M_{j_{\ell+1}}\overline{\otimes}\cdots\overline{\otimes}M_{j_k})^t$, \;\;\; for every $1\leq \ell\leq k$.}$$
  
Corollary \ref{corG} allows us to construct coarse hyperfinite subfactors satisfying $(\diamond)$ for any given sequence $(\theta_n)$ of such embeddings, though not a priori for all embeddings simultaneously.

However, property (T) for $M_i$ provides the needed separability: the space of all (not necessarily unital) embeddings  $\theta:M_i\rightarrow (M_{j_1}\overline{\otimes}\cdots\overline{\otimes}M_{j_k})^t$ is separable with respect to the metric $$d(\theta,\theta')=\sup_{\|x\|\leq 1}\|\theta(x)-\theta'(x)\|_2,$$ for every $k\geq 1,i,j_1,\dots, j_k\in\{1,2\}$ and $t\in\mathbb N$  \cite{Pop86} (see also \cite{Con80a}). Since the set of embeddings $\theta$ that satisfy $(\diamond)$ is closed in the metric $d$,  we  obtain subfactors $R_i\subset M_i$ with the desired properties.

Having constructed $R_1$ and $R_2$, we fix an identification $R_1\cong R_2$ and define
$M=M_1*_{R_1=R_2}M_2$.
We now analyze an arbitrary embedding $\Delta:M\rightarrow M^{\mathbb N}$.
For $n\in\mathbb N$, denote by $\Delta_n:M\rightarrow M^{\mathbb N}$ the natural embedding into the $n^{\text{th}}$ tensor copy of $M$. For a finite set $F\subset\mathbb N$ and $j=(j_k)_{k\in F}\subset \{1,2\}^F$, define $M_{F,j}=\overline{\otimes}_{k\in F}\Delta_k(M_{j_k})$.

Fix $i\in\{1,2\}$. Using once more that $M_i$ has property (T) together with structural results from \cite{IPP05} on property (T) subfactors of amalgamated free products, 
we find a partition of unity  
$\{p^i_{F,j}\mid\text{$F\subset \mathbb N$ finite, $j\in \{1,2\}^F$}\}$ in $\Delta(M_i)'\cap M^{\mathbb N}$ such that $$\text{$\Delta(M_i)p_{F,j}^i\prec^s M_{F,j}$ \;\;\; but \;\;\; $\Delta(M_i)p_{F,j}^i\nprec M_{F',j_{|F'}}$, \;\;\; for every $F'\subsetneq F$.}$$

Whenever $p_{F,j}^i\not=0$,  these conditions yield  a non-degenerate embedding $\theta:M_i\rightarrow M_{F,j}^t$, for some $t>0$.
Applying $(\diamond)$ to $\theta$ and interpreting the conclusion in terms of $\Delta$ we deduce that 
$$(\dagger)\;\;\;\;\;\text{$\Delta(R_i)p_{F,j}^i\nprec (\overline{\otimes}_{k\in F\setminus\{\ell\}}\Delta_k(M_{j_k}))\overline{\otimes}\Delta_\ell(R_{j_\ell})$, for every $\ell\in F$,}$$
{\it unless} $F=\{k\}$ is a singleton, $j_k=i$, and $\Delta_{|M_i}$ has a nonzero intertwiner with ${\Delta_k}_{|M_i}$.

Finally, recall that we identified $R_1=R_2$ and thus $\Delta(R_1)=\Delta(R_2)$. 
The conditions from $(\dagger)$ are incompatible for $i=1$ and $i=2$. Consequently, the only remaining possibility is that the supports of the maximal intertwiners between $\Delta_{|M_i}$ and the embeddings $\{{\Delta_k}_{|M_i}\}_{k\in\mathbb N}$  form a partition of unity in $\Delta(M_i)'\cap M^{\mathbb N}$.

Using again the equality $\Delta(R_1)=\Delta(R_2)$, we deduce that $\Delta(x)=\sum_{k\in\mathbb N}u_k\Delta_k(x)q_ku_k^*$, for unitaries $(u_k)_{k\in\mathbb N}\subset M^{\mathbb N}$ and projections $p_k\in\Delta_k(M)'\cap M^{\mathbb N}$, for every $k\in\mathbb N$.

\vspace{2mm}

\subsection*{Organization of the paper.} 
 Following the introduction, the paper is organized into seven sections. Section \ref{Sec:preliminaries} collects preliminaries concerning bimodules, intertwining techniques, property (T) factors, and crossed products. In particular, it includes a criterion for the coarseness of bimodules (Lemma \ref{coarse-criterion}) which is essential in the proof of Theorem \ref{popa_extension}. 
 
  Sections \ref{Sec3:localquant}-\ref{Sec:hyperfinite}  are devoted to the proof of Theorem \ref{popa_extension}, in its more general form stated as Theorem \ref{bimodulenonintertwining}.
  In Section \ref{Sec3:localquant}, we extend Popa's well-known local quantization principle \cite[Lemma A.1]{Pop94} to general $M$-$M$ and $M$-$M^{\text{op}}$ bimodules.
  Combined with an integration argument introduced in Section \ref{Sec:integration}, this principle is used in Section \ref{Sec:weaklymixunitaries} to construct unitaries satisfying prescribed coarse and weak mixing conditions. Section \ref{Sec:hyperfinite} then applies Lemma \ref{one_unitary} to construct certain coarse hyperfinite subfactors, thereby completing the proof of Theorem \ref{popa_extension}.  
  
  Section \ref{Sec:embeddings} establishes a rigidity theorem (Theorem \ref{embeddings}), which is subsequently used in Section \ref{sec:proofmain} to derive the main Theorems \ref{no_crossed1} and \ref{no_comult1}, as well as  their Corollaries \ref{C}-\ref{applicationsC}.

\subsection*{Acknowledgements and AI tool disclosure} 
We would like to thank Ionu\c{t} Chifan, Milan Donvil and Sorin Popa for several insightful comments, and Yoonje Jeong for a number of helpful corrections.

%ChatGPT and Gemini were used for English language editing, proofreading, and grammatical corrections. All mathematical content was generated solely by the human authors.

ChatGPT and Gemini were used exclusively for English-language editing, proofreading, and grammatical corrections. No AI tools were used to generate, develop, verify, or otherwise assist with any part of the mathematics. All mathematical content was produced solely by the human authors.

\section{Preliminaries}\label{Sec:preliminaries}

\subsection{Tracial von Neumann algebras}
We first recall some basic facts concerning tracial von Neumann algebras, and refer the reader to \cite{AP17} for additional information.
 
 A {\it tracial von Neumann algebra} is a pair $(M,\tau)$ consisting of a von Neumann algebra endowed with a normal faithful tracial state (i.e., a trace) $\tau:M\rightarrow\mathbb C$. 

Let $(M,\tau)$ be a tracial von Neumann algebra. 
For $x\in M$, we denote by $\|x\|$ the {\it operator norm} of $x$ and by $\|x\|_p=\tau(|x|^p)^{\frac{1}{p}}$ the {\it $p$-norm} of $x$, for $p>0$, where $|x|=(x^*x)^{\frac{1}{2}}$.
We denote by $\mathcal U(M)$ the group of unitaries of $M$, by $\mathcal Z(M)=\{x\in M\mid xy=yx, \text{for every $y\in M$}\}$ the {\it center} of $M$, and by $(M)_1=\{x\in M\mid \|x\|\leq 1\}$ the {\it unit ball} of $M$.
We denote by $\text{L}^2(M)$ the Hilbert space obtained by completing $M$ with respect to $\|\cdot\|_2$. 
A {\it partition of unity} in $M$ is a finite family of projections $\{p_i\}_{i=1}^m\subset M$ with  $\sum_{i=1}^mp_i=1$. A partition of unity $\{q_j\}_{j=1}^n$ in $M$ {\it refines} $\{p_i\}_{i=1}^m$ if for every $1\leq i\leq m$ we have $p_i=\sum_{j\in F_i}q_j$, for some subset $F_i\subset\{1,\dots,n\}$.

The {\it opposite algebra} $M^{\text{op}}$ is a tracial von Neumann algebra which is equal to $M$ as a vector space, has the same involution and trace, and multiplication given by $x\cdot y=yx$. Thus, we can write $M^{\text{op}}=\{x^{\text{op}}\mid x\in M\}$ so that $x^{\text{op}}y^{\text{op}}=(yx)^{\text{op}}$.
We consider the standard representations $M, M^{\text{op}}\subset\mathbb B(\text{L}^2(M))$ given by the left and right multiplication actions of $M$ on $\text{L}^2(M)$.

Let $\mathcal H$ be a Hilbert space. Given $X\subset\mathcal H$, we denote by $\text{sp}(X)$ the {\it linear span} of $X$ and by $\overline{\text{sp}}(X)$ the {\it closure of the linear span} of $X$. 
Given $Y\subset\mathbb B(\mathcal H)$, we denote by $Y'$ its {\it commutant} and by $Y''=(Y')'$ its {\it double commutant}. By von Neumann's bicommutant theorem, if $Y$ is self-adjoint and contains the identity operator, then $Y''$ is the von Neumann algebra generated by $Y$.

Let $Q\subset M$ be a von Neumann subalgebra, which we will always assume to be unital  and endowed with the trace $\tau_{|Q}$. Thus, we have an inclusion $\text{L}^2(Q)\subset\text{L}^2(M)$. Let $e_Q\in\mathbb B(\text{L}^2(M))$  be the orthogonal projection from $\text{L}^2(M)$ onto $\text{L}^2(Q)$. The restriction of $e_Q$ to $M$ is equal to the unique trace preserving conditional expectation $\text{E}_Q:M\rightarrow Q$. The
von Neumann algebra $$\langle M,e_Q\rangle:=\{M,e_Q\}''\subset\mathbb B(\text{L}^2(M))$$ generated by $M$ and $e_Q$ is called {\it Jones' basic construction} of the inclusion $Q\subset M$. It is endowed with a normal faithful semifinite trace $\text{Tr}$ which satisfies $\text{Tr}(xe_Qy)=\tau(xy)$, for every $x,y\in M$. We denote by $\text{L}^2(\langle M,e_Q\rangle)$ the Hilbert space associated to $\langle M,e_Q\rangle$ and $\text{Tr}$.
We also recall that the {\it normalizer} of $Q$ in $M$ is defined as $\mathcal N_M(Q)=\{u\in\mathcal U(M)\mid uQu^*=Q\}$. We say that $Q$ is {\it regular} in $M$ if $\mathcal N_M(Q)''=M$.

Given tracial von Neumann algebras $(M,\tau)$ and $(\mathcal M,\tau)$, an {\it embedding} of $M$ into $\mathcal M$ is a unital $*$-homomorphism $\theta:M\rightarrow\mathcal M$ which is trace preserving, i.e., satisfies $\tau(\theta(x))=\tau(x)$, for every $x\in M$.
If $(N_i,\tau_i)_{i\in I}$ is a family of tracial von Neumann algebras, we denote by $\overline{\otimes}_{i\in I}N_i$ the tensor product von Neumann algebra endowed with the tensor product trace $\tau=\otimes_{\in I}\tau_i$. For a subset $J\subset I$, we view $\overline{\otimes}_{i\in J}N_i$ as a subalgebra of $\overline{\otimes}_{i\in I}N_i$ by identifying it with $(\overline{\otimes}_{i\in J}N_i)\overline{\otimes}(\overline{\otimes}_{i\in I\setminus J}1)$. If $N_i=N$, for every $i\in I$, we write $N^I$ instead of $\overline{\otimes}_{i\in I}N_i$.
If $(M_1,\tau_1)$ and $(M_2,\tau_2)$ are tracial von Neumann algebras with a common von Neumann subalgebra $B$ such that ${\tau_1}_{|B}={\tau_2}_{|B}$, we denote by $M_1*_BM_2$ the amalgamated free product von Neumann algebra endowed with its canonical trace $\tau$.

\subsection{Bimodules and completely positive maps}\label{bimm} Let $(M,\tau)$ and  $(N,\tau)$ be tracial von Neumann algebras.
A  {\it Hilbert $M$-$N$-bimodule} (which, hereafter, we call for simplicity an {\it $M$-$N$-bimodule}) is a Hilbert space $\mathcal H$ equipped with a unital $*$-homomorphism $\pi:M\otimes_{\text{alg}} N^{\text{op}}\rightarrow\mathbb B(\mathcal H)$ whose restrictions to $M\otimes 1$ and $1\otimes N^{\text{op}}$ are normal.
For $x\in M,y\in N$ and $\xi\in\mathcal H$ we denote $x\xi y=\pi(x\otimes y^{\text{op}})\xi$.
If $M=N$, we simply say that $\mathcal H$ is an $M$-bimodule. For example, the Hilbert spaces $\text{L}^2(M)$ and $\text{L}^2(\langle M,e_Q\rangle)$, for a von Neumann subalgebra $Q\subset M$, carry natural $M$-bimodule structures.

Let $\mathcal H$ be an $M$-$N$-bimodule.
The conjugate Hilbert space $\overline{\mathcal H}=\{\overline{\xi}\mid\xi\in\mathcal H\}$ carries a {\it contragradient} $N$-$M$-bimodule structure given by $$\text{$y\cdot \overline{\xi}\cdot x=\overline{x^*\xi y^*}$, for every $x\in M$ and $y\in N$.}$$

A vector  $\xi\in\mathcal H$ is called  {\it right bounded} if there exists $C>0$ such that $\|\xi y\|\leq C\|y\|_2$, for every $y\in N$,  {\it left bounded} if there exists $C>0$ such that $\|x\xi\|\leq C\|x\|_2$, for every $x\in M$, and {\it bounded} if it is left and right bounded.
The subspace of bounded vectors $\xi\in\mathcal H$  is a dense subspace of $\mathcal H$. 

Denote by $\mathcal H^0$ the set of right bounded vectors $\xi\in\mathcal H$.
Given $\xi\in\mathcal H^0$, we  define a bounded  operator $T_\xi:\text{L}^2(N)\rightarrow \mathcal H$ by letting $T_\xi(y)=\xi y$, for every $y\in N$.
If $\xi_1,\xi_2\in\mathcal H^0$ and $x\in M$, then the operator 
$T_{\xi_2}^*xT_{\xi_1}\in\mathbb B(\text{L}^2(N))$ commutes with the right multiplication action of $N$ on $\text{L}^2(N)$ and therefore belongs to $N$.  In other words, we have  $$\text{$T_{\xi_2}^*x T_{\xi_1}\in N$,  for every $\xi_1,\xi_2\in\mathcal H^0$ and $x\in M$}.$$ 

We next recall the correspondence between Hilbert modules and completely positive maps.
If $\xi\in\mathcal H^0$, then  $\Phi_\xi:M\rightarrow N$ given by $\Phi_\xi(x)=T_{\xi}^*xT_\xi$ is a normal, completely positive map such that 
$$\text{$\tau(\Phi_\xi(x)y)=\langle x\xi y,\xi\rangle$, for every $x\in M$ and $y\in N$.}$$ If $\|\xi y\|\leq \|y\|_2$, for every $y\in N$, then $\|T_\xi\|\leq 1$. Thus, $\Phi_\xi$ is subunital ($\Phi_\xi(1)=T_\xi^*T_\xi\leq 1$) and $\|\cdot\|$-contractive ($\|\Phi_\xi(x)\|\leq \|x\|$, for every $x\in M$). Assume that $\|x\xi\|\leq \|x\|_2$, for every $x\in M$. Then for every $x\in M$ we have $$\Phi_\xi(x^*)\Phi_\xi(x)=T_\xi^*x^*(T_\xi T_\xi^*)xT_\xi\leq T_\xi^* x^*x T_\xi=\Phi_\xi(x^*x),$$ 
hence  $\|\Phi_\xi(x)\|_2^2=\tau(\Phi_\xi(x^*)\Phi_\xi(x))\leq\tau(\Phi_\xi(x^*x))=\|x\xi\|^2\leq \|x\|_2^2$. Therefore, $\Phi_\xi$ is $\|\cdot\|_2$-contractive.

Conversely, let $\Phi:M\rightarrow N$ be a normal completely positive map. Then there exist an $M$-$N$-bimodule $\mathcal H_\Phi$  and a vector $\xi_\Phi\in\mathcal H_\Phi$ such that the linear span of $M\xi_\Phi N$ is dense in $\mathcal H_\Phi$ and  
$$\text{$\tau(\Phi(x)y)=\langle x\xi_\Phi y,\xi_\Phi\rangle$, for every $x\in M$ and $y\in N$.}$$
Assume that there exists $C>0$ such that $\tau\circ\Phi\leq C\tau$. Then the vector $\overline{\xi}_\Phi\in\overline{\mathcal H}_\Phi$ is right bounded: $\|\overline{\xi}_\Phi x\|=\|x^*\xi_\Phi\|=\sqrt{\tau(\Phi(xx^*))}\leq \sqrt{C}\|x\|_2$, for every $x\in M$.
By the above discussion, we obtain a normal completely positive map 
 $\Phi^*:N\rightarrow M$ (called the {\it adjoint} of $\Phi$) such that for every $x\in M$ and $y\in N$ we have $$\text{$\tau(x\Phi^*(y))=\langle y\overline{\xi}_\Phi x,\overline{\xi}_\Phi\rangle=\langle\overline{x^*\xi_\Phi y^*},\overline{\xi_\Phi}\rangle=\langle\xi_\Phi,x^*\xi_\Phi y^*\rangle=\langle x\xi_\Phi y,\xi_\Phi\rangle=\tau(\Phi(x)y)$.}$$
 By construction, we have an isomorphism of $N$-$M$-bimodules $\overline{\mathcal H}_\Phi\cong \mathcal H_{\Phi^*}$ which sends $\overline{\xi}_\Phi$ to $\xi_{\Phi^*}$.

Let $(P,\tau)$ be a tracial von Neumann algebra and $\mathcal K$ be an $N$-$P$-bimodule. The {\it Connes tensor product} of $\mathcal H$ and $\mathcal K$ is an $M$-$P$-bimodule denoted by $\mathcal H\otimes_N\mathcal K$ and obtained as the separation/completion of $\mathcal H^0\otimes_{\text{alg}}\mathcal K$ with respect to the scalar product
$$\langle\xi_1\otimes_N\eta_1,\xi_2\otimes_N\eta_2\rangle=\langle (T_{\xi_2}^*T_{\xi_1})
\eta_1,\eta_2\rangle$$
and endowed with the $M$-$P$-bimodule structure given by $x(\xi\otimes_N\eta)y=x\xi\otimes_N\eta y$.

If $Q\subset M$ is a von Neumann subalgebra, then the $M$-bimodule $\text{L}^2(\langle M,e_Q\rangle)$ is isomorphic to $\text{L}^2(M)\otimes_Q\text{L}^2(M)$.
If $\Phi:M\rightarrow N$ and $\Psi:N\rightarrow P$ are normal completely positive maps then there is an isomorphism of $M$-$P$-bimodules $\mathcal H_\Phi\otimes_N\mathcal H_\Psi\cong \mathcal H_{\Psi\circ\Phi}$ sending $\xi_\Phi\otimes_N\xi_\Psi$ to $\xi_{\Psi\circ\Phi}$.

\subsection{The space of central vectors}
Let $(M,\tau)$ be a tracial von Neumann algebra and $\mathcal H$ be an $M$-bimodule. We denote by $$\mathcal H^M:=\{\xi\in\mathcal H\mid x\xi=\xi x, \text{ for every $x\in M$}\}$$ the space of {\it $M$-central} vectors. In this subsection, we establish several basic properties of $\mathcal H^M$.

\begin{lemma}\label{fixed_point_proj}
Let $(M,\tau)$ be a tracial von Neumann algebra and $\mathcal H$ be an $M$-bimodule.

\begin{enumerate}
\item  $(p\mathcal H p)^{pMp}=p\mathcal H^Mp$, for every projection $p\in M$.  Thus, if $\mathcal H^M=\{0\}$, then  $(p\mathcal Hp)^{pMp}=\{0\}$. 
\item  $\mathcal H^{B'\cap M}=\emph{sp}(B\mathcal H^M)$, for every finite dimensional subfactor $B\subset M$.
\end{enumerate}
\end{lemma}

\begin{proof}  (1) Since we clearly have $p\mathcal H^Mp\subset (p\mathcal Hp)^{pMp}$, it remains to prove the opposite inclusion.  To this end, let $\eta\in (p\mathcal Hp)^{pMp}$. Let $z\in\mathcal Z(M)$ be the central support of $p$ and $(v_i)_{i\in I}\subset M$ be partial isometries such that $\sum_{i\in I}v_i v_i^*=z$, $v_i^*v_i\leq p$, for every $i\in I$, and $v_{i_0}=p$, for some $i_0\in I$. Let $\zeta=\sum_{i\in I}v_i \eta v_i^*\in z\mathcal H z$. If $x\in M$, then $v_i^*xv_{j}\in pMp$, for every $i,j\in I$, hence 
$$x\zeta =xz\zeta=zx\zeta=\sum_{i,j\in I}v_i (v_i^*xv_{j})\eta v_{j}^*.$$
Similarly, we derive that $\zeta x=\zeta zx=\zeta xz=\sum_{i,j\in I}v_i \eta(v_i^*xv_{j})v_{j}^*$, and thus $\zeta\in\mathcal H^M$.
Since $pv_i=v_{i_0}^*v_i=0$, for every $i\in I\setminus\{i_0\}$, we get that $\eta=p\zeta p\in p\mathcal H^Mp$,  which finishes the proof.

(2)  Since $B\mathcal H^M\subset \mathcal H^{B'\cap M}$, it is sufficient to show that $\mathcal H^{B'\cap M}\subset\text{sp}(B\mathcal H^M)$.
Let $n\geq 1$ such that $B\cong\mathbb M_n(\mathbb C)$ and $(e_{i,j})_{i,j=1}^n\subset B$ denote the usual matrix units. If $\eta\in\mathcal H$, then a direct calculation shows that $\sum_{i,j=1}^ne_{i,j}\eta e_{j,i}\in\mathcal H^B$. Hence, if $\eta\in\mathcal H^{B'\cap M}$, then $\sum_{i,j=1}^ne_{i,j}\eta e_{j,i}\in\mathcal H^B\cap\mathcal H^{B'\cap M}=\mathcal H^M$, where we used the fact that $\{B,B'\cap M\}''=M$.

Now, let $\xi\in\mathcal H^{B'\cap M}$ and fix $1\leq k,l\leq n$. Since $e_{1,k},e_{l,1}\in B$, we get that $e_{1,k}\xi e_{l,1}\in\mathcal H^{B'\cap M}$. 
The previous paragraph gives that $\sum_{i=1}^ne_{i,k}\xi e_{l,i}=\sum_{i,j=1}^ne_{i,j}(e_{1,k}\xi e_{l,1})e_{j,i}\in\mathcal H^M$. Since $e_{k,1},e_{1,l}\in B$ we get that $e_{k,k}\xi e_{l,l}=\sum_{i=1}^ne_{k,1}(e_{i,k}\xi e_{l,i})e_{1,l}\in \text{sp}(B\mathcal H^M)$. Since this holds for every $1\leq k,l\leq n$ and $\xi=\sum_{k,l=1}e_{k,k}\xi e_{l,l}$, we get that $\xi$ belongs to $\text{sp}(B\mathcal H^M)$, which finishes the proof.
\end{proof}

\begin{lemma}\label{fixed_point}
Let $(M,\tau)$ be a tracial von Neumann algebra, $\mathcal H$ an $M$-bimodule and  $e:\mathcal H\rightarrow\mathcal H^M$ the orthogonal projection.
Let $\xi\in\mathcal H$ and $\mathcal C\subset\mathcal H$ be the norm closure of the convex hull of the set $\{u\xi u^*\mid u\in\mathcal U(M)\}$. Let $\varepsilon>0$.
Then the following hold:

\begin{enumerate}
\item $e(\xi)$ is equal to the unique element of minimal norm in $\mathcal C$. 
\item If $\xi\in\mathcal H\ominus\mathcal H^M$, then there exists $u\in\mathcal U(M)$ such that $\Re\langle u\xi u^*,\xi\rangle\leq \varepsilon\|\xi\|^2$.
\item If $\xi\in\mathcal H\ominus \mathcal H^M$, then there exists a bounded vector $\xi'\in\mathcal H\ominus\mathcal H^M$ with $\|\xi'-\xi\|<\varepsilon$.
\end{enumerate}
\end{lemma}

\begin{proof} (1) Let $\eta$ be the element of minimal norm in $\mathcal C$. Since $\mathcal C$ is invariant under the norm preserving map $\rho\mapsto u\rho u^*$, we get that $u\eta u^*=\eta$, for every $u\in\mathcal U(M)$. Thus, $\eta\in\mathcal H^M$. On the other hand, since $\eta\in\mathcal C$, it is immediate that $\xi-\eta\perp\mathcal H^M$. This proves that $e(\xi)=\eta$.

(2) Assume that $e(\xi)=0$. Assume by contradiction that $\Re\langle u\xi u^*,\xi\rangle> \varepsilon\|\xi\|^2$, for every $u\in\mathcal U(M)$, so $\xi\not=0$. Then $\langle \rho,\xi\rangle\geq \varepsilon\|\xi\|^2$, for every $\rho\in\mathcal C$. In particular, since $e(\xi)\in\mathcal C$ by (1), we get that $\|e(\xi)\|^2=\langle e(\xi),\xi\rangle\geq \varepsilon\|\xi\|^2>0$, which  contradicts the fact that $e(\xi)=0$.

(3) Assume that $e(\xi)=0$  and let $\varepsilon>0$. Let $\xi_0\in\mathcal H$ be a bounded vector such that $\|\xi_0-\xi\|<\varepsilon$. Put $\xi'=\xi_0-e(\xi_0)$. Then $\xi'\in\mathcal H\ominus\mathcal H^M$ and $\|\xi'-\xi\|=\|(\text{Id}-e)(\xi_0-\xi)\|\leq \|\xi_0-\xi\|<\varepsilon$. 
Since $\xi_0$ is bounded, there exists $D>0$ such that $\|x\xi_0\|\leq D\|x\|_2$ and $\|\xi_0 x\|\leq D\|x\|_2$, for every $x\in M$. Since $e(\xi_0)$ belongs to the norm closure of the convex hull of the set $\{u\xi_0u^*\mid u\in\mathcal U(M)\}$ by (1), we get $\|xe(\xi_0)\|\leq D\|x\|_2$ and $\|e(\xi_0) x\|\leq D\|x\|_2$, for every $x\in M$. Thus, $e(\xi_0)$ and so $\xi'$ is bounded.
\end{proof}

\begin{lemma}\label{central_space}
Let $M$ be a II$_1$ factor and $\mathcal H$ be an $M$-bimodule. 
Then there exists an $M$-bimodular unitary operator $T:\overline{\emph{sp}}(M\mathcal H^M)\rightarrow \emph{L}^2(M)\otimes\ell^2(I)$, for a set $I$, such that 
$T(\mathcal H^M)=\mathbb C1\otimes\ell^2(I)$.
\end{lemma}

\begin{proof}
Let $\{\xi_i\}_{i\in I}\subset\mathcal H^M$ be a maximal family of unit vectors such that $M\xi_i\perp M\xi_{i'}$, for all $i,i'\in I$ with $i\not=i'$.
Put $\mathcal H_i=\overline{\text{sp}}(M\xi_i)$ and note that $\mathcal H_i$ is an $M$-subbimodule, for every $i\in I$.
 We claim that $\overline{\text{sp}}(M\mathcal H^M)=\bigoplus_{i\in I}\mathcal H_i$. Assume by contradiction that $\mathcal K:=\overline{\text{sp}}(M\mathcal H^M)\ominus(\bigoplus_{i\in I}\mathcal H_i)\not=\{0\}$. Let $e:\mathcal H\rightarrow\mathcal H^M$ be the orthogonal projection. By Lemma \ref{fixed_point}(1) we have that $e(\mathcal L)=\mathcal L^M$, for every $M$-subbimodule $\mathcal L\subset\mathcal H$. Thus, we derive that $\mathcal H^M=e(\mathcal H)=(\bigoplus_{i\in I}\mathcal H_i^M)\oplus\mathcal K^M,$ and further that $$\overline{\text{sp}}(M\mathcal H^M)=(\bigoplus_{i\in I}\overline{\text{sp}}(M\mathcal H_i^M))\oplus\overline{\text{sp}}(M\mathcal K^M).$$This implies that $\mathcal K^M\not=\{0\}$ since otherwise $\overline{\text{sp}}(M\mathcal H^M)=\bigoplus_{i\in I}\overline{\text{sp}}(M\mathcal H_i^M)\subset\bigoplus_{i\in I}\mathcal H_i$, which contradicts our assumption.
If $\xi\in\mathcal K^M$ is a unit vector, then $\xi\perp \mathcal H_i$ and thus $M\xi\perp M\xi_i$, for every $i\in I$. This contradicts the maximality of the family $\{\xi_i\}_{i\in I}$, which altogether proves our claim.

If $i\in I$, then the map $M\ni x\mapsto \langle x\xi_i,\xi_i\rangle\in\mathbb C$ is a trace. Since $M$ is a II$_1$ factor, we get that $\langle x\xi_i,\xi_i\rangle=\tau(x)$, for every $x\in M$. 
Thus, we have an $M$-bimodular unitary operator $T_i:\mathcal H_i\rightarrow\text{L}^2(M)$ such that $T_i(\xi_i)=1$.
Hence, $T=\oplus_{i\in I}T_i:\overline{\text{sp}}(M\mathcal H^M)=\bigoplus_{i\in I}\mathcal H_i\rightarrow \text{L}^2(M)\otimes\ell^2(I)=\bigoplus_{i\in I}\text{L}^2(M)$ is an $M$-bimodular unitary operator. 
Since $\mathcal H^M=\bigoplus_{i\in I}\mathcal H_i^M$, we get that $T(\mathcal H^M)=\mathbb C1\otimes\ell^2(I)$. \end{proof}

\subsection{Coarse and weakly mixing bimodules}
In this subsection we discuss the notions of coarse and weakly mixing bimodules. 
\begin{definition}
Let $(M,\tau)$ and $(N,\tau)$ be tracial von Neumann algebras. 
The Hilbert space $\text{L}^2(M)\otimes\text{L}^2(N)$ equipped with the $M$-$N$-bimodule structure given by $x(\xi\otimes \eta)y=x\xi\otimes y\eta$ is called the {\it coarse} $M$-$N$-bimodule. We say that an $M$-$N$-bimodule $\mathcal H$ is  {\it coarse} if it is isomorphic to a subbimodule of a multiple of the coarse bimodule, i.e., $\mathcal H\subset(\text{L}^2(M)\otimes\text{L}^2(N))\otimes\ell^2(I),$ for some set $I$. Equivalently, $\mathcal H$ is coarse if and only if its defining $*$-homomorphism $\pi:M\otimes_{\text{alg}} N^{\text{op}}\rightarrow\mathbb B(\mathcal H)$ extends to a normal $*$-homomorphism from the von Neumann algebra $M\overline{\otimes}N^{\text{op}}$.

A von Neumann subalgebra $P\subset M$ (more generally, an embedding $\theta:P\rightarrow M$) is called {\it coarse} if the $P$-bimodule $\text{L}^2(M)\ominus\text{L}^2(P)$ (respectively, the $\theta(P)$-bimodule $\text{L}^2(M)\ominus\text{L}^2(\theta(P))$) is coarse.
\end{definition}

We next record an elementary criterion for the coarseness of bimodules. This criterion will be used in the proof of Theorem \ref{bimodulenonintertwining} as a substitute of the criterion stated in \cite[Lemma 2.7.1]{Pop18b}.

\begin{lemma}\label{coarse-criterion}
Let $(M,\tau),(N,\tau),(P,\tau)$ be tracial von Neumann algebras. Let $(u_i)_{i\in I}\subset M$ and $(v_j)_{j\in J}\subset N$ be orthonormal bases for
 $\emph{L}^2(M)$ and $\emph{L}^2(N)$ whose linear spans are WOT-dense in $M$ and $N$. 
Let  $\mathcal H$ be an $M$-$N\overline{\otimes}P$-bimodule and $\{\xi_k\}_{k\in K}\subset\mathcal H$ a dense set. For $k\in K,i\in I, j\in J$, let $c_k(i,j)=\sup_{x\in (P)_1}|\langle u_i\xi_k(v_j\otimes x),\xi_k\rangle|$. Assume that $c_k\in\ell^1(I\times J)$, for every $k\in K$.

Then $\mathcal H$ is a coarse $M$-$N\overline{\otimes}P$-bimodule.
\end{lemma}

\begin{proof}

Let $k\in K, i\in I,j\in J$. Since the linear functional $P\ni x\mapsto  \langle u_i\xi_k(v_j\otimes x),\xi_k\rangle\in\mathbb C$ is normal, we can find $Y_k(i,j)\in\text{L}^1(P)$ such that  $\|Y_k(i,j)\|_1=c_k(i,j)$ and $\langle u_i\xi_k(v_j\otimes x),\xi_k\rangle=\tau(xY_k(i,j))$, for every $x\in P$. Since 
$\|u_i^*\otimes (v_j^*)^{\text{op}}\otimes Y_k(i,j)^{\text{op}}\|_1=\|u_i\|_1\|v_j\|_1\|Y_k(i,j)\|_1\leq \|u_i\|_2\|v_j\|_2 c_k(i,j)=c_k(i,j)$, for every $i\in I$ and $j\in J$, and $c_k\in \ell^1(I\times J)$,  we can define $Y_k\in\text{L}^1(M\overline{\otimes}(N\overline{\otimes}P)^{\text{op}})$ by letting 
$$Y_k=\sum_{i\in I,j\in J}u_i^*\otimes (v_j^*)^{\text{op}}\otimes Y_k(i,j)^{\text{op}}.$$
If $x\in P$, then since $\tau(u_iu_{i'}^*)=\delta_{i,i'}$ and $\tau(v_jv_{j'}^*)=\delta_{j,j'}$, for every $i'\in I$ and $j'\in J$, we get that
\begin{align*}\tau((u_i\otimes (v_j\otimes x)^{\text{op}})Y_k)&=\tau((u_i\otimes v_j^{\text{op}}\otimes x^{\text{op}})Y_k)\\&=\tau(x^{\text{op}}Y_k(i,j)^{\text{op}})=\tau(xY_k(i,j))=\langle u_i\xi_k(v_j\otimes x),\xi_k\rangle.\end{align*}
Since the linear spans of $(u_i)_{i\in I}$ and $(v_j\otimes P)_{j\in J}$ are WOT-dense in $M$ and $N\overline{\otimes}P$, respectively, we get that $\tau((a\otimes b^{\text{op}})Y_k)=\langle a\xi_k b,\xi_k\rangle$, for every $a\in M$ and $b\in N\overline{\otimes}P$. Hence, the $M$-$N\overline{\otimes}P$-bimodule $\overline{M\xi_k(N\overline{\otimes}P)}$ is coarse. Since $(\xi_k)_{k\in K}$ is dense in $\mathcal H$, we get that $\mathcal H$ is a coarse $M$-$N\overline{\otimes}P$-bimodule.
\end{proof}

\begin{definition}\label{weaklymixing}
Let $(M,\tau)$ and $(N,\tau)$ be tracial von Neumann algebras. 
An $M$-$N$-bimodule is called {\it left weakly mixing}  if 
one of the following three equivalent conditions holds:
\begin{enumerate}
\item $(\mathcal H\otimes_N\overline{\mathcal H})^M=\{0\}$.
\item There exists a net $(u_n)\subset\mathcal U(M)$ such that $\sup_{y\in (N)_1}|\langle u_n\xi y,\eta\rangle|\rightarrow 0$, for every $\xi,\eta\in\mathcal H$.
\item $\mathcal H$ contains no nonzero $M$-$N$-subbimodule $\mathcal K\subset\mathcal H$ such that $\text{dim}(\mathcal K_N)<\infty$.

\end{enumerate}
\end{definition}
Left weak mixingness for $M$-bimodules was defined by requiring (1) in \cite[Definition 1.3]{PS09}.
For the equivalence of (1)-(3), see \cite[Theorem A.2.2]{Bou14}.

\subsection{Intertwining by bimodules}

We recall from  \cite[Theorem 2.1 and Corollary 2.3]{Pop03} (see also \cite[Appendix C]{Vae06}) Popa's {\it intertwining-by-bimodules} theory.

\begin{theorem}[\cite{Pop03}]
\label{corner} Let $(M,\tau)$ be a tracial von Neumann algebra and $P\subset p_0Mp_0,Q\subset  q_0Mq_0$ be von Neumann subalgebras. Then the following conditions are equivalent
\begin{enumerate}
\item[(a)] There exist projections $p\in P, q\in Q$, a $*$-homomorphism $\varphi: pPp\rightarrow qQq$  and a nonzero partial isometry $v\in qM p$ such that $\varphi(x)v=vx$, for all $x\in pPp$.

\item[(b)] There exist $k\in\mathbb N$, a projection $r\in\mathbb M_k(\mathbb C)\overline{\otimes}Q$,  a $*$-homomorphism $\theta:P\rightarrow r(\mathbb M_k(\mathbb C)\overline{\otimes}Q)r$ and a nonzero partial isometry $w\in r(\mathbb M_{k,1}(\mathbb C)\overline{\otimes}M)$ such that $\theta(x)w=wx$, for every $x\in P$.

\item[(c)] There is no net $(u_n)\subset\mathcal U(P)$ satisfying $\|\emph{E}_Q(xu_ny)\|_2\rightarrow 0$, for all $x,y\in M$.

\end{enumerate}

If $Q$ is assumed unital, then conditions (a)-(c) are also equivalent to
\begin{enumerate}
\item[(d)] $(p_0\emph{L}^2(\langle M,e_Q\rangle) p_0)^P\not=\{0\}$.

\item [(e)] The $P$-$Q$-bimodule $\emph{L}^2(p_0M)$ is not left weakly mixing.
\end{enumerate}
\end{theorem}

If  (a)-(c) hold then we write $P\prec_M Q$ and say that a corner of $P$ embeds into $Q$ inside $M$.
Moreover, if $Pp'\prec_MQ$, for every nonzero projection $p'\in P'\cap p_0Mp_0$, then we write $P\prec_M^sQ$.
The equivalence of (d) and (e) follows from Definition \ref{weaklymixing} since $\text{L}^2(\langle M,e_Q\rangle)=\text{L}^2(M)\otimes_Q\text{L}^2(M)$.

As a consequence of Theorem \ref{local_quant} which we will prove in Section \ref{Sec3:localquant} we are able to add a new equivalence to Theorem \ref{corner}, which improves \cite[Lemma 1.4]{Ioa11}.
\begin{lemma}
\label{small_projections} Let $(M,\tau)$ be a tracial von Neumann algebra and $P\subset p_0Mp_0,Q\subset  q_0Mq_0$ be von Neumann subalgebras. Assume that $P$ is diffuse. Then $P\prec_M Q$ if and only if there is no net of nonzero projections $(p_n)\subset P$ satisfying $\frac{\|\emph{E}_Q(xp_ny)\|_2}{\|p_n\|_2}\rightarrow 0$, for every $x,y\in M$.

\end{lemma}

\begin{proof} Let $P\subset p_0Mp_0,	Q\subset q_0Mq_0$ be von Neumann subalgebras. If $P\prec_M Q$, it is easy to see that there is no net of nonzero projections $(p_n)\subset P$ satisfying $\frac{\|\text{E}_Q(xp_ny)\|_2}{\|p_n\|_2}\rightarrow 0$, for every $x,y\in M$. 

To prove the converse, assume that $P\nprec_MQ$. 
Since $P\nprec_MQ\oplus\mathbb C(1-q_0)$, after replacing $Q$ by $Q\oplus\mathbb C(1-q_0)$, we may assume that $Q\subset M$ is unital. 
Let $F\subset M$ be a finite set and $\varepsilon>0$. Put $\xi=\sum_{x\in F}p_0(xe_Qx^*)p_0\in p_0\text{L}^2(\langle M,e_Q\rangle)p_0$.
If $p\in P$ is a projection, then for every $x\in M$ we have $$\|\text{E}_Q(x^*py)\|_2^2=\text{Tr}(x^*pye_Qy^*pxe_Q)=\text{Tr}(p(ye_Qy^*)p(xe_Qx^*))=\langle p(ye_Qy^*)p,xe_Qx^*\rangle_{\text{Tr}}.$$
Using this identity we conclude that
\begin{equation}\label{xpy}\text{$\sum_{x,y\in F}\|\text{E}_Q(x^*py)\|_2^2=\langle p\xi p,\xi\rangle_{\text{Tr}}=\|p\xi p\|_{2,\text{Tr}}^2$, for every projection $p\in M$.}\end{equation} Since $P\nprec_MQ$, Theorem \ref{corner} gives that  $(p_0\emph{L}^2(\langle M,e_Q\rangle) p_0)^P\not=\{0\}$. Theorem \ref{local_quant} provides a partition of unity $\{p_i\}_{i=1}^m$ in $P$ such that $\|\sum_{i=1}^mp_i\xi p_i\|_{2,\text{Tr}}<\varepsilon$. Thus, $\sum_{i=1}^m\|p_i\xi p_i\|_{2,\text{Tr}}^2<\varepsilon^2=\varepsilon^2\sum_{i=1}^m\|p_i\|_2^2.$ Therefore, we can find $1\leq i\leq n$ such that $\|p_i\xi p_i\|_{2,\text{Tr}}^2<\varepsilon^2\|p_i\|_2^2$. In combination with \eqref{xpy}, we derive that $\sum_{x,y\in F}\|\text{E}_Q(x^*p_iy)\|_2^2<\varepsilon^2\|p_i\|_2^2$. In conclusion, $p_i\in P$ is a nonzero projection such that $\frac{\|\text{E}_Q(x^*p_iy)\|_2}{\|p_i\|_2}<\varepsilon$, for every $x,y\in F$. This finishes the proof.
\end{proof}

We next record three lemmas containing several intertwining results
 that will be needed later on. We start by collecting some useful facts from the literature. 

\begin{lemma}\label{elementary_facts}
Let $(M,\tau)$ be a tracial von Neumann algebra. Let $P\subset p_0Mp_0,P_0\subset P,Q\subset  q_0Mq_0$ and  $S_1,\dots,S_k\subset Q$  be von Neumann subalgebras. Then the following hold:

\begin{enumerate}

\item Assume that $P\prec_M^sQ$ and let  $\varepsilon>0$. Then there exist $k\in\mathbb N$, a projection $r\in\mathbb M_k(\mathbb C)\overline{\otimes}Q$,  a $*$-homomorphism $\theta:P\rightarrow r(\mathbb M_k(\mathbb C)\overline{\otimes}Q)r$ and a nonzero partial isometry $w\in r(\mathbb M_{k,1}(\mathbb C)\overline{\otimes}M)p_0$ such that $\theta(x)w=wx$, for every $x\in P$, and $\|w^*w-p_0\|_2\leq\varepsilon$.
\item  Assume that $P\prec_MQ$ and $P_0\nprec_MS_i$, for every $1\leq i\leq k$. Additionally, assume that every projection in $P$ is equivalent (in $P$) to a projection in $P_0$.
Then there exist projections $p\in P_0, q\in Q$,  a $*$-homomorphism $\varphi:pPp\rightarrow qQq$ and a nonzero partial isometry $v\in qM p$  such that $\varphi(x)v=vx$, for every $x\in pPp$, and $\varphi(pP_0p)\nprec_QS_i$, for every $1\leq i\leq k$.
\item Assume that $P\prec_MQ$. Then there exists a nonzero projection $z\in \mathcal N_{p_0Mp_0}(P)'\cap p_0Mp_0$ such that $Pz\prec_M^sQ$.
\end{enumerate}

\end{lemma}

\begin{proof}
For (1), see the proof of \cite[Lemma 3.7]{Vae07}.

In the case $P_0=P$, (2) is treated in \cite[Remark 3.8]{Vae07} and \cite[Lemma 5.4]{BCC24}. The general case follows  by adapting the arguments therein. First, since 
$P\prec_MQ$, there exist projections $p\in P, q\in Q$,  a $*$-homomorphism $\varphi:pPp\rightarrow qQq$ and a nonzero partial isometry $v\in qMp$  such that $\varphi(x)v=vx$, for every $x\in pPp$. Since $p$ is equivalent to some projection in $P_0$, we may assume that $p\in P_0$. Moreover, after replacing $v$ with $sv$ and $\varphi(\cdot)$ with $\varphi(\cdot)s$, where $s\in\varphi(pPp)'\cap qQq$ is the support projection of $\text{E}_{qQq}(vv^*)$, we may assume that the support projection of $\text{E}_{qQq}(vv^*)$ is equal to $q$. We claim that $\varphi(pP_0p)\nprec_Q S_i$, for every $1\leq i\leq k$.

Assume by contradiction that $\varphi(pP_0p)\prec_QS_i$, for some $1\leq i\leq k$. Then there exist a $*$-homomorphism $\psi:\varphi(pP_0p)\rightarrow t(\mathbb M_\ell(\mathbb C)\overline{\otimes}S_i)t$, for some $\ell\in\mathbb N$ and a projection $t\in\mathbb M_\ell(\mathbb C)\overline{\otimes}S_i$, and a nonzero partial isometry $w\in t(\mathbb M_{\ell,1}(\mathbb C)\overline{\otimes}Q)q$ such that $\psi(y)w=wy$, for every $y\in \varphi(pP_0p)$. Let $\theta=\psi\circ(\varphi_{|pP_0p}):pP_0p\rightarrow t(\mathbb M_\ell(\mathbb C)\overline{\otimes}S_i)t$. Then $\theta(x)wv=wvx$, for every $x\in pP_0p$. 
Since $w^*w\leq q$ and the support of $\text{E}_Q(vv^*)$ is equal to $q$, we get $\text{E}_Q(w^*wvv^*)=w^*w\text{E}_Q(vv^*)\not=0$. Hence $wv\not=0$. Since $wv\in t(\mathbb M_{\ell,1}(\mathbb C)\overline{\otimes}M)p$, we deduce that $P_0\prec_M S_i$, which is a contradiction.

For (3), see \cite[Lemma 2.4(3)]{DHI16}.
\end{proof}

\begin{lemma}\label{tensor_products}
Let $(M,\tau)$ and $(N,\tau)$ be tracial von Neumann algebras. Let $B\subset M$, $N_0\subset N$, $P\subset p(M\overline{\otimes}N_0)p$ and $Q\subset q(M\overline{\otimes}N)q$ be von Neumann subalgebras.

\begin{enumerate}
\item Assume that $P\nprec_{M\overline{\otimes}N_0}B\overline{\otimes}N_0$. Then $P\nprec_{M\overline{\otimes}N}B\overline{\otimes}N$. 
\item Assume that $M=M_1*_BM_2$, for some von Neumann subalgebras $M_1,M_2\subset M$. Assume additionally that $P\subset M_1\overline{\otimes}N_0$ and $P\nprec_{M_1\overline{\otimes}N_0}B\overline{\otimes}N_0$.
Then $P\nprec_{M\overline{\otimes}N}B\overline{\otimes}N$.
\item Assume that $M=M_1*_BM_2$, for some von Neumann subalgebras $M_1,M_2\subset M$. Assume additionally that $Q\prec^s_{M\overline{\otimes}N}M_1\overline{\otimes}N_0$ and $Q\prec_{M\overline{\otimes}N}^sM_2\overline{\otimes}N$. Then $Q\prec_{M\overline{\otimes}N}^s B\overline{\otimes}N_0$.
\end{enumerate}
\end{lemma}

\begin{proof}
(1) Since $P\nprec_{M\overline{\otimes}N_0}B\overline{\otimes}N_0$, there exists a net $(u_n)\subset\mathcal U(P)$ satisfying $\|\text{E}_{B\overline{\otimes}N_0}(au_nb)\|_2\rightarrow 0$, for every $a,b\in M\overline{\otimes}N_0$. We claim that $\|\text{E}_{B\overline{\otimes}N}(au_nb)\|_2\rightarrow 0$, for every $a,b\in M\overline{\otimes}N$. To prove the claim, we may assume that $a=a_1\otimes b_1$ and $b=a_2\otimes b_2$, for some $a_1,a_2\in M$ and $b_1,b_2\in N$. Since $(a_1\otimes 1)u_n(a_2\otimes 1)\in M\overline{\otimes}N_0$ and $\text{E}_{B\overline{\otimes}N}\circ\text{E}_{M\overline{\otimes}N_0}=\text{E}_{B\overline{\otimes}N_0}$, for every $n$, we have \begin{align*}\text{E}_{B\overline{\otimes}N}(au_nb)&=(1\otimes b_1)\text{E}_{B\overline{\otimes} N}((a_1\otimes 1)u_n(a_2\otimes 1))(1\otimes b_2)\\&=(1\otimes b_1)\text{E}_{B\overline{\otimes} N_0}((a_1\otimes 1)u_n(a_2\otimes 1))(1\otimes b_2).\end{align*}
Since $a_1\otimes 1,a_2\otimes 1\in M\overline{\otimes}1\subset M\overline{\otimes}N_0$, we have that $\|\text{E}_{B\overline{\otimes} N_0}((a_1\otimes 1)u_n(a_2\otimes 1))\|_2\rightarrow 0$. Hence, $\|\text{E}_{B\overline{\otimes}N}(au_nb)\|_2\rightarrow 0$, which proves the claim and the assertion.

(2) By applying part (1), we get that $P\nprec_{M_1\overline{\otimes}N}B\overline{\otimes}N$. Since we have that $P\subset M_1\overline{\otimes}N$ and $M\overline{\otimes}N=(M_1\overline{\otimes}N)*_{B\overline{\otimes}N}(M_2\overline{\otimes}N)$, \cite[Theorem 1.1]{IPP05} implies that $P\nprec_{M\overline{\otimes}N}B\overline{\otimes}N$.

(3) Suppose by contradiction that the conclusion fails. Thus, there exists a nonzero projection $q'\in Q'\cap q(M\overline{\otimes}N)q$ such that $Qq'\nprec_{M\overline{\otimes}N}B\overline{\otimes}N_0$. Since $Qq'\prec_{M\overline{\otimes}N}M_1\overline{\otimes}N_0$, by Lemma \ref{elementary_facts}(2) we can find projections $q_0\in Q$, $r\in M_1\overline{\otimes}N_0$, a nonzero partial isometry $v\in r(M\overline{\otimes}N)q_0q'$ and a $*$-homomorphism $\theta:q_0Qq_0q'\rightarrow r(M_1\overline{\otimes}N_0)r$ such that $\theta(x)v=vx$, for every $x\in q_0Qq_0q'$ and that $P:=\theta(q_0Qq_0q')$ satisfies $P\nprec_{M_1\overline{\otimes}N_0}B\overline{\otimes}N_0$. 

By (2) we deduce that $P\not\prec_{M\overline{\otimes}N}B\overline{\otimes}N$. Since $P\subset M_1\overline{\otimes}N$ and  $M\overline{\otimes}N=(M_1\overline{\otimes}N)*_{B\overline{\otimes}N}(M_2\overline{\otimes}N)$, \cite[Theorem 1.1]{IPP05} implies that $P\nprec_{M\overline{\otimes}N}M_2\overline{\otimes}N$. Since $v(q_0Qq_0q')v^*=Pvv^*$, we derive that $v(q_0Qq_0q')v^*\nprec_{M\overline{\otimes}N}M_2\overline{\otimes}N$, which contradicts the assumption that $Q\prec_{M\overline{\otimes}N}^sM_2\overline{\otimes}N$.
\end{proof}

\begin{lemma}\label{intertwiners} Let $(M,\tau)$ and $(\mathcal M,\tau)$ be tracial von Neumann algebras and $\alpha,\beta:M\rightarrow\mathcal M$ be embeddings.
Assume that $\beta(M)'\cap\mathcal M$ is a factor. Then there exists a projection $z\in\mathcal Z(\alpha(M)'\cap\mathcal M)$ such that the following two conditions hold:
\begin{enumerate}
\item There is a partial isometry $v\in\mathcal M$ satisfying $vv^*=z$ and $\alpha(x)v=v\beta(x)$, for every $x\in M$.
\item If $w\in\mathcal M$ satisfies $\alpha(x)w=w\beta(x)$, for every $x\in M$, then $(1-z)w=0$.
\end{enumerate}

\begin{proof}
Let $\{p_i\}_{i\in I}\subset\alpha(M)'\cap\mathcal M$ be a maximal family of pairwise orthogonal projections such that for every $i\in I$, there exists a partial isometry $v_i\in\mathcal M$ satisfying $v_iv_i^*=p_i$ and $\alpha(x)v_i=v_i\beta(x)$, for every $x\in M$. We will show that $z=\sum_{i\in I}p_i\in \alpha(M)'\cap\mathcal M$ has the desired properties.

First, assuming by contradiction that $z$ does not belong to the center of $\alpha(M)'\cap\mathcal M$, we can find $i\in I$ and a non-zero partial isometry $\zeta\in\alpha(M)'\cap\mathcal M$ such that $\zeta^*\zeta\leq p_i$ and $\zeta\zeta^*\leq 1-z$. Then $\xi=\zeta v_i$ is a nonzero partial isometry satisfying $\xi\xi^*\leq 1-z$ and $\alpha(x)\xi=\xi\beta(x)$, for every $x\in M$. Hence, the family of projections $\{p_i\}_{i\in I}\cup\{\xi\xi^*\}$ contradicts the maximality of the family $\{p_i\}_{i\in I}$.

Second, let $w\in\mathcal M$ such that $\alpha(x)w=w\beta(x)$, for every $x\in M$. Then $\eta=(1-z)w$ satisfies $\alpha(x)\eta=\eta \beta(x)$, for every $x\in M$. If $\delta\in\mathcal M$ is the partial isometry in the polar decomposition of $\eta$, then $\delta\delta^*\leq 1-z$ and  $\alpha(x)\delta=\delta \beta(x)$, for every $x\in M$. Thus, if $\eta\not=0$, then $\delta\not=0$, which would again contradict the maximality of the family $\{p_i\}_{i\in I}$. This proves (2).

Finally, since $v_i^*v_i\in\beta(M)'\cap \mathcal M$, for every $i\in I$,  $\sum_{i\in I}\tau(v_i^*v_i)=\sum_{i\in I}\tau(p_i)\leq 1$, and $\beta(M)'\cap \mathcal M$ is a factor, we can find unitaries $\{u_i\}_{i\in I}\subset\beta(M)'\cap \mathcal M$ such that the projections $\{u_i^*(v_i^*v_i)u_i\}_{i\in I}$ are pairwise orthogonal. Then $v=\sum_{i\in I}v_iu_i\in\mathcal M$ is a partial isometry which satisfies (1).
\end{proof}

\end{lemma}

\subsection{Property (T) for II\texorpdfstring{$_1$}{1} factors and local rigidity} A II$_1$ factor $M$ has {\it property (T)}  \cite{Con80b,CJ85} if one of the following two equivalent conditions holds:
\begin{enumerate}
\item There exist a finite set $S\subset M$ and $C>0$ such that if $\mathcal H$ is an $M$-bimodule which admits a vector $\xi\in\mathcal H$ satisfying $\max_{y\in S}\|y\xi-\xi y\|<C\|\xi\|$, then $\mathcal H^M\not=\{0\}$. 
\item For any sequence $\Phi_n:M\rightarrow M$ of unital, trace preserving, completely positive maps such that $\|\Phi_n(x)-x\|_2\rightarrow 0$, for every $x\in M$, we have that $\sup_{x\in (M)_1}\|\Phi_n(x)-x\|_2\rightarrow 0$.
\end{enumerate}
We refer the reader to \cite{Pop01} for the more general version of relative property (T) for inclusions. A  II$_1$ factor $M$ has property (T) if and only if the inclusion $M\subset M$ has relative property (T).

 If $G$ is a countable discrete ICC group with Kazhdan's property (T) (e.g., $G=\text{PSL}_n(\mathbb Z)$, for $n\geq 3$), then  $\text{L}(G)$ is a II$_1$ factor with property (T). Property (T) was first used in the context of von Neumann algebras by Connes in \cite{Con80b} to prove that   $\text{L}(G)$ has countable fundamental and outer automorphism groups, for every ICC property (T) group $G$. In \cite{Pop86}, Popa used property (T)  together with a separability argument to obtain further rigidity results. The combination of property (T) and separability arguments has since led to many rigidity statements up to countable classes (see \cite[Section 4]{Pop07}). Here we record the following well-known fact underpinning such arguments:
 \begin{definition}
 Let $(M,\tau)$ and $(\mathcal M,\tau)$ be tracial von Neumann algebras. 
  We denote by $\mathcal E(M,\mathcal M)$ the set of all embeddings $\theta:M\rightarrow p\mathcal Mp$, for some projection $p\in\mathcal M$, endowed with the metric $$\text{$\text{d}(\theta,\theta'):=\sup_{x\in (M)_1}\|\theta(x)-\theta'(x)\|_{2}$.}$$ 

 \end{definition}
 
 \begin{lemma}\label{sep_argument}
Let $M$ be a II$_1$ factor with property (T) and $(\mathcal M,\tau)$ be a separable tracial von Neumann algebra.
 Then $(\mathcal E(M,\mathcal M),\emph{d})$ is a separable metric space.
 \end{lemma}
 
 \begin{proof}

  Denote by $\mathcal E=\mathcal E(M,\mathcal M)$. Since $M$ has property (T), \cite[Proposition 1]{CJ85} implies the existence of a finite set $F\subset M$ and $C>0$ such that $1\in F$, $F\subset (M)_1$, and for every $M$-bimodule $\mathcal H$ we have 
 \begin{equation}\label{uniform}\text{$\sup_{x\in (M)_1}\|x\xi -\xi x\|\leq C\max_{y\in F}\|y\xi-\xi y\|$, for every $\xi\in \mathcal H$.}
 \end{equation}

 Let $\theta,\theta'\in\mathcal E$, put $p=\theta(1),p'=\theta'(1)\in\mathcal M$, and consider the Hilbert space $\mathcal H=\text{L}^2(pL^2(\mathcal M)p')$  with the $M$-bimodule structure given by $x\cdot \xi \cdot y=\theta(x)\xi\theta'(y)$. Applying \eqref{uniform} to $\xi=pp'\in\mathcal H$ we derive that
 \begin{equation}\label{gap}
 \text{$\sup_{x\in (M)_1}\|\theta(x)p'-p\theta'(x)\|_2\leq C\max_{y\in F}\|\theta(y)p'-p\theta'(y)\|_2$.}
 \end{equation}
 If $x\in (M)_1$ and $y\in (M)_1$, then $\|\theta(x)-\theta'(x)\|_2=\|\theta(x)p-p'\theta'(x)\|_2\leq \|\theta(x)p'-p\theta'(x)\|_2+2\|p-p'\|_2$ and $\|\theta(y)p'-p\theta'(y)\|_2\leq \|\theta(y)p-p'\theta'(y)\|_2+2\|p-p'\|_2= \|\theta(y)-\theta'(y)\|_2+2\|p-p'\|_2$. 
 Combining these two inequalities with \eqref{gap} and using that $1\in F$, we get that $$\text{d}(\theta,\theta')=\sup_{x\in (M)_1}\|\theta(x)-\theta'(x)\|_2\leq (3C+2)\max_{y\in F}\|\theta(y)-\theta'(y)\|_2.$$
 Since $\text{L}^2(\mathcal M)$ is separable and $F$ is finite, the last inequality implies that $(\mathcal E,\text{d})$ is separable.
 \end{proof}
 
 \begin{lemma}\label{dense_inter}
  Let $(M,\tau)$, $(\mathcal M,\tau)$ be tracial von Neumann algebras and $P\subset M$, $\mathcal P\subset \mathcal M$ be von Neumann subalgebras.
  Assume that there exists a dense subset $\mathcal E\subset\mathcal E(M,\mathcal M)$ such that $\theta(P)\nprec_{\mathcal M}\mathcal P$, for every $\theta\in\mathcal E$. 
  Then $\theta(P)\nprec_{\mathcal M}\mathcal P$, for every $\theta\in\mathcal E(M,\mathcal M)$. 
 \end{lemma}
 \begin{proof}
 Let $\theta\in\mathcal E(M,\mathcal M)$, $\varepsilon>0$, and $F\subset\mathcal M$ be a finite set and put $C=\max_{x\in F}\|x\|$. Let $\theta'\in\mathcal E$ such that $\text{d}(\theta,\theta')<\frac{\varepsilon}{2C^2+1}$.
 Since $\theta'(P)\nprec_{\mathcal M}\mathcal P$, we can find $u\in\mathcal U(P)$ such that $\|\text{E}_{\mathcal P}(x\theta'(u)y)\|_2<\frac{\varepsilon}{2}$, for every $x,y\in F$. Since for every $x,y\in F$, we have 
 $$\|\text{E}_{\mathcal P}(x(\theta(u)-\theta'(u))y)\|_2\leq \|x\| \|y\|\|\theta(u)-\theta'(u)\|_2\leq C^2\frac{\varepsilon}{2C^2+1}\leq\frac{\varepsilon}{2},$$ we derive that $\|\text{E}_{\mathcal P}(x\theta(u)y)\|_2<\varepsilon$, for every $x,y\in F$. Theorem \ref{corner} implies that $\theta(P)\nprec_{\mathcal M}\mathcal P$.
 \end{proof}

\subsection{Property (T) subfactors of tensor products of amalgamated free products} 
In this subsection, we establish the following result that will be needed later on.
\begin{theorem}\label{amalgamated} Let $K$ be a set. For every $k\in K$, let  $(M_{k,1},\tau_{k,1})$ and $(M_{k,2},\tau_{k,2})$ be tracial von Neumann algebras with a common von Neumann subalgebra $B_k$ such that ${\tau_{k,1}}_{|B_k}={\tau_{k,2}}_{|B_k}$, and put $M_k=M_{k,1}*_{B_k}M_{k,2}$. Define $\mathcal M=\overline{\otimes}_{k\in K}M_k$ and let $P\subset p\mathcal Mp$ be a II$_1$ subfactor with property (T). 

Then there exist a finite set $F\subset K$ and $i_k\in\{1,2\}$, for every $k\in F$, such that $P\prec_{\mathcal M}\overline{\otimes}_{k\in F} M_{k,i_k}$.
\end{theorem}

Theorem \ref{amalgamated} is an immediate consequence of the following result.
 
\begin{theorem}[\cite{IPP05}]\label{IPP}
Let $(M_1,\tau_1)$ and $(M_2,\tau_2)$ be tracial von Neumann algebras with a common von Neumann subalgebra $B$ such that ${\tau_1}_{|B}={\tau_2}_{|B}$. Denote $M=M_1*_BM_2$ and let $P\subset pMp$ be a von Neumann subalgebra with the relative property (T).
Then $P\prec_M M_1$ or $P\prec_M M_2$.
\end{theorem}

Theorem \ref{IPP} is proved in \cite[Theorem 4.3]{IPP05} under an additional assumption, and in \cite[Section 5]{Hou07} in full generality.

\begin{proof}[\bf Proof of Theorem \ref{amalgamated}] We first 
prove the conclusion when $K$ is finite by using induction on $|K|$. If $|K|=1$, the conclusion follows from Theorem \ref{IPP}. Assume that the conclusion holds if $|K|=n$, for some fixed $n\in\mathbb N$. Suppose that $|K|=n+1$ and by contradiction that  $P\nprec_{\mathcal M}\overline{\otimes}_{k\in K}M_{k,i_k}$, for any $(i_k)_{k\in K}\subset \{1,2\}^K$. After relabelling, we may write $K=\{1,2,\dots,n+1\}$.
Let $\mathcal N=\overline{\otimes}_{k=2}^{n+1}M_k$. 

Then $\mathcal M=M_1\overline{\otimes}\mathcal N=(M_{1,1}\overline{\otimes}\mathcal N)*_{B_{1}\overline{\otimes}{\mathcal N}}(M_{1,2}\overline{\otimes}\mathcal N)$.
By Theorem \ref{IPP} we can find $i_1\in\{1,2\}$ such that $P\prec_{\mathcal M}M_{1,i_1}\overline{\otimes}\mathcal N$.
Note that $P\nprec_{\mathcal M}\overline{\otimes}_{k=1}^{n+1}M_{k,i_k}$, for any choice of $i_2,\dots,i_{n+1}\in\{1,2\}$.
Applying Lemma \ref{elementary_facts}(2) provides nonzero projections $r\in P,q\in M_{1,i_1}\overline{\otimes}\mathcal N$ and a $*$-homomorphism $\varphi:rPr\rightarrow q(M_{1,i_1}\overline{\otimes}\mathcal N)q$ such that  $\varphi(rPr)\nprec_{M_{1,i_1}\overline{\otimes}\mathcal N}\overline{\otimes}_{k=1}^{n+1}M_{k,i_k}$,  for any  $i_2,\dots,i_{n+1}\in\{1,2\}$.

Let $M_{n+2}=M_{1,i_1}\overline{\otimes}M_2$ and note that $M_{n+2}=M_{n+2,1}*_{B_{n+2}}M_{n+2,2}$,
 where $M_{n+2,1}=M_{1,i_1}\overline{\otimes}M_{2,1}$, $M_{n+2,2}=M_{1,i_1}\overline{\otimes}M_{2,2}$ and $B_{n+2}=M_{1,i_1}\overline{\otimes}B_2$. Then $M_{1,i_1}\overline{\otimes}\mathcal N$ can be identified with $\overline{\otimes}_{k=3}^{n+2}M_k$.  Since $S=\varphi(rPr)$ is a property (T) subfactor of $q(M_{1,i_1}\overline{\otimes}\mathcal N)q$ with  $S\nprec_{M_{1,i_1}\overline{\otimes}\mathcal N}\overline{\otimes}_{k=1}^{n+1}M_{k,i_k}=\overline{\otimes}_{k=3}^{n+2}M_{k,i_k}$,  for any  choice $i_2,\dots,i_{n+1},i_{n+2}\in\{1,2\}$ with $i_{n+2}=i_2$, this contradicts the induction assumption.

Finally, assume that $K$ is infinite.
Suppose by contradiction that $P\nprec_{\mathcal M}\overline{\otimes}_{k\in F} M_{k,i_k}$, for every finite subset $F\subset K$ and any $i_k\in\{1,2\}$, for every $k\in F$.
 Since $P$ has property (T), it is separable (see \cite[Theorem 4.4.1]{Pop86}). Thus, we can find a countable infinite set $K_0\subset K$ such that denoting $\mathcal M_0=\overline{\otimes}_{k\in K_0}M_k$ we have
 $p\in\mathcal M_0$ and $P\subset p\mathcal M_0p$. After relabelling, we may assume that $K_0=\mathbb N$. If $\text{E}_n$ denotes the conditional expectation from $\mathcal M_0$ onto $\overline{\otimes}_{k=1}^nM_k$, for $n\in\mathbb N$, then
 $\|\text{E}_n(x)-x\|_2\rightarrow 0$, for every $x\in\mathcal M_0$. Since $P$ has property (T), we can find $n\in\mathbb N$ such that $\|\text{E}_n(x)-x\|_2<\frac{\|p\|_2}{2}$, for every $x\in (P)_1$. Thus, $\|\text{E}_n(u)\|_2>\frac{\|p\|_2}{2}$, for every $u\in\mathcal U(P)$, which by Theorem \ref{corner} implies $P\prec_{\mathcal M_0}\overline{\otimes}_{k=1}^nM_k$. Note  also that $P\nprec_{\mathcal M_0}\overline{\otimes}_{k=1}^n M_{k,i_k}$, for every choice of $i_k\in\{1,2\}$, for every $1\leq k\leq n$. Lemma \ref{elementary_facts}(2) then gives  nonzero projections $s\in P$, $t\in \overline{\otimes}_{k=1}^nM_k$ and a $*$-homomorphism $\psi:sPs\rightarrow t(\overline{\otimes}_{k=1}^nM_k)t$ such that $\psi(sPs)\nprec_{\overline{\otimes}_{k=1}^nM_k}\overline{\otimes}_{k=1}^n M_{k,i_k}$, for every choice of $i_k\in\{1,2\}$, for every $1\leq k\leq n$.  Since $\psi(sPs)$ is a property (T) subfactor of $\overline{\otimes}_{k=1}^nM_k$, this contradicts the case when $K$ is finite treated above.
\end{proof}

\subsection{Cocycle actions and crossed products}\label{cocycle_section} We start by recalling the construction of crossed product von Neumann algebras associated to cocycle actions. Let $(B,\tau)$ be a tracial von Neumann algebra and $G$ be a discrete group. We denote by $\text{Aut}(B)$ the group of automorphisms of $B$, i.e., of trace preserving $*$-isomorphisms $\theta:B\rightarrow B$.
An {\it action} of $G$ on $B$ is a homomorphism $\sigma:G\rightarrow\text{Aut}(B)$. 
We denote by $B\rtimes_\sigma G$ the associated {\it crossed product von Neumann algebra}.

A {\it cocycle action} of $G$ on $B$ is a pair $(\sigma,c)$ consisting of maps $\sigma:G\rightarrow\text{Aut}(B)$ and $c:G\times G\rightarrow\mathcal U(B)$ such that  $\sigma_e=\text{Id}_B$ and for every $g,h,k\in G$ we have $$\text{$\sigma_g\sigma_h=\text{Ad}(c_{g,h})\sigma_{gh}$ \;\;\; and\;\;\;  $c_{g,h}c_{gh,k}=\sigma_g(c_{h,k})c_{g,hk}$.}$$
The map  $c$  is called the {\it $2$-cocycle} of the cocycle action $(\sigma,c)$, and is said to be {\it normalized} if $c_{g,e}=c_{e,g}=1$, for every $g\in G$.

Let  $(\sigma,c)$ be a cocycle action of $G$ on $B$, whose $2$-cocycle $c$ is normalized.
We denote by $B\rtimes_{\sigma,c}G$ the associated {\it cocycle crossed product von Neumann algebra} (see, e.g., \cite[Section 1.1]{Pop18a}). Let $\mathcal H=\text{L}^2(B)\otimes\ell^2(G)$, viewed as a left $B$-module and consider the induced embedding $B\subset \mathbb B(\mathcal H)$.
Then $B\rtimes_{\sigma,c}G\subset\mathbb B(\mathcal H)$ is generated by $B$ and unitary operators  $\{u_g\}_{g\in G}\subset\mathcal U(\mathcal H)$  given by $u_g(b\otimes \delta_h)=\sigma_g(b)c_{g,h}\otimes\delta_{gh}$. It is easy to check that for every $g,h\in G$ and $b\in B$ we have $$\text{$u_gbu_g^*=\sigma_g(b)$\;\;\; and \;\;\; $u_gu_h=c_{g,h}u_{gh}$. }$$
Moreover, $B\rtimes_{\sigma,c}G$ is a tracial von Neumann algebra with its trace given by $\widetilde\tau(x)=\langle x(1\otimes\delta_e),1\otimes \delta_e\rangle$. Then  $\widetilde\tau(bu_g)=\tau(b)\delta_{g,e}$, for every $b\in B$ and $g\in G$. Consequently,  $\text{L}^2(B\rtimes_{\sigma,c}G)=\bigoplus_{g\in G}\text{L}^2(B)u_g$ and every $x\in B\rtimes_{\sigma,c}G$ admits a Fourier decomposition $x=\sum_{g\in G}b_gu_g$, where the convergence holds in $\|\cdot\|_2$ and $b_g=\text{E}_B(xu_g^*)\in B$, for every $g\in G$.

Let $\sigma:G\rightarrow\text{Aut}(B)$ be an action. Let $c:G\times G\rightarrow\mathbb T$ be a normalized $\mathbb T$-valued $2$-cocycle, i.e., a map  satisfying $c_{g,h}c_{gh,k}=c_{g,hk}c_{h,k}$ and $c_{g,e}=c_{e,g}=1$, for every $g,h,k\in G$. If we view $\mathbb T=\mathcal U(\mathbb C)\subset\mathcal U(B)$, then $(\sigma,c)$ is a cocycle action. The cocycle crossed product $B\rtimes_{\sigma,c}G$ reduces to
\begin{itemize}
\item the usual crossed product $B\rtimes_\sigma G$ if  $c$ is trivial, i.e., $c_{g,h}=1$, for every $g,h\in G$, and
\item the twisted group von Neumann algebra $\text{L}_c(G)$ if $B=\mathbb C$.
\end{itemize}

If $M=B\rtimes_\sigma G$ is a crossed product, we have a {\it comultiplication} embedding $\Delta:M\rightarrow M\overline{\otimes}M$  given by $\Delta(bu_g)=bu_g\otimes u_g$, for every $b\in B$ and $g\in G$ \cite{PV09}. If $M=\text{L}_c(G)$ is a twisted group von Neumann algebra, then we have a {\it triple comultiplication} embedding $\Delta:M\rightarrow M\overline{\otimes}M\overline{\otimes}M^{\text{op}}$ given by $\Delta(u_g)=u_g\otimes u_g\otimes {(u_{g^{-1}})}^{\text{op}}$, for every $g\in G$ \cite{Ioa10}. 
Generalizing the last case, assume that $M=B\rtimes_{\sigma,c}G$, where $\sigma$ is an action of $G$ on $B$ and $c:G\times G\rightarrow\mathbb T$ is a normalized $2$-cocycle. Then we also have a triple comultiplication $\Delta: M\rightarrow M\overline{\otimes}M\overline{\otimes}M^{\text{op}}$ given by $$\text{$\Delta(bu_g)=bu_g\otimes u_g\otimes {(u_{g^{-1}})}^{\text{op}}$, \;\;\; for every $b\in B$ and $g\in G$}.$$
If $G$ is infinite and we put $\mathcal M=M\overline{\otimes}M\overline{\otimes}M^{\text{op}}$, then  $\Delta(M)\nprec_{\mathcal M}1\overline{\otimes}M\overline{\otimes}M^{\text{op}}$, $\Delta(M)\nprec_{\mathcal M}M\overline{\otimes}1\overline{\otimes}M^{\text{op}}$ and $\Delta(M)\nprec_{\mathcal M}M\overline{\otimes}M\overline{\otimes}1$.

More generally, we  show the existence of a triple comultiplication with an analogous property whenever the cocycle $c$ takes values in a finite dimensional subalgebra of $B$.

\begin{lemma}\label{cocyclefindim}
Let $(\sigma,c)$ be a cocycle action of an infinite countable group $G$ on a tracial von Neumann algebra $(B,\tau)$. 
Assume that the $2$-cocycle $c$ is normalized and there exists a finite dimensional von Neumann subalgebra $B_0\subset B$ such that $c(G\times G)\subset\mathcal U(B_0)$.  Denote $M=B\rtimes_{\sigma,c}G$.

Then there exists an embedding $\Delta:M\rightarrow p\mathcal Mp$, where $\mathcal M=\mathbb M_n(\mathbb C)\overline{\otimes}M\overline{\otimes}M\overline{\otimes}M^{\emph{op}}$, for some $n\in\mathbb N$ and projection $p\in\mathcal M$ such that if $S\in\{\mathbb M_n(\mathbb C)\overline{\otimes}1\overline{\otimes}M\overline{\otimes}M^{\emph{op}},
\mathbb M_n(\mathbb C)\overline{\otimes}M\overline{\otimes}1\overline{\otimes}M^{\emph{op}}, \mathbb M_n(\mathbb C)\overline{\otimes}M\overline{\otimes}M\overline{\otimes}1\}$, then $\Delta(M)\nprec_{\mathcal M}S$.
\end{lemma}

\begin{proof}
Let $C\subset B$ be the von Neumann subalgebra generated by $c(G\times G)$. Then $C\subset B_0$, hence $C$ is finite dimensional. Moreover, since $\sigma_g(c_{h,k})=c_{g,h}c_{gh,k}c_{g,hk}^*$, for every $g,h,k\in G$, we get that $C$, and hence $\mathcal Z(C)$, is $\sigma(G)$-invariant.  Since $\mathcal Z(C)$ is a finite dimensional abelian von Neumann algebra, we can find a finite index subgroup $H<G$ such that $\sigma(g)_{|\mathcal Z(C)}=\text{Id}_{\mathcal Z(C)}$, for every $g\in H$.

If $g\in H$, then $\sigma_g=\text{Ad}(u_g)$ induces an automorphism of the finite dimensional algebra $C$ which is the identity on its center, $\mathcal Z(C)$. Hence, we can find $w_g\in C$ such that $\text{Ad}(u_g)_{|C}=\text{Ad}(w_g)$. Denoting $v_g=w_g^*u_g\in C'\cap M$, we have that $d_{g,h}=v_gv_hv_{gh}^*\in C'\cap M$, for every $g,h\in H$. Moreover, since $d_{g,h}=w_g^*\sigma_g(w_h^*)c_{g,h}w_{gh}\in C$ we deduce that $d_{g,h}\in\mathcal Z(C)$, for every $g,h\in H$.

Put $N=B\rtimes_{\sigma_{|H},c_{|(H\times H)}}H$. Then $N=B\rtimes_{\rho,d}H$, where $\rho_g=\text{Ad}(v_g)\in\text{Aut}(B)$, for every $g\in H$. Define $P=C'\cap N$ and $Q=C'\cap B$.
Let $x\in P$ and write $x=\sum_{g\in H}b_gv_g$, where $(b_g)_{g\in H}\subset B$. 
Since $(v_g)_{g\in H}\subset P$, we get that $(b_g)_{g\in H}\subset Q$. Hence, $P=(Q\cup\{v_g\}_{g\in H})''$ and $P=Q\rtimes_{\rho,d}H$.
Write $\mathcal Z(C)=\oplus_{i=1}^m\mathbb Cp_i$, for a partition of unity $\{p_i\}_{i=1}^m$ in $\mathcal Z(C)$. Since $\mathcal Z(C)\subset\mathcal Z(P)$, we have $P=\oplus_{i=1}^mP_i$, where $P_i=Pp_i$. Moreover, since $\mathcal Z(C)\subset\mathcal Z(Q)$, if $1\leq i\leq m$, then $P_i=Qp_i\rtimes_{\rho_i,d_i}H$, where $\rho_i(g)=\text{Ad}(v_gp_i)$ and ${(d_i)}_{g,h}=d_{g,h}p_i$, for every $g,h\in H$. Since $p_i\in\mathcal Z(C)$ is a minimal projection and $d_{g,h}\in\mathcal Z(C)$, we get that ${(d_i)}_{g,h}\in\mathbb Tp_i$, for every $g,h\in H$.

For every $1\leq i\leq m$, we can thus define the triple comultiplication $\Delta_i:P_i\rightarrow P_i\overline{\otimes}P_i\overline{\otimes}P_i^{\text{op}}$ given by $\Delta_i(bv_gp_i)=bv_gp_i\otimes v_gp_i\otimes (v_{g^{-1}}p_i)^{\text{op}}$, for every $b\in (C'\cap B)p_i$ and $g\in H$.
Let $p=\sum_{i=1}^mp_i\otimes p_i\otimes p_i^{\text{op}}$ and define the embedding
 $\Delta_0:P\rightarrow p(P\overline{\otimes}P\overline{\otimes}P^{\text{op}})p$ given by letting for every $b\in Q$ and $g\in H$, \begin{equation}\label{triple}\Delta_0(bv_g)=\sum_{i=1}^m\Delta_i(bv_gp_i)=\sum_{i=1}^m(bv_gp_i\otimes v_gp_i\otimes (v_{g^{-1}}p_i)^{\text{op}})=(bv_g\otimes v_g\otimes (v_{g^{-1}})^{\text{op}})p.\end{equation} 
  Since the inclusion $P\subset M$ has finite index, we can find an embedding $\pi:M\rightarrow q(\mathbb M_n(\mathbb C)\overline{\otimes}P)q$,  for some $n\in\mathbb N$ and a projection $q\in\mathbb M_n(\mathbb C)\overline{\otimes}P$, such that the inclusion $\pi(M)\subset q(\mathbb M_n(\mathbb C)\overline{\otimes}P)q$ has finite index. In particular, since $H$ is infinite, we have that $\pi(M)\nprec_{\mathbb M_n(\mathbb C)\overline{\otimes}P}\mathbb M_n(\mathbb C)\overline{\otimes}B$.

Let $\mathcal P=\mathbb M_n(\mathbb C)\overline{\otimes}P\overline{\otimes}P\overline{\otimes}P^{\text{op}}$, $\mathcal M=\mathbb M_n(\mathbb C)\overline{\otimes}M\overline{\otimes}M\overline{\otimes}M^{\text{op}}$ and define $$\zeta=\text{Id}_{\mathbb M_n(\mathbb C)}\otimes\Delta_0:\mathbb M_n(\mathbb C)\overline{\otimes}P\rightarrow (1\otimes p)\mathcal P(1\otimes p).$$ 
Put $\mathcal F=\{\mathbb M_n(\mathbb C)\overline{\otimes}1\overline{\otimes}M\overline{\otimes}M^{\text{op}},
\mathbb M_n(\mathbb C)\overline{\otimes}M\overline{\otimes}1\overline{\otimes}M^{\text{op}}, \mathbb M_n(\mathbb C)\overline{\otimes}M\overline{\otimes}M\overline{\otimes}1\}$.
Equation \eqref{triple} implies that if $(x_i)_{i\in I}\subset (\mathbb M_n(\mathbb C)\overline{\otimes}P)_1$ is a net satisfying $\|\text{E}_{\mathbb M_n(\mathbb C)\overline{\otimes}B}(ax_ib)\|_2\rightarrow 0$, for every $a,b\in \mathbb M_n(\mathbb C)\overline{\otimes}P$, then $\|\text{E}_S(c\zeta(x_i)d)\|_2\rightarrow 0$, for every $c,d\in\mathcal M$ and $S\in\mathcal F$.

Finally, let $\Delta=\zeta\circ\pi:M\rightarrow r\mathcal Pr$, where $r=\zeta(q)$. Since $\pi(M)\nprec_{\mathbb M_n(\mathbb C)\overline{\otimes}P}\mathbb M_n(\mathbb C)\overline{\otimes}B$, the last paragraph implies that $\Delta(M)\nprec_{\mathcal M}S$, for every $S\in\mathcal F$.
This finishes the proof.
 \end{proof}

\subsection{Factorial crossed products} In this subsection, we prove two results concerning factorial crossed product von Neumann algebras. First, we prove that if a II$_1$ factor $M$ only has trivial embeddings into its amplifications, then one can describe all of its decompositions of the form $M=B\rtimes_\sigma G$, with $G$ a finite group.

\begin{lemma}\label{finite_cross}
Let $M$ be a II$_1$ factor satisfying that if $\theta: M\rightarrow M^t$ is an embedding, for some $t>0$, then $t\in\mathbb N$ and there is a unitary $u\in M^t=\mathbb M_t(\mathbb C)\overline{\otimes}M$ such that $\theta(x)=u(1\otimes x)u^*$, for every $x\in M$. Let $M=B\rtimes_\sigma G$ be a crossed product decomposition of $M$, where $G$ is a finite group.

Then there exist a subgroup $H<G$, a normalized $2$-cocycle $c\in Z^2(H,\mathbb T)$ 
such that $\emph{L}_c(H)$ is a factor and the following holds. 
Let $\{u_h\}_{h\in H}\subset\emph{L}_c(H)$ be the canonical generating unitaries and consider the action $H\curvearrowright^\rho\emph{L}_c(H)$ given by $\rho_h=\emph{Ad}(u_h)$, for every $h\in H$. Let $C=\emph{L}_c(H)\overline{\otimes}M^{|G|^{-1}}$ and consider the action $H\curvearrowright^{\rho\otimes\emph{Id}}C$.

Then the action $G\curvearrowright^\sigma B$ is conjugate to the action $G\curvearrowright^\delta\ell^\infty(G/H)\overline{\otimes}C$
obtained by inducing $\rho\otimes\emph{Id}$. 

\end{lemma}

\begin{proof}
We  claim that if $N\subset M$ is a finite index subfactor, then $(N\subset M)\subset (N\overline{\otimes}1\subset N\overline{\otimes}\mathbb M_n(\mathbb C))$, for some $n\in\mathbb N$.
Indeed, the  downward basic construction (see \cite[Corollary 3.1.9]{Jon83}) provides a finite index subfactor $P\subset N$ such that $M\cong \langle N,e_P\rangle$. Then $P\cong M^s$, for some $s>0$, which induces an embedding $M^s\cong P\hookrightarrow M$. The assumption on $M$ implies that $S=P'\cap M$ is a finite dimensional factor and $M=P\overline{\otimes}S$. Hence,  $P\overline{\otimes}1\subset N\subset P\overline{\otimes}S$, which implies that 
$N=P\overline{\otimes}T$, where $T\subset S$ is a finite dimensional subfactor. Since $T\subset S$ are finite dimensional factors, we have  $S=T\overline{\otimes}U$, for a subfactor $U\subset S$.
Hence, we have $N=P\overline{\otimes}T\overline{\otimes}1$ and $M=P\overline{\otimes}T\overline{\otimes}U$. Therefore, $(N\subset M)\cong (N\overline{\otimes}1\subset N\overline{\otimes}U)$, which proves the claim since $U\cong\mathbb M_n(\mathbb C)$, for some $n\in\mathbb N$.

We first prove the conclusion under the assumption that $B$ is a factor.
In this case, the above claim implies that $B'\cap M$ is a finite dimensional factor and $M=B\overline{\otimes}(B'\cap M)$.
Let $H$ be the subgroup of $g\in G$ such that $\sigma_g\in\text{Aut}(B)$ is inner. Since $B$ is a factor, it is easy to see that $B'\cap M\subset B\rtimes_{\sigma|H}H$. Hence, $M=B\overline{\otimes}(B'\cap M)\subset B\rtimes_{\sigma|H}H$, which forces that $H=G$.

For $g\in G$, let $w_g\in\mathcal U(B)$ such that $\sigma_g=\text{Ad}(w_g)$. We take $w_e=1$. Then $\text{Ad}(w_g)=\sigma_g=\text{Ad}(u_g)$, therefore $z_g:=w_g^*u_g\in B'\cap M$ and so $u_g=w_g\otimes z_g$, for every $g\in G$.
In particular, if $g,h\in G$,  then since $u_gu_h=u_{gh}$, we can find $c_{g,h}\in\mathbb T$ such that $w_gw_h=c_{g,h}w_{gh}$ and $z_gz_h=\overline{c_{g,h}}z_{gh}$. Then $c\in Z^2(G,\mathbb T)$ is a normalized $2$-cocycle.

Next, we claim that $B'\cap M=\text{sp}(\{z_g\}_{g\in G})$. Indeed, let $x\in B'\cap M$ and write $x=\sum_{g\in G}x_gu_g$, where $(x_g)_{g\in G}\subset B$. Then for every $g\in G$, we have $x_g\sigma_g(b)=bx_g$, for every $b\in B$. Since $\sigma_g=\text{Ad}(w_g)$ and $B$ is a factor, we get that $x_g\in\mathbb Cw_g^*$ and thus $x_gu_g\in\mathbb Cw_g^*u_g=\mathbb Cz_g$, for every $g\in G$. This proves that $x\in \text{sp}(\{z_g\}_{g\in G})$, which gives  the claim. 

If $g\in G\setminus\{e\}$, then since $u_g=w_g\otimes z_g$, we get $0=\text{E}_B(u_g)=\tau(z_g)w_g$ and hence $\tau(z_g)=0$. 
This  proves that $\text{sp}(\{z_g\}_{g\in G})$ is canonically isomorphic to $\text{L}_{\overline{c}}(G)$ and so $\text{L}_{\overline{c}}(G)\cong B'\cap M$. In particular, $\text{L}_{\overline{c}}(G)$ is a factor and thus so is  $\text{L}_c(G)\cong \text{L}_{\overline{c}}(G)^{\text{op}}$. Since  $w_gw_h={c_{g,h}}w_{gh}$, for every $g,h\in G$, the map $G\ni g\mapsto w_g\in B$ extends to a $*$-homomorphism $\text{L}_c(G)\to B$. Since $\text{L}_c(G)$ is a factor, this $*$-homomorphism is injective, thus $\text{sp}(\{w_g\}_{g\in G})$ is a copy of $\text{L}_c(G)$.
Since $B$ is a II$_1$ factor and $\text{L}_c(G)$ is a finite dimensional factor, we can decompose $B=\text{L}_c(G)\overline{\otimes}D$, where $D$ is a II$_1$ factor.

Finally, consider the action $G\curvearrowright^\rho\text{L}_c(G)$ given by $\rho_g=\text{Ad}(w_g)$. Since $u_g=w_g\otimes z_g$ and $z_g\in B'\cap M$, we get that $\sigma_g=\text{Ad}(w_g)=\rho_g\otimes\text{Id}_D$, for every $g\in G$. 
Since $\text{L}_c(G)$ is a factor, $|G|=\text{dim}(\text{L}_c(G))$ is a perfect square.
Since  $M=B\overline{\otimes}(B'\cap M)$ and $B'\cap M\cong\text{L}_{\overline{c}}(G)$, we have that $B\cong M^{\sqrt{|G|^{-1}}}$.
Similarly, since $B=\text{L}_c(G)\overline{\otimes}D$, we get $D\cong B^{\sqrt{|G|^{-1}}}$. By combining these facts, we derive that $D\cong M^{|G|^{-1}}$.
This shows that the conclusion holds in the case $B$ is a factor with $H=G$.

Second, for general $B$, write $\mathcal Z(B)=\bigoplus_{i\in I}\mathbb Cp_i$ for a partition of unity $\{p_i\}_{i\in I}$ in $\mathcal Z(B)$.
Since $M=B\rtimes_\sigma G$ is a factor, $G$ acts ergodically on $\mathcal Z(B)$. Fixing $i_0\in I$, we have $\{p_i\}_{i\in I}=\{\sigma_g(p_{i_0})\}_{g\in G}$.
Let $H=\{g\in G\mid \sigma_g(p_{i_0})=p_{i_0}\}$, $B_0=Bp_{i_0}$ and consider the action $H\curvearrowright^\zeta B_0$ given by $\zeta_h=(\sigma_h)_{|B_0}$, for every $h\in H$.
Then we can identify $B=\ell^\infty(G/H)\overline{\otimes}B_0$ in a way that identifies $\sigma$ with the induced action $\text{Ind}_H^G(\zeta)$.
Let $M_0=B_0\rtimes_\zeta H$. 
Since $M_0\cong p_{i_0}Mp_{i_0}$, we get that $M_0$ is a factor which satisfies the same property as $M$. Since $B_0$ is a factor, the case treated above shows that we can write $B_0=\text{L}_c(H)\overline{\otimes}D$, where $c\in Z^2(H,\mathbb T)$ is a $2$-cocycle such that $\text{L}_c(H)$ is a factor, $D\cong M_0^{|H|^{-1}}$ and $\zeta_h=\rho_h\otimes\text{Id}_D$, where $\rho_h=\text{Ad}(u_h)$, for every $h\in H$. Since $\tau(p_{i_0})=[G:H]^{-1}$, we get that $M_0=p_{i_0}Mp_{i_0}\cong M^{[G:H]^{-1}}$ and thus $D\cong M_0^{|H|^{-1}}\cong M^{|G|^{-1}}$.
This finishes the proof.
\end{proof}

The following fact is likely known, but since we could not find a reference, we include a short proof for the reader's convenience. It shows that if a continuous crossed product $B\rtimes_\sigma G$ is a II$_1$ factor, then the locally compact acting group $G$ is necessarily discrete.

\begin{lemma}\label{non-discrete}
Let $G\curvearrowright^\sigma B$ be a continuous action of a locally compact group $G$ on a von Neumann algebra $B$ such that the crossed product  algebra $M=B\rtimes_\sigma G$ is a II$_1$ factor.
Then $G$ is discrete.
\end{lemma}

\begin{proof} Denote by $\tau$ the trace of $M$. Since $M$ is a II$_1$ factor, its von Neumann subalgebra $\text{L}(G)$ must be tracial. Hence, by a result of Dixmier \cite[Proposition 13.10.5]{Dix77}, $G$ is a SIN group, i.e., it admits a basis of neighborhoods of $e$ which are invariant under conjugation. In particular, $G$ is unimodular.

Let $m$ be a Haar measure of $G$. Let $V$ be a compact neighborhood of $e$ which is invariant under conjugation, $\xi_V=\frac{1_V}{\sqrt{m(V)}}\in\text{L}^2(G)$ and define $\varphi_V:\text{L}(G)\rightarrow\mathbb C$ by letting $\varphi_V(x)=\langle x\xi_V,\xi_V\rangle$. 

Then $\varphi_V$ is a tracial state on $\text{L}(G)$. 
Let $\Delta:M\rightarrow M\overline{\otimes}\text{L}(G)$ be the comultiplication $*$-homomorphism given by $\Delta(bu_g)=bu_g\otimes u_g$, for every $b\in B$ and $g\in G$.
Then
$(\tau\otimes\varphi_V)\circ\Delta$ is a trace on $M$. Since $M$ is II$_1$ factor, the uniqueness of its trace implies that $(\tau\otimes\varphi_V)(\Delta(x))=\tau(x)$, for every $x\in M$. Hence, if $g\in G$, then $\tau(u_g)\varphi_V(u_g)=(\tau\otimes\varphi_V)(\Delta(u_g))=\tau(u_g)$.

If $g\in G\setminus\{e\}$, then we can find a conjugation-invariant compact neighborhood $V$ of $e$ such that $g\not\in VV^{-1}$. This implies that $\varphi_V(u_g)=0$ and further that $\tau(u_g)=\tau(u_g)\varphi_V(u_g)=0$. In conclusion, we have that $\tau(u_g)=0$, for every $g\in G\setminus\{e\}$. Since $\tau(u_e)=1$ and the map $G\ni g\mapsto\tau(u_g)\in\mathbb C$ is continuous, we conclude that $\{e\}$ is an open subset of $G$, and hence $G$ is discrete. 
\end{proof}

\section{Local quantization for bimodules}\label{Sec3:localquant}

%The following local quantization principle was proved by Popa in \cite[Lemma A.1.1]{Pop94} using a technique from \cite{Pop81}. 

We begin this section by recalling Popa's local quantization principle, proved in \cite[Lemma A.1.1]{Pop94} using a technique introduced in \cite{Pop81}. 
Recently, Marrakchi established analogues of this result for type III von Neumann algebras, see \cite[Theorems 11.6 and 11.7]{Mar23}. 

\begin{theorem}[Popa]\label{Popa}
Let $(N,\tau)$ be a tracial von Neumann algebra and $M\subset N$ be a von Neumann subalgebra of type \emph{II}$_1$. Let $S\subset N$ be a finite set such that $\emph{E}_{M'\cap N}(x)=0$, for every $x\in S$.  Then for every $\varepsilon>0$, there exists a partition of unity $\{p_i\}_{i=1}^m$ in $M$ such that 
$$\text{$\|\sum_{i=1}^mp_i\xi p_i\|_2<\varepsilon$, for every $\xi\in S$.}$$
\end{theorem}

In this section, we generalize Popa’s theorem to the settings of $M$-$M$ and $M$-$M^{\text{op}}$ bimodules, and a simultaneous treatment of both. The section is divided into three parts, each dealing with one of these cases.

\subsection{Local quantization for \texorpdfstring{$M$}{M}-bimodules}
Our goal here is to provide the following extension of Popa's result to bimodules over type II$_1$ von Neumann algebras:
\begin{theorem}\label{local_quant}
Let $(M,\tau)$ be a tracial von Neumann algebra of type \emph{II}$_1$ and $\mathcal H$ be an $M$-bimodule. Let $S\subset\mathcal H\ominus\mathcal H^M$ be a finite set. Then for every $\varepsilon>0$, there exists a partition of unity $\{p_i\}_{i=1}^m$ in $M$ such that $$\text{$\|\sum_{i=1}^mp_i\xi p_i\|<\varepsilon$, for every $\xi\in S$}.$$
\end{theorem}

If $\mathcal H=\text{L}^2(N)$, where $N$ is  a tracial von Neumann algebra containing $M$, then $\mathcal H^M=\text{L}^2(M'\cap N)$. Thus, Theorem \ref{local_quant}  extends Theorem \ref{Popa}. Its proof follows the same strategy as Popa's proof of Theorem \ref{Popa} in \cite{Pop94}. More precisely, we first consider vectors in $\overline{\text{sp}}(M\mathcal H^M)$. When $\mathcal H=\text{L}^2(N)$, with $N\supset M$, this space is equal to the space $\overline{\text{sp}}(M(M'\cap N))$ considered in \cite[Lemma A.1.1]{Pop94}. We then consider vectors in  its orthogonal complement.
The first case reduces to the special case $\mathcal H=\text{L}^2(M)$.

Accordingly, the proof of Theorem \ref{local_quant} is obtained by combining two results. The first is the following reformulation of Theorem \ref{Popa} in the special case $N=M$.

\begin{lemma}\label{L^2(M)}
The conclusion of Theorem \ref{local_quant} holds if $\mathcal H=\emph{L}^2(M)$.
\end{lemma}
While Lemma \ref{L^2(M)} is a consequence of Theorem \ref{Popa}, for completeness, we provide a self-contained proof which bypasses the construction of hyperfinite subfactors used in the original proof from \cite{Pop94}.

The second ingredient needed to prove Theorem \ref{local_quant} is the following lemma:
\begin{lemma}\label{no_central}

The conclusion of Theorem \ref{local_quant} holds if $\mathcal H^M=\{0\}$.
\end{lemma}

Before proving Lemmas \ref{L^2(M)} and \ref{no_central}, let us explain how their combination implies Theorem \ref{local_quant}.
\begin{proof}[\bf Proof of Theorem \ref{local_quant}]
We first argue that we may take $M$ to be separable. Since $\xi\in\mathcal H\ominus\mathcal H^M$, the $\|\cdot\|_2$-closure of the convex hull of $\{u\xi u^*\mid u\in\mathcal U(M)\}$ contains $0$ by Lemma \ref{fixed_point}(1), for every $\xi\in S$. Since $S$ is finite,  we may find a separable von Neumann subalgebra $M_0\subset M$ such that $0$ belongs to the $\|\cdot\|_2$-closure of the convex hull of $\{u\xi u^*\mid u\in\mathcal U(M_0)\}$, for every $\xi\in S$. This readily implies that $S\subset\mathcal H\ominus\mathcal H^{M_0}$. 
Since $M$ is of type II$_1$, it contains a copy of the hyperfinite II$_1$ factor, $R$. Let $M_1=(M_0\cup R)''$.
Then $M_1$ is separable, of type II$_1$, and satisfies $S\subset\mathcal H\ominus\mathcal H^{M_0}\subset\mathcal H\ominus\mathcal H^{M_1}$. 
Hence, after replacing $M$ with $M_1$, we may assume that $M$ is separable.

Let $\mathcal H_1=\overline{\text{sp}}(M\mathcal H^M)$ and $\mathcal H_2=\mathcal H\ominus\mathcal H_1$. Let $e_i:\mathcal H\rightarrow\mathcal H_i$ be the orthogonal projection onto $\mathcal H_i$, for every $i\in\{1,2\}$.
Since $e_1(S)\subset\mathcal H_1\ominus\mathcal H_1^M$ and $\mathcal H_1$ is isomorphic to an $M$-subbimodule of $\text{L}^2(M)\otimes\ell^2(I)$, for some set $I$, Lemma \ref{L^2(M)} gives a partition of unity $\{q_j\}_{j=1}^n$ in $M$ such that $\|\sum_{j=1}^nq_je_1(\xi)q_j\|<\frac{\varepsilon}{2}$, for every $\xi\in S$.

Let $1\leq j\leq n$. Since $\mathcal H_2^{M}=\{0\}$, we have  $(q_j\mathcal H_2q_j)^{q_jMq_j}=\{0\}$ by Lemma \ref{fixed_point_proj}. Applying Lemma \ref{no_central} to $q_je_2(S)q_j\subset q_j\mathcal H_2q_j$ yields projections $(q_{j,k})_{k=1}^{K_j}\subset q_jMq_j$ such that $\sum_{k=1}^{K_j} q_{j,k}=q_j$ and $\|\sum_{k=1}^{K_j}q_{j,k}e_2(\xi) q_{j,k}\|<\frac{\varepsilon}{2n}$, for every $\xi\in S$. Thus, $\{p_i\}_{i=1}^m=\{q_{j,k}\mid 1\leq j\leq n,1\leq k\leq K_j\}$ is a partition of unity  in $M$ such that $\|\sum_{i=1}^mp_ie_2(\xi)p_i\|<\frac{\varepsilon}{2}$, for every $\xi\in S$. Since $\{p_i\}_{i=1}^m$ refines the partition $\{q_j\}_{j=1}^n$, we also have $\|\sum_{i=1}^mp_ie_1(\xi)p_i\|\leq\|\sum_{j=1}^nq_je_1(\xi)q_j\|<\frac{\varepsilon}{2}$, for every $\xi\in S$. This implies that $\|\sum_{i=1}^mp_i\xi p_i\|<\varepsilon$, for every $\xi\in S$, which finishes the proof.
\end{proof}

We now turn to the proof of Lemma \ref{L^2(M)} which is based on the following two lemmas.

\begin{lemma}\label{factor}
Let $M$ be a II$_1$ factor, $F\subset M$ finite and $z\in M$ with $\tau(z)=0$. Then there exists $u\in\mathcal U(M)$ such that $u^2=1$, $\tau(u)=0$, $|\tau(ux)|<\varepsilon$, for every $x\in F$, and $\|uzu^*-z\|_2\geq \frac{1}{3}\|z\|_2$.
\end{lemma}

\begin{proof} We may  assume that $z\not=0$. Since $\tau(z)=0$, Lemma \ref{fixed_point}(1) provides $v\in\mathcal U(M)$ such that 
$\Re\tau(vzv^*z^*)<\frac{1}{2}\|z\|_2^2$. Thus, $\|vzv^*-z\|_2^2=2(\|z\|_2^2-\Re\tau(vzv^*z^*))>\|z\|_2^2$, hence
$\|vzv^*-z\|_2>\|z\|_2$.

Next, since $M$ is diffuse,  there exists a diffuse separable abelian von Neumann subalgebra $A\subset M$ which contains $v$. 
Since $v$ commutes with $\text{E}_{A'\cap M}(z)$, we get that 
\begin{equation}
\label{z_inv}\|z\|_2<\|vzv^*-z\|_2=\|v(z-\text{E}_{A'\cap M}(z))v^*-(z-\text{E}_{A'\cap M}(z))\|_2\leq 2\|z-\text{E}_{A'\cap M}(z)\|_2.
\end{equation}
Since $A$ is abelian, diffuse and separable, we can identify it with $\text{L}(G)$, where $G=\bigoplus_{i=1}^\infty\frac{\mathbb Z}{2\mathbb Z}$. For $n\geq 1$, let $G_n=\bigoplus_{i=1}^n\frac{\mathbb Z}{2\mathbb Z}$ and $A_n=\text{L}(G_n)$. Then $A_n\subset A_{n+1}$, for every $n\geq 1$, and $(\cup_{n\geq 1}A_n)''=A$. Hence $A'_{n+1}\cap M\subset A_n'\cap M$, for every $n\geq 1$, and $A'\cap M=\cap_{n\geq 1}(A_n'\cap M)$. Thus, we have $\lim\limits_{n\rightarrow\infty}\|\text{E}_{A_n'\cap M}(y)-\text{E}_{A'\cap M}(y)\|_2=0$, for every $y\in M$. By using \eqref{z_inv} we can find $N\geq 1$ such that \begin{equation}\label{N}\text{$\|z\|_2<2\|\text{E}_{A_n'\cap M}(z)-z\|_2$, for every $n\geq N$.}\end{equation}
Since $\text{E}_{A_n'\cap M}(z)=\frac{1}{|G_n|}\sum_{g\in G_n}u_gzu_g^*$, using \eqref{N} we derive that
\begin{equation}\label{G_n}
\text{$\|z\|_2<2\|\emph{E}_{A_n'\cap M}(z)-z\|_2\leq\frac{1}{|G_n|}\sum_{g\in G_n}2\|u_gzu_g^*-z\|_2$, for every $n\geq N$}.
\end{equation}
Let $H_n=\{g\in G_n\mid \|u_gzu_g^*-z\|_2\geq\frac{1}{3}\|z\|_2\}$. Since $\|u_gzu_g^*-z\|_2\leq 2\|z\|_2$,  for every $g\in G$, \eqref{G_n} implies that for every $n\geq N$ we have $1< \frac{2|G_n\setminus H_n|}{3|G_n|}+\frac{4|H_n|}{|G_n|}$. This further implies that $|H_n|>\frac{|G_n|}{10}$.

On the other hand, since $u_g\rightarrow 0$ weakly, as $g\rightarrow\infty$, we can find a finite set $T\subset G$ such that $e\in T$ and $|\tau(u_gx)|<\varepsilon$, for every $x\in F$ and $g\in G\setminus T$. Let $n\geq N$ such that $\frac{|G_n|}{10}>|T|$. Then $|H_n|>|T|$, hence there exists $g\in H_n\setminus T$. Then $u=u_g\in\mathcal U(M)$ satisfies the conclusion.
\end{proof}

\begin{lemma}\label{abelian}
Let $(M,\tau)$ be a separable tracial von Neumann algebra of type \emph{II}$_1$, $S\subset M$ a finite set, and $\varepsilon>0$. Then there exists a maximal abelian subalgebra  $A\subset M$ such that $$\text{$\|\emph{E}_A(x)-\emph{E}_{\mathcal Z(M)}(x)\|_2<\varepsilon$, for every $x\in S$.}$$
\end{lemma}
Before proving Lemma \ref{abelian}, let us point out that it is a consequence of Theorem \ref{Popa} in the case $N=M$, or, equivalently, of Lemma \ref{L^2(M)}.
Indeed, by applying Theorem \ref{Popa} to $N=M$ and the set $\{x-\text{E}_{\mathcal Z(M)}(x)\mid x\in S\}$, we obtain a finite dimensional abelian von Neumann subalgebra $A_0\subset M$ with $\|\text{E}_{A_0'\cap M}(x)-\text{E}_{\mathcal Z(M)}(x)\|_2=\|\text{E}_{A_0'\cap M}(x-\text{E}_{\mathcal Z(M)}(x))\|_2<\varepsilon$, for every $x\in S$. If $A\subset M$ is a maximal abelian von Neumann subalgebra containing $A_0$, then $\mathcal Z(M)\subset A=A'\cap M\subset A_0'\cap M$. Hence, $\|\text{E}_A(x)-\text{E}_{\mathcal Z(M)}(x)\|_2\leq\|\text{E}_{A_0'\cap M}(x)-\text{E}_{\mathcal Z(M)}(x)\|_2<\varepsilon$, for every $x\in S$, proving Lemma \ref{abelian}. 
However, our goal is, conversely, to use Lemma \ref{abelian} in order to deduce Lemma \ref{L^2(M)}.

\begin{proof}[\bf Proof of Lemma \ref{abelian}]

Using standard disintegration arguments, we may assume that $M$ is a II$_1$ factor. 
Our goal  therefore becomes to find a maximal abelian subalgebra $A\subset M$ such that $\|\text{E}_A(x)-\tau(x)\|_2<\varepsilon$, for every $x\in S$.
Let $\{y_n\}_{n\geq 1}\subset M$ be a $\|\cdot\|_2$-dense sequence with $y_1=1$.

We will construct inductively an increasing sequence of finite dimensional abelian von Neumann subalgebras $\{A_n\}_{n\geq 1}$ of $M$ with $A_0=\mathbb C1$ such that for every $n\geq 1$ we  have 
\begin{equation}\label{1}\text{$\|\text{E}_{A_n}(x)-\tau(x)\|_2<\varepsilon$, for every $x\in S$,}\end{equation}
and there exists $u_n\in\mathcal U(A_n)$ such that denoting $z_n=\text{E}_{A_{n-1}'\cap M}(y_n)-\text{E}_{A_{n-1}}(y_n)$, we have 
\begin{equation}\label{2}
\|u_nz_nu_n^*-z_n\|_2\geq\frac{1}{3}{\|z_n\|_2}.
\end{equation}
For $n=1$, $A_1=\mathbb C1$ and $u_1=1$ satisfy the conclusion since $y_1=1$. Assuming that we have constructed $A_1,\dots,A_{n}$, for some $n\geq 1$, we will construct $A_{n+1}$ and $u_{n+1}\in\mathcal U(A_{n+1})$ satisfying \eqref{1} and \eqref{2} for $n+1$ instead of $n$. 
First, by \eqref{1}, we find $\delta>0$ such that 
\begin{equation}\label{e_delta}\text{$\|\text{E}_{A_{n}}(x)-\tau(x)\|_2+\delta<\varepsilon$, for every $x\in S$.}\end{equation}
Write $A_{n}=\bigoplus_{j=1}^{k}\mathbb Cq_{j}$, for a partition of unity $\{q_{j}\}_{j=1}^{k}$  in $M$. Then we have $A_{n}'\cap M=\bigoplus_{j=1}^{k}q_{j}Mq_{j}$. Put $\xi=z_{n+1}=\text{E}_{A_{n}'\cap M}(y_{n+1})-\text{E}_{A_{n}}(y_{n+1})$.
Since $\xi\in (A_{n}'\cap M)\ominus A_{n}$, we can write $\xi=\sum_{j=1}^{k}\xi_{j}$, where $\xi_{j}\in q_{j}Mq_{j}$ satisfies $\tau(\xi_{j})=0$, for every $1\leq j\leq k$.

By using that $q_jMq_j$ is a II$_1$ factor and Lemma \ref{factor} we find $v_{j}\in\mathcal U(q_jMq_j)$ such that 
\begin{equation}\label{3}\text{$v_{j}^2=q_j$,\;\; $\tau(v_{j})=0$,  \;\; $\|v_{j}\xi_jv_{j}^*-\xi_{j}\|_2\geq\frac{1}{3}\|\xi_{j}\|_2$,}
\end{equation}
and
\begin{equation}\label{4}
\text{ $|\tau(xv_{j})|=|\tau((q_jxq_j)v_{j})|<\delta\tau(q_j)$,\;\; for every $1\leq j\leq k$ and $x\in S$.}
\end{equation}
Let $u_{n+1}=\sum_{j=1}^{k}v_{j}\in\mathcal U(A_n'\cap M)$ and $A_{n+1}=(A_n\cup\{u_{n+1}\})''$.
Then $A_{n+1}$ is finite dimensional abelian and contains $A_n$.  
Since $u_{n+1}\xi u_{n+1}^*-\xi=\sum_{j=1}^{k}(v_{j}\xi_{j}v_{j}^*-\xi_{j})$ and  $v_{j}\xi_{j}v_{j}^*,\xi_{j}\in q_jMq_j$, for every $1\leq j\leq k$,  \eqref{2} is a consequence of the last part of \eqref{3}.

To prove \eqref{1}, we note first that $\text{E}_{A_n}(x)=\sum_{j=1}^k\frac{\tau(xq_j)}{\tau(q_j)}q_j$, for every $x\in M$. Since we have $\tau(u_{n+1}q_j)=\tau(v_j)=0$, for every $1\leq j\leq k$, we get that $\text{E}_{A_n}(u_{n+1})=0$. Since $u_{n+1}^2=1$, we also derive that $\text{E}_{A_{n+1}}(x)=\text{E}_{A_n}(x)+\text{E}_{A_n}(xu_{n+1})u_{n+1}$ and so 
\begin{equation}\label{5}\text{$\|\text{E}_{A_{n+1}}(x)-\text{E}_{A_n}(x)\|_2=\|\text{E}_{A_n}(xu_{n+1})\|_2$, for every $x\in M$.}\end{equation}
Now, if $x\in S$, then $\text{E}_{A_n}(xu_{n+1})=\sum_{j=1}^k\frac{\tau(xu_{n+1}q_j)}{\tau(q_j)}q_j=\sum_{j=1}^k\frac{\tau(xv_j)}{\tau(q_j)}q_j$. By using \eqref{4} we deduce that $\|\text{E}_{A_n}(xu_{n+1})\|_2^2=\sum_{j=1}^k\frac{|\tau(xv_j)|^2}{\tau(q_j)}<\sum_{j=1}^k\frac{\delta^2\tau(q_j)^2}{\tau(q_j)}=\delta^2$ and thus $\|\text{E}_{A_n}(xu_{n+1})\|_2<\delta$.
In combination with \eqref{5}, we deduce that $\|\text{E}_{A_{n+1}}(x)-\text{E}_{A_n}(x)\|_2<\delta$, for every $x\in S$. Together with \eqref{e_delta}, this implies \eqref{1}, which finishes the proof of the inductive step.

Finally, let $A=(\cup_{n=1}^\infty A_n)''$. By using \eqref{1} we get $\|\text{E}_A(x)-\tau(x)\|_2=\lim\limits_{n\rightarrow\infty}\|\text{E}_{A_n}(x)-\tau(x)\|_2\leq \varepsilon$, for every $x\in S$. To prove that $A$ is maximal abelian, let $y\in M$ and put $z=\text{E}_{A'\cap M}(y)-\text{E}_A(y)$. Let $(y_{n_k})$ be a subsequence of $(y_n)$ such that $\lim\limits_{k\rightarrow\infty}\|y_{n_k}-y\|_2=0$. Since $$\|z_{n_k}-(\text{E}_{A_{n_k-1}'\cap M}(y)-\text{E}_{A_{n_k-1}}(y))\|_2=\|(\text{E}_{A_{n_k-1}'\cap M}-\text{E}_{A_{n_k-1}})(y_{n_k}-y)\|_2\leq\|y_{n_k}-y\|_2,$$ and $\lim\limits_{n\rightarrow\infty}\|(E_{A_n'\cap M}(y)-\text{E}_{A_n}(y))-z\|_2=0$, we conclude that $\lim\limits_{k\rightarrow\infty}\|z_{n_k}-z\|_2=0$. Since $z\in A'\cap M$ and $u_{n_k}\in\mathcal U(A)$ we also have that $u_{n_k}zu_{n_k}^*=z$, for every $k\geq 1$. Combining the last two facts gives that $\lim\limits_{k\rightarrow\infty}\|u_{n_k}z_ku_{n_k}^*-z_k\|_2=0$.
In combination with \eqref{2}, we derive that $\|z\|_2=\lim\limits_{k\rightarrow\infty}\|z_k\|_2=0$. In other words, $\text{E}_{A'\cap M}(y)=\text{E}_A(y)$, for every $y\in M$. This implies that $A'\cap M=A$ and so $A$ is maximal abelian, which finishes the proof. 
\end{proof}

\begin{proof}[\bf Proof of Lemma \ref{L^2(M)}] 
As in the first paragraph of the proof of Theorem \ref{local_quant}, we may assume that $M$ is separable. Let $\mathcal H=\text{L}^2(M)$ and $S\subset\mathcal H\ominus\mathcal H^M=\text{L}^2(M)\ominus\text{L}^2(\mathcal Z(M))$ be a finite set.  By approximating in $\|\cdot\|_2$, we may assume that $S\subset M$ and $\text{E}_{\mathcal Z(M)}(\xi)=0$, for every $\xi\in S$. Lemma \ref{abelian} implies that we can find a maximal abelian subalgebra $A\subset M$ such that $\|\text{E}_A(\xi)\|_2<\varepsilon$, for every $\xi\in S$. Since $A$ is separable, $A=(\cup_{k\geq 1}A_k)''$, where $(A_k)_{k\geq 1}$ is a increasing sequence of finite dimensional abelian subalgebras of $M$. Since $A\subset M$ is maximal abelian, we have $\lim\limits_{k\rightarrow\infty}\|\text{E}_{A_k'\cap M}(x)-\text{E}_A(x)\|_2=0$, for every $x\in M$.  Therefore, we can find $k\geq 1$ such that $\|\text{E}_{A_k'\cap M}(\xi)\|_2<\varepsilon$, for every $\xi\in S$. If $A_k=\bigoplus_{i=1}^m\mathbb Cp_i$, for a partition of unity $\{p_i\}_{i=1}^m$  in $M$, then $\text{E}_{A_k'\cap M}(x)=\sum_{i=1}^mp_ixp_i$, for every $x\in M$, and the conclusion follows.
\end{proof}

In order to complete the proof of Theorem \ref{local_quant}, it remains to prove Lemma \ref{no_central}.

\begin{proof}[\bf Proof of Lemma \ref{no_central}] The proof is a straightforward adaptation of Popa's proof of \cite[Lemma 2.4]{Pop81}. Since the argument used here will also be needed to prove Theorem \ref{mop} below, we include a complete proof.

\begin{claim}\label{quant} For every $\xi\in\mathcal H$, there exists a partition of unity $\{e_i\}_{i=1}^n$ in $M$ such that we have $\|\sum_{i=1}^ne_i\xi e_i||^2\leq\frac{35}{36}\|\xi\|^2.$ \end{claim}

\begin{proof}[Proof of Claim \ref{quant}] We may clearly assume that $\xi\not=0$. 
Since $\mathcal H^M=\{0\}$, Lemma \ref{fixed_point}(2) provides $u\in\mathcal U(M)$ such that $\Re\langle u\xi u^*,\xi\rangle<\frac{17}{18} \|\xi\|^2$.
This implies that  $\|u\xi u^*-\xi\|>\frac{1}{3}\|\xi\|$.

Let $u\in\mathcal U(M)$ such that $\|u\xi u^*-\xi\|>\frac{1}{3}\|\xi\|$. By approximating $u$ in the operator norm by a unitary with finite spectrum, we may assume that $u=\sum_{j=1}^n\lambda_je_j$, where $\lambda_1,\dots,\lambda_n\in\mathbb T$ and $\{e_j\}_{j=1}^n$ is a partition of unity in $M$.  Since the vectors $\{e_i\xi e_j\}_{i,j=1}^n\subset\mathcal H$ are pairwise orthogonal  and $|\lambda_i\overline{\lambda}_j-1|\leq 2$, for every $1\leq i,j\leq n$, we get 
\begin{align*}4\|\xi\|^2-4\|\sum_{i=1}^ne_i\xi e_i\|^2&=4\|\sum_{i\not=j}e_i\xi e_j\|^2\\&\geq\|\sum_{i,j}^n(\lambda_i\overline{\lambda}_j-1)e_i\xi e_j\|^2=\|u\xi u^*-\xi\|^2>\frac{1}{9}\|\xi\|^2.
\end{align*}
Thus, we get that $\|\sum_{i=1}^ne_i\xi e_i\|^2\leq \frac{35}{36}\|\xi\|^2$, which proves the claim.
\end{proof}

Next, we enumerate $S=\{\xi_1,\xi_2,\dots,\xi_m\}$ and prove the following claim.

\begin{claim}\label{induction}
For every integers $n_1,n_2,\dots,n_m\geq 0$, there exists a partition of unity $\{f_k\}_{k=1}^r$ in $M$ such that $\|\sum_{k=1}^r f_k\xi_jf_k\|^2\leq \big(\frac{35}{36}\big)^{n_j}\|\xi_j\|^2$, for every $1\leq j\leq m$.
\end{claim}

\begin{proof} We prove the claim by induction.
Assume the claim holds for some $n_1,n_2,\dots,n_m\geq 0$ and a partition of unity $\{f_k\}_{k=1}^r$ in $M$. Fix $1\leq j_0\leq m$ and $1\leq k\leq r$. Since $\mathcal H^{M}=\{0\}$, Lemma \ref{fixed_point_proj} gives that $(f_k\mathcal Hf_k)^{f_kMf_k}=\{0\}$. By Claim \ref{quant} applied to $f_k\xi_{j_0}f_k$ we find projections $\{f_{k,l}\}_{l=1}^{l_k}$ such that $\|\sum_{l=1}^{l_k}f_{k,l}\xi_{j_0}f_{k,l}\|^2\leq\frac{35}{36}\|f_k\xi _{j_0}f_k\|^2$ and $\sum_{l=1}^{l_k}f_{k,l}=f_k$. 
Then $\{f_{k'}'\}_{k'=1}^{r'}:= \{f_{k,l}\mid 1\leq k\leq r, 1\leq l\leq l_k\}$ is a partition of unity in $M$
which satisfies $$\|\sum_{k=1}^{r'}f_k'\xi_{j_0}f_k'\|^2=\sum_{k=1}^r\|\sum_{l=1}^{l_k}f_{k,l}\xi_{j_0}f_{k,l}\|^2\leq\frac{35}{36}\sum_{k=1}^r\|f_k\xi_{j_0}f_k\|^2\leq\big(\frac{35}{36}\big)^{n_{j_0}+1}\|\xi_{j_0}\|^2.$$

Since the partition $\{f_{k'}'\}_{k'=1}^{r'}$ refines $\{f_k\}_{k=1}^r$, we also have  that

$$\text{$\|\sum_{k'=1}^{r'}f_{k'}'\xi_{j}f_{k'}'\|^2\leq \sum_{k=1}^r\|f_k\xi_jf_k\|^2\leq\big(\frac{35}{36}\big)^{n_j}\|\xi_j\|^2$,\;\;\; for every $1\leq j\leq m$.}$$

This shows that the claim holds when $(n_1,\dots,n_m)$ is replaced by $(n_1,\dots,n_{j_0-1},n_{j_0}+1,n_{j_0+1},\dots,n_m)$ and the partition $\{f_k\}_{k=1}^r$ is replaced by $\{f_{k'}'\}_{k'=1}^{r'}$. By induction, the claim follows.
\end{proof}

The conclusion follows by applying Claim \ref{induction} to $n_j\in\mathbb N$ with $\big(\frac{35}{36})^{n_j}<\varepsilon$, for every $1\leq j\leq m$.
\end{proof}

\subsection{Local quantization for \texorpdfstring{$M$}{M}-\texorpdfstring{$M^{\text{op}}$}{M{\text{op}}}-bimodules} Next, we establish a variant of Theorem \ref{local_quant} to $M$-$M^{\text{op}}$-bimodules.

\begin{theorem}\label{mop}
Let $(M,\tau)$ be a tracial von Neumann algebra of type $\emph{II}_1$ and $\mathcal H$ be an $M$-$M^{\emph{op}}$-bimodule. Let $S\subset\mathcal H$ be a finite set. Then for every $\varepsilon>0$, there exists a partition of unity $\{p_i\}_{i=1}^m$ in $M$ such that $$\text{$\|\sum_{i=1}^mp_i\xi p_i^{\emph{op}}\|<\varepsilon$, for every $\xi\in S$}.$$
\end{theorem}

\begin{lemma}\label{opposite}
Let $(M,\tau)$ be a tracial von Neumann algebra of type $\emph{II}_1$. Then there exist $u,v\in \mathcal U(M)$ such that for every $M$-$M^{\emph{op}}$-bimodule $\mathcal H$ we have that 
$$\text{$\|\xi\|\leq \|u\xi {(u^{\emph{op}})}^*-\xi\|+\|v\xi {(v^{\emph{op}})}^*-\xi\|$, for every  $\xi\in\mathcal H.$}$$
\end{lemma}

\begin{proof}
Since $M$ is of type \text{II}$_1$, we can find a unital embedding $\mathbb M_4(\mathbb C)\subset M$. Therefore, there exist $u,v\in\mathcal U(M)$ such that $vu=iuv$. Let $\xi\in\mathcal H$, and put  $\alpha=\|u\xi {(u^{\text{op}})}^*-\xi\|$ and $\beta=\|v\xi {(v^{\text{op}})}^*-\xi\|$. Then $\|vu\xi {(u^{\text{op}})}^*{(v^{\text{op}})}^*-\xi\|\leq\alpha+\beta$ and $\|uv\xi {(v^{\text{op}})}^*{(u^{\text{op}})}^*-\xi\|\leq\alpha+\beta$. Since $uv=-ivu$, we get  \begin{equation} {(u^{\text{op}})}^*{(v^{\text{op}})}^*=((uv)^{\text{op}})^*=((-ivu)^{\text{op}})^*=i((vu)^{\text{op}})^*=i(v^{\text{op}})^*(u^{\text{op}})^*.\end{equation}
Using also that $vu=iuv$, we get $vu\xi {(u^{\text{op}})}^*{(v^{\text{op}})}=i^2 uv\xi {(v^{\text{op}})}^*{(u^{\text{op}})}^*=-uv\xi {(v^{\text{op}})}^*{(u^{\text{op}})}^*$. By using this fact, the triangle inequality and the above inequalities, we conclude that $$2\|\xi\|\leq \|vu\xi {(u^{\text{op}})}^*{(v^{\text{op}})}^*-\xi\|+\|uv\xi {(v^{\text{op}})}^*{(u^{\text{op}})}^*-\xi\|\leq 2(\alpha+\beta).$$
Hence, $\|\xi\|\leq\alpha+\beta$, as desired.
\end{proof}

\begin{remark}
If $\mathcal H$ is an $M$-$M^{\text{op}}$-bimodule, then $x\xi y^{\text{op}}=\pi(x\otimes y)\xi$, where $\pi:M\overline{\otimes}M\rightarrow\mathbb B(\mathcal H)$ is a $*$-homomorphism such that $\pi_{|(M\otimes 1)}$ and $\pi_{|(1\otimes M)}$ are normal. The conclusion of Lemma \ref{opposite} thus rewrites as $\|\xi\|\leq \|\pi(u\otimes u^*)\xi-\xi\|+\|\pi(v\otimes v^*)\xi-\xi\|$, for every $\xi\in\mathcal H$. It is not hard to see that this implies that $\|\pi(1\otimes 1+u\otimes u^*+v\otimes v^*)\|\leq 2\sqrt{2}$ and hence $\|\pi(\frac{1}{3}(1\otimes 1+u\otimes u^*+v\otimes v^*))\|\leq\frac{2\sqrt{2}}{3}<1$. Thus, if $M\otimes_{\text{bin}}M$ denotes the binormal tensor product of $M$ with itself in the sense of \cite{EL77}, then 
$$\|\frac{1}{3}\big(1\otimes 1+u\otimes u^*+v\otimes v^*\big)\|_{M\overline{\otimes}_{\text{bin}}M}\leq\frac{2\sqrt{2}}{3}.$$
This fact strengthens \cite[Lemma 9.3]{IM19}.
\end{remark}

\begin{proof}[\bf Proof of Theorem \ref{mop}]
Theorem \ref{mop} follows from Lemma \ref{opposite} by adapting the proof of Lemma \ref{no_central}. For this reason, we only sketch the proof, leaving the details to the reader. 

We  first claim that if $\xi\in\mathcal H$, then there exists a partition of unity $\{e_i\}_{i=1}^n$  in $M$ such that  $\|\sum_{i=1}^ne_i\xi e_i^{\text{op}}||^2\leq\frac{35}{36}\|\xi\|^2$. 
To prove this, we may assume that $\xi\not=0$. Lemma \ref{opposite} implies the existence of $u\in\mathcal U(M)$ with finite spectrum such that $\|u\xi (u^{\text{op}})^*-\xi\|>\frac{1}{3}\|\xi\|$. 
Write $u=\sum_{j=1}^n\lambda_je_j$, where $\lambda_1,\dots,\lambda_n\in\mathbb T$ and $\{e_j\}_{j=1}^n$ is a partition of unity  in $M$.  
Then $u^{\text{op}}=\sum_{i=1}^n\lambda_ie_i^{\text{op}}$,  $(u^{\text{op}})^*=\sum_{i=1}^n\overline{\lambda}_ie_i^{\text{op}}$ and $u\xi (u^{\text{op}})^*=\sum_{i,j=1}^n\lambda_i\overline{\lambda}_je_i\xi e_j^{\text{op}}$.
Since the vectors $\{e_i\xi e_j^{\text{op}}\}_{i,j=1}^n\subset\mathcal H$ are pairwise orthogonal  and $|\lambda_i\overline{\lambda}_j-1|\leq 2$, for every $1\leq i,j\leq n$, we get 
\begin{align*}4\|\xi\|^2-4\|\sum_{i=1}^ne_i\xi e_i^{\text{op}}\|^2&=4\|\sum_{i\not=j}e_i\xi e_j^{\text{op}}\|^2\\&\geq\|\sum_{i,j}^n(\lambda_i\overline{\lambda}_j-1)e_i\xi e_j^{\text{op}}\|^2=\|u\xi u^*-\xi\|^2>\frac{1}{9}\|\xi\|^2.
\end{align*}

Therefore, $\|\sum_{i=1}^ne_i\xi e_i^{\text{op}}\|^2\leq \frac{35}{36}\|\xi\|^2$, which proves the claim.
The claim now implies the conclusion of Theorem \ref{mop}  via an obvious modification of Claim \ref{induction}.
\end{proof}

\subsection{Simultaneous local quantization for \texorpdfstring{$M$}{M}-\texorpdfstring{$M$}{M} and \texorpdfstring{$M$}{M}-\texorpdfstring{$M^{\text{op}}$}{M{\text{op}}} bimodules}
We end this section with the following result which puts together Theorems \ref{local_quant} and \ref{mop}. 

\begin{theorem}\label{2-for-1}
Let $(M,\tau)$ be a tracial von Neumann algebra of type $\emph{II}_1$. Let $\mathcal H$ be an $M$-bimodule and $\mathcal K$ be an $M$-$M^{\emph{op}}$-bimodule. Let $S\subset\mathcal H\ominus\mathcal H^M$ and $T\subset\mathcal K$ be finite sets. Then for every $\varepsilon>0$, there exists a partition of unity $\{p_i\}_{i=1}^m$ in  $M$ such that 
$$\text{$\|\sum_{i=1}^mp_i\xi p_i\|<\varepsilon$, for every $\xi\in S$,\;\;\; and \;\;\; $\|\sum_{i=1}^mp_i\eta p_i^{\emph{op}}\|<\varepsilon$, for every $\eta\in T$}.$$
Moreover, we can take $m$ to be any large enough integer, and $\tau(p_i)=\frac{1}{m}$, for every $1\leq i\leq m$.
\end{theorem}

\begin{proof}
Theorem \ref{local_quant} provides a partition of unity $\{q_j\}_{j=1}^n$ in $M$ such that $\|\sum_{j=1}^nq_j\xi q_j\|<\varepsilon$, for every $\xi\in S$. Let $1\leq j\leq n$. Since $q_jMq_j$ is of type II$_1$, by applying Theorem \ref{mop} to the $(q_jMq_j)-(q_j^{\text{op}}M^{\text{op}}q_j^{\text{op}})$-bimodule $q_j\mathcal Kq_j^{\text{op}}$ and $q_jTq_j^{\text{op}}\subset q_j\mathcal Kq_j^{\text{op}}$ we find a partition of unity $\{q_{j,k}\}_{k=1}^{K_j}$ in $q_jMq_j$ with $\|\sum_{k=1}^{K_j}q_{j,k}\eta q_{j,k}^{\text{op}}\|<\frac{\varepsilon}{n}$, for every $\eta\in T$. Let $\{r_l\}_{l=1}^L=\{q_{j,k}\mid 1\leq j\leq n,1\leq k\leq K_j\}$. Then $\{r_l\}_{l=1}^L$ is a partition of unity in $M$ such that $\|\sum_{l=1}^Lr_l\eta r_l^{\text{op}}\|<\varepsilon$, for every $\eta\in T$. Moreover, since $\{r_l\}_{l=1}^L$ refines $\{q_j\}_{j=1}^n$, we also have $\|\sum_{l=1}^Lr_l\xi r_l\|<\|\sum_{j=1}^nq_j\xi q_j\|<\varepsilon$, for every $\xi \in S$. This proves the main assertion. 

To justify the moreover assertion, note that for every $\delta>0$ and every large enough $m\in\mathbb N$ we can find a partition of unity $\{r_l'\}_{l=1}^L$ such that $\tau(r_l')\in\mathbb Z/m$, $\|(r_l'-r_l)\xi\|<\delta$, $\|\xi(r_l'-r_l)\|<\delta$, $\|(r_l'-r_l)\eta\|<\delta$ and $\|\eta(r_l'-r_l)^{\text{op}}\|<\delta$, for every $1\leq l\leq L$, $\xi\in S$ and $\eta\in T$. For $\delta>0$ small enough we have $\|\sum_{l=1}^Lr_l'\xi r_l'\|<\varepsilon$ and $\|\sum_{l=1}^Lr_l'\eta {r_l'}^{\text{op}}\|<\varepsilon$, for every $\xi\in S$ and $\eta\in T$. The conclusion holds for any partition of unity $\{p_i\}_{i=1}^m$ in $M$ which refines $\{r_l'\}_{l=1}^L$ and satisfies $\tau(p_i)=\frac{1}{m}$, for every $1\leq i\leq m$.
\end{proof}

\section{An integration trick}\label{Sec:integration} 

In this section, we establish the following integration trick of independent interest which will be needed in the proofs of our main results. This result provides a method of finding a unitary $u\in M$ such that the quantity $\|\Psi(\Phi_1(u)\Phi_2(u))\|_2$ is small, where $\Phi_1,\Phi_2,\Psi:M\rightarrow M$ are completely positive maps defined on a tracial von Neumann algebra $(M,\tau)$. More generally, we prove:

\begin{lemma}\label{randomization}
Let $(M,\tau)$, $(N,\tau)$, $(P,\tau)$ be tracial von Neumann algebras. Let $\{p_i\}_{i=1}^m$ be a partition of unity in $M$. Endow $\Omega=\{\pm 1\}^m$ with the uniform probability measure $\mu$ and $\mathcal U(M)$ with the probability measure $\nu=\pi_*\mu$, where $\pi:\Omega\rightarrow\mathcal U(M)$ is the map given by $\pi(\varepsilon)=\varepsilon_1p_1+\dots+\varepsilon_mp_m$, for every $\varepsilon=(\varepsilon_1,\dots,\varepsilon_m)\in\Omega$. Then the following hold:

\begin{enumerate}
\item Let $\Phi:M\rightarrow N$ be a linear map. Then $$\int_{\mathcal U(M)}\|\Phi(u)\|_2^2\;\emph{d}\nu(u)=\sum_{i=1}^m\|\Phi(p_i)\|_2^2.$$

\item Let $\Phi_1,\Phi_2:M\rightarrow N$ be positive linear maps. Let $\Psi:N\rightarrow P$ be a normal completely positive map such that $\Psi(1)\leq 1$ and $\tau\circ\Psi\leq\tau$. Denote by $\Psi^*:P\rightarrow N$ the adjoint of $\Psi$ and let $\xi:=\xi_{\Psi^*\circ\Psi}\in\mathcal H_{\Psi^*\circ\Psi}$, using the notation from Subsection \ref{bimm}.
Then $$\int_{\mathcal U(M)}\|\Psi(\Phi_1(u)\Phi_2(u))\|_2^2\;\emph{d}\nu(u)\leq \|\sum_{i=1}^m\Phi_1(p_i)\Phi_2(p_i)\|_2^2+2\|\Phi_1(1)\|^2\|\sum_{i=1}^m\Phi_2(p_i)\xi\Phi_2(p_i)\|.$$

\end{enumerate}
\end{lemma}

Lemma \ref{randomization} is especially useful when used in conjunction with the local quantization Theorem \ref{local_quant}, as we will see in the proof of Lemma \ref{one_unitary} in the next section.

\begin{proof}
(1) If $\varepsilon=(\varepsilon_1,\dots,\varepsilon_m)\in\Omega$ and $u=\varepsilon_1p_1+\cdots+\varepsilon_mp_m$, then 
$\|\Phi(u)\|_2^2=\sum_{i,j=1}^m\varepsilon_i\varepsilon_j\tau(\Phi(p_i)\Phi(p_j)^*)$. Since $\int_\Omega\varepsilon_i\varepsilon_j\;\text{d}\mu(\varepsilon)=\delta_{i,j}$, for every $1\leq i,j\leq m$, the conclusion follows.

(2) 
Put $I=\int_{\mathcal U(M)}\|\Psi(\Phi_1(u)\Phi_2(u))\|_2^2\;\text{d}\nu(u)$.
Since $\Phi_1,\Phi_2$ and $\Psi$ are self-adjoint, for every $u\in\mathcal U(M)$ we have 
 $\|\Psi(\Phi_1(u)\Phi_2(u))\|_2^2=\tau(\Psi(\Phi_1(u)\Phi_2(u))\Psi(\Phi_2(u^*)\Phi_1(u^*)))$.
This implies that
\begin{equation}\label{i}
I=\int_{\Omega}\sum_{i,j,k,l=1}^m\varepsilon_i\varepsilon_j\varepsilon_k\varepsilon_l\;\tau(\Psi(\Phi_1(p_i)\Phi_2(p_j))\Psi(\Phi_2(p_k)\Phi_1(p_l)))\;\text{d}\mu(\varepsilon_1,\dots,\varepsilon_m).
\end{equation}
If $i,j,k,l\in\{1,\dots,m\}$, then $$\displaystyle{\int_\Omega\varepsilon_i\varepsilon_j\varepsilon_k\varepsilon_l\;\text{d}\mu(\varepsilon)}=
\begin{cases} 1, \text{if 
$(i,k)=(j,l), (i,j)=(k,l)$ or $(i,j)=(l,k)$} \\ \text{$0$, otherwise.}
\end{cases}$$
Together with \eqref{i}, this implies that $I=I_1+I_2+I_3-2I_4$, where
\begin{equation}\label{I_1}I_1=\sum_{i,j=1}^m\tau(\Psi(\Phi_1(p_i)\Phi_2(p_i))\Psi(\Phi_2(p_j)\Phi_1(p_j)))
=\|\Psi(\sum_{i=1}^m\Phi_1(p_i)\Phi_2(p_i))\|_2^2,\end{equation}
\begin{equation}\label{I_2}I_2=\sum_{i,j=1}^m\tau(\Psi(\Phi_1(p_i)\Phi_2(p_j))\Psi(\Phi_2(p_i)\Phi_1(p_j))),\end{equation}
\begin{equation}\label{I_3}I_3=\sum_{i,j=1}^m\tau(\Psi(\Phi_1(p_i)\Phi_2(p_j))\Psi(\Phi_2(p_j)\Phi_1(p_i)))=\sum_{i,j=1}^m\|\Psi(\Phi_1(p_i)\Phi_2(p_j))\|_2^2,\end{equation} and
\begin{equation}\label{I4}I_4=\sum_{i=1}^m\tau(\Psi(\Phi_1(p_i)\Phi_2(p_i))\Psi(\Phi_2(p_i)\Phi_1(p_i)))=\sum_{i=1}^m\|\Psi(\Phi_1(p_i)\Phi_2(p_i))\|_2^2.\end{equation}

Thus, $I_1,I_3,I_4\geq 0$ and since $I\geq 0$, we get that $I_2\in\mathbb R$.
Since $|\tau(xy^*)|\leq \|x\|_2\|y\|_2\leq \frac{\|x\|_2^2+\|y\|_2^2}{2}$, for every $x,y\in P$,  by using \eqref{I_2} and \eqref{I_3} we get that
\begin{equation}\label{I2}
I_2\leq \sum_{i,j=1}^m\frac{\|\Psi(\Phi_1(p_i)\Phi_2(p_j))\|_2^2+\|\Psi(\Phi_1(p_j)\Phi_2(p_i))\|_2^2}{2}=I_3
\end{equation}
If $x\in N$, then $\|\Psi(x)\|_2^2=\tau(\Psi(x)\Psi(x^*))=\tau((\Psi^*\circ\Psi)(x)x^*)=\langle x\xi x^*,\xi\rangle$. Thus, by using  \eqref{I_3} and 
denoting $S_k=\sum_{i=1}^m\Phi_k(p_i)\xi\Phi_k(p_i)$, for $k\in\{1,2\}$, we get 
\begin{equation}\label{I_33}
I_3=\sum_{i,j=1}^m\langle\Phi_1(p_i)\Phi_2(p_j)\xi\Phi_2(p_j)\Phi_1(p_i),\xi\rangle=\langle S_2,S_1\rangle\leq \|S_2\|\cdot\|S_1\|.
\end{equation}
Next, the assumptions on $\Psi$ imply that it is $\|\cdot\|_2$-contractive. 
Thus, we derive that
\begin{equation}\label{SS1}\|S_1\|^2=
\sum_{i,j=1}^m\|\Psi(\Phi_1(p_i)\Phi_1(p_j))\|_2^2\leq\sum_{i,j=1}^m\|\Phi_1(p_i)\Phi_1(p_j)\|_2^2
=\|\sum_{i=1}^m\Phi_1(p_i)^2\|_2^2.
\end{equation}

Since $\Phi_1$ is positive, $0\leq\Phi_1(p)\leq\Phi_1(1)\leq \|\Phi_1(1)\|1$ and thus $\Phi_1(p)^2\leq \|\Phi_1(1)\|\Phi_1(p)$, for every projection $p\in M$. Hence, $\sum_{i=1}^m\Phi_1(p_i)^2\leq \|\Phi_1(1)\|\sum_{i=1}^m\Phi_1(p_i)=\|\Phi_1(1)\|\Phi_1(1)\leq \|\Phi_1(1)\|^21$.
By combining this fact with \eqref{SS1} we get that $\|S_1\|\leq \|\Phi_1(1)\|^2$.
Together with \eqref{I_33} we further get that $I_3\leq \|\Phi_1(1)\|^2\|\sum_{i=1}^m\Phi_2(p_i)\xi\Phi_2(p_i)\|$.
Since $\Psi$ is $\|\cdot\|_2$-contractive, \eqref{I_1} implies that $I_1\leq\|\sum_{i=1}^m\Phi_1(p_i)\Phi_2(p_i)\|_2^2$. Since also $I_4\geq 0$ by \eqref{I4} and $I_2\leq I_3$ by \eqref{I2}, we conclude that $$I=I_1+I_2+I_3-2I_4\leq I_1+2I_3\leq \|\sum_{i=1}^m\Phi_1(p_i)\Phi_2(p_i)\|_2^2+2\|\Phi_1(1)\|^2\|\sum_{i=1}^m\Phi_2(p_i)\xi\Phi_2(p_i)\|.$$
This finishes the proof.
\end{proof}

\section{Weakly mixing bimodules and order two unitaries}\label{Sec:weaklymixunitaries}

In this section we derive the following consequence of the local quantization Theorem \ref{2-for-1} and Lemma \ref{randomization}. The result constitutes the main technical ingredient in the inductive construction of hyperfinite II$_1$ factors with certain coarseness properties in Theorem \ref{bimodulenonintertwining}. It provides a method for constructing unitaries $u\in M$ whose matrix coefficients are simultaneously small on several prescribed bimodules. More precisely, conditions (1)--(3) below give the estimates needed to obtain the coarse property in part (a) of Theorem \ref{bimodulenonintertwining}, while conditions (4)--(6) provide the analogous estimate required for part (b). The weak mixing conclusion in part (c) will follow from a combination of diffuseness of the ambient factor and condition (3); see Remark \ref{unitaryleftaction}.

\begin{lemma}\label{one_unitary}
Let $(M,\tau)$ be a  tracial von Neumann algebra of type $\emph{II}_1$. Let $(N,\tau)$ and $(P,\tau)$ be tracial von Neumann algebras. 
Let $\mathcal H$ be an  $M$-$M\overline{\otimes}N$-bimodule  such that the $M$-$N$-bimodule $_M\mathcal H_{1\overline{\otimes} N}$ is left weakly mixing. Let  $\mathcal K$ be an $M$-$M^{\emph{op}}\overline{\otimes}P$-bimodule such that the $M$-$P$-bimodule $_M\mathcal K_{1\overline{\otimes} P}$ is left weakly mixing.
Let $\varepsilon>0$ and $X\subset \mathcal H\ominus\mathcal H^M$, $Y\subset \mathcal K$ be finite sets.

Then there exists $u\in\mathcal U(M)$ such that $u^2=1$, $\tau(u)=0$, and the following conditions hold:
 \begin{enumerate}
 \item $\sup\{|\langle u\xi_1 (u\otimes y),\xi_2\rangle|\mid \xi_1,\xi_2\in X,y\in (N)_1\}<\varepsilon,$
\item $\sup\{|\langle u\xi_1 (1\otimes y),\xi_2\rangle|\mid \xi_1,\xi_2\in X,y\in (N)_1\}<\varepsilon,$
\item $\sup\{|\langle \xi_1 (u\otimes y),\xi_2\rangle|\mid \xi_1,\xi_2\in X,y\in (N)_1\}<\varepsilon,$
\item $\sup\{|\langle u\eta_1(u^{\emph{op}}\otimes z),\eta_2\rangle|\mid\eta_1,\eta_2\in Y,z\in (P)_1\}<\varepsilon,$
\item $\sup\{|\langle u\eta_1(1\otimes z),\eta_2\rangle|\mid\eta_1,\eta_2\in Y,z\in (P)_1\}<\varepsilon,$
\item $\sup\{|\langle \eta_1(u^{\emph{op}}\otimes z),\eta_2\rangle|\mid\eta_1,\eta_2\in Y,z\in (P)_1\}<\varepsilon.$
\end{enumerate}

\end{lemma}

\begin{proof}
We first claim that we may assume that $|X|=|Y|=1$. 
Let $F=\{\pm 1,\pm i\}$,
$I=X\times X\times F$, $J=Y\times Y\times F$,
$\widetilde{\mathcal H}=\mathcal H^{\oplus_I}$ and $\widetilde{\mathcal K}=\mathcal K^{\oplus_J}$.
Let $N_0=N\overline{\otimes}\ell^\infty(I)$ and $P_0=P\overline{\otimes}\ell^\infty(J)$.
Then $\widetilde{\mathcal H}$ is an  $M$-$M\overline{\otimes}N_0$-bimodule  such that  $_M\widetilde{\mathcal H}_{1\otimes N_0}$ is left weakly mixing, and  $\widetilde{\mathcal K}$ is  an $M$-$M^{\text{op}}\overline{\otimes}P_0$-bimodule such that  $_M\widetilde{\mathcal K}_{1\otimes P_0}$ is left weakly mixing. 
Put $$\text{$\xi=\bigoplus_{(\xi_1,\xi_2,c)\in I}(\xi_1+c\xi_2)\in\widetilde{\mathcal H}$ \;\;\;\;\; and\;\;\;\;\; $\eta=\bigoplus_{(\eta_1,\eta_2,c)\in J}(\eta_1+c\eta_2)\in\widetilde{\mathcal K}$.}$$
By the polarization identity we have that $\langle T(\zeta_1),\zeta_2\rangle=\frac{1}{4}\sum_{c\in F}c\langle T(\zeta_1+c\zeta_2),\zeta_1+c\zeta_2\rangle$ and therefore $|\langle T(\zeta_1),\zeta_2\rangle|\leq\max_{c\in F}|\langle T(\zeta_1+c\zeta_2),\zeta_1+c\zeta_2\rangle|$, for any  $\zeta_1,\zeta_2$ in a Hilbert space $\mathcal L$ and $T\in\mathbb B(\mathcal L)$. 
For $i\in I$ and $j\in J$, let $e_i\in\ell^\infty(I)$ and $f_j\in\ell^\infty(J)$  be the corresponding coordinate projections.
Then $|\langle a\xi_1b,\xi_2\rangle|\leq \max_{i \in I} |\langle a\xi (b \otimes e_i),\xi\rangle|$ 
and $|\langle a\eta_1c,\eta_2\rangle|\leq \max_{j \in J} |\langle a\eta (c \otimes f_j),\eta\rangle|$, 
for every
$\xi_1,\xi_2\in X$, $\eta_1,\eta_2\in Y$
 $a\in M,b\in M\overline{\otimes}N$ and $c\in M^{\text{op}}\overline{\otimes}P$.
Since $X\subset \mathcal H\ominus\mathcal H^M$ and $\widetilde{\mathcal H}^M={(\mathcal H^M)}^{\oplus_I}$, we get that $\xi\in\widetilde{\mathcal H}\ominus\widetilde{\mathcal H}^M$.
Thus, if the lemma is proved in the case $|X|=|Y|=1$ for arbitrary tracial von Neumann algebras $N$ and $P$, then applying that case to $N_0$ and $P_0$ gives the general case.
This altogether justifies our claim.

By the claim we can assume that $X=\{\xi\}$ and $Y=\{\eta\}$, for some $\xi\in\mathcal H\ominus\mathcal H^M$ and $\eta\in\mathcal K$. By using the density of bounded vectors and Lemma \ref{fixed_point}(3), we may take $\xi$ and $\eta$ to be bounded. Further,
by rescaling $\xi$ and $\eta$, we may assume that $\|a\xi\|\leq \|a\|_2$, $\|\xi b\|\leq \|b\|_2$,  $\|a\eta\|\leq \|a\|_2$ and $\|\eta c\|\leq \|c\|_2$,  for every $a\in M$, $b\in M\overline{\otimes}N$ and $c\in M^{\text{op}}\overline{\otimes}P$. 
Let $$\text{$\Phi:M\rightarrow M\overline{\otimes}N$ \;\;\;\;\; and \;\;\;\;\;  $\Psi:M\rightarrow M^{\text{op}}\overline{\otimes}P$}$$ be the subunital, $\|\cdot\|$-contractive and $\|\cdot\|_2$-contractive, completely positive maps defined by the identities $\tau(\Phi(a)b)=\langle a\xi b,\xi\rangle$ and $\tau(\Psi(a)c)=\langle a\eta c,\eta\rangle$, for all $a\in M,b\in M\overline{\otimes}N$ and $c\in M^{\text{op}}\overline{\otimes}P$.
 
 For simplicity of notation, we write $x$ instead of $x\otimes 1$, if $x$ belongs to $M$ or $M^{\text{op}}$, and $y$ instead of $1\otimes y$, if $y$ belongs to $N$ or $P$.
We will also write $M,M^{\text{op}},N,P$ instead of $M\overline{\otimes} 1, M^{\text{op}}\overline{\otimes} 1,1\overline{\otimes} N,1\overline{\otimes} P$.

Let $\varepsilon\in (0,1)$. Then conditions (1)--(6) for $u\in\mathcal U(M)$ rewrite as:
  \begin{enumerate}
 \item[(i)] $\|\text{E}_{N}(\Phi(u)u)\|_1<\varepsilon,$
\item [(ii)] $\|\text{E}_N(\Phi(u))\|_1<\varepsilon,$
\item [(iii)] $\|\text{E}_N(\Phi(1)u)\|_1<\varepsilon,$
\item [(iv)] $\|\text{E}_P(\Psi(u)u^{\text{op}})\|_1<\varepsilon,$
\item [(v)] $\|\text{E}_P(\Psi(u))\|_1<\varepsilon,$
\item [(vi)]  $\|\text{E}_P(\Psi(1)u^{\text{op}})\|_1<\varepsilon$
.\end{enumerate}
Indeed, if $u\in\mathcal U(M)$ and $y\in N$, then $\langle u\xi(u\otimes y),\xi\rangle=\tau(\Phi(u)(u\otimes y))=\tau(\text{E}_N(\Phi(u)u)y)$. Hence,  $\sup\{|\langle u\xi (u\otimes y),\xi\rangle| \mid y\in (N)_1\}=\|\text{E}_N(\Phi(u)u)\|_1$, which justifies the equivalence of (1) and (i). The other equivalences follow similarly.

 Note that there exists $S\subset M$ finite such that  $\|\text{E}_N(\Phi(1)u)\|_1+\|\text{E}_P(\Psi(1)u^{\text{op}})\|_1<\frac{\varepsilon}{2}+\max_{\zeta\in S}|\tau(u\zeta)|$, for every $u\in\mathcal U(M)$. After possibly enlarging $S$, we may clearly assume that $1\in S$.
  Using these facts,  that $\Phi$ and $\Psi$ are $\|\cdot\|$-contractive,  and that $\|x\|_1\leq \|x\|_2$, for every $x\in M$, it follows that conditions (i)--(vi) for $u\in\mathcal U(M)$ are implied by the following conditions:
 
   \begin{enumerate}
 \item[(a)] $\alpha(u):=\|\text{E}_{N}(\Phi(u)u)\|_2^2<\varepsilon^2,$
\item [(b)] $\beta(u):=\|\text{E}_N(\Phi(u))\|_2^2<\varepsilon^2,$
\item [(c)] $\gamma(u):=\|\text{E}_P(\Psi(u)u^{\text{op}})\|_2^2<\varepsilon^2,$
\item [(d)] $\delta(u):=\|\text{E}_P(\Psi(u))\|_2^2<\varepsilon^2,$
\item [(e)]  $\sigma(u):=\sum_{\zeta\in S}|\tau(u\zeta)|^2 <\frac{\varepsilon^2}{4}$.
\end{enumerate}
Our goal now becomes to find $u\in\mathcal U(M)$ such that $u^2=1$, $\tau(u)=0$ and (a)--(e) are satisfied.
To this end, we define $\varepsilon(u)=\alpha(u)+\beta(u)+\gamma(u)+\delta(u)+\sigma(u)$, for every $u\in\mathcal U(M)$.

Let $C=2\sum_{\zeta\in S}\|\zeta\|_2^2>0$. If $u,v\in\mathcal U(M)$, then   $|\alpha(u)-\alpha(v)|\leq 4\|u-v\|_2$, $|\beta(u)-\beta(v)|\leq 2\|u-v\|_2$, $|\gamma(u)-\gamma(v)|\leq 4\|u-v\|_2$, $|\delta(u)-\delta(v)|\leq 2\|u-v\|_2$ and $|\sigma(u)-\sigma(v)|\leq C\|u-v\|_2$. Hence, we have
\begin{equation}\label{varr}\text{$|\varepsilon(u)-\varepsilon(v)|\leq (12+C)\|u-v\|_2$, for every $u,v\in\mathcal U(M)$.}
\end{equation}

 In order to prove the conclusion, it is therefore enough to prove the following claim:
 
\begin{claim}\label{special_u}
There exists $u\in\mathcal U(M)$ such that 
$u^2=1$, $\tau(u)=0$ and $\varepsilon(u)<\frac{\varepsilon^2}{4}$.
\end{claim}

{\bf Proof of Claim \ref{special_u}.}
We define the completely positive maps $\Phi'=\Phi^*\circ \text{E}_N\circ \Phi:M\rightarrow M$ and $\Psi'=\Psi^*\circ\text{E}_P\circ\Psi:M\rightarrow M$.
Let $\widetilde\zeta=\sum_{\zeta\in S}\zeta\otimes\zeta^*\in\text{L}^2(M)\otimes\text{L}^2(M)$.

Recall that $\xi\in\mathcal H\ominus\mathcal H^M$. Since $M$ is diffuse, we have $(\text{L}^2(M)\otimes \text{L}^2(M))^{M}=\{0\}$.
Since $M\nprec_{M\overline{\otimes}N}N$ and $M\nprec_{M\overline{\otimes}P^{\text{op}}}P^{\text{op}}$, we derive that  $\text{L}^2(\langle M\overline{\otimes}N,e_N\rangle)^M=\{0\}$ and $\text{L}^2(\langle M\overline{\otimes}P^{\text{op}},e_{P^{\text{op}}}\rangle)^M=\{0\}$.
Finally, we claim that $\mathcal H_{\Phi'}^{M}=\{0\}$ and $\mathcal H_{\Psi'}^{M}=\{0\}$. Since $\Phi'=(\text{E}_N\circ\Phi)^*\circ (\text{E}_N\circ\Phi)$, we get that 
$\mathcal H_{\Phi'}\cong\mathcal H_{\text{E}_N\circ\Phi}\otimes_N\overline{{\mathcal H}_{\text{E}_N\circ\Phi}}$, as $M$-bimodules.  Since the $M$-$N$-bimodule $\mathcal H_{\text{E}_N\circ\Phi}$ is contained in $_M\mathcal H_{1\overline{\otimes}N}$, it is left weakly mixing and therefore $\mathcal H_{\Phi'}^M=\{0\}$. Similarly, it follows that $\mathcal H_{\Psi'}^{M}=\{0\}$.

Consider the $M$-bimodule $\mathcal H\oplus \text{L}^2(\langle M\overline{\otimes}N,e_{N}\rangle)\oplus \mathcal H_{\Phi'}\oplus
\text{L}^2(\langle M\overline{\otimes}P^{\text{op}},e_{P^{\text{op}}}\rangle)\oplus \mathcal H_{\Psi'}\oplus (\text{L}^2(M)\otimes\text{L}^2(M))$ and the $M$-$M^{\text{op}}$-bimodule $\mathcal K$.
By using the previous paragraph and applying Theorem \ref{2-for-1} to these bimodules, we deduce the existence of a partition of unity $\{p_i\}_{i=1}^m$ in $M$, such that $m\in\mathbb N$ is even, 
$\tau(p_i)=\frac{1}{m}$, for every $1\leq i\leq m$, and the following inequalities hold: 
\begin{enumerate}
\item[(x)] $\|\sum_{i=1}^mp_i\xi p_i\|<\frac{\varepsilon}{10}$,
\item[(y)] $\|\sum_{i=1}^mp_ie_Np_i\|_{2,\text{Tr}}<\frac{\varepsilon^2}{100}$,
\item[(z)] $\|\sum_{i=1}^mp_i\xi_{\Phi'}p_i\|<\frac{\varepsilon}{10}$,
\item[(t)] $\|\sum_{i=1}^mp_i\eta p_i^{\text{op}}\|<\frac{\varepsilon}{10}$,
\item[(u)] $\|\sum_{i=1}^mp_ie_{P^{\text{op}}}p_i\|_{2,\text{Tr}}<\frac{\varepsilon^2}{100}$,
\item[(v)] $\|\sum_{i=1}^mp_i\xi_{\Psi'}p_i\|<\frac{\varepsilon}{10}$,
\item[(w)] $\|\sum_{i=1}^mp_i\widetilde\zeta p_i\|<\frac{\varepsilon^8}{10^7(12+C)^4}$.
\end{enumerate}

Let $\Omega=\{-1,1\}^m$ with the uniform probability measure $\nu$. 
If $1\leq i\leq m$, then $0\leq\Phi(p_i)\leq \Phi(1)\leq 1$ and thus $\Phi(p_i)^2\leq \Phi(p_i)$ and so $\tau(\Phi(p_i)^2p_i)\leq \tau(\Phi(p_i)p_i)$. Hence by using (x) we get that \begin{equation}\label{phi}
\|\sum_{i=1}^m\Phi(p_i)p_i\|_2^2=\tau(\sum_{i=1}^m\Phi(p_i)^2p_i)\leq\tau(\sum_{i=1}^m\Phi(p_i)p_i)=\sum_{i=1}^m\langle p_i\xi p_i,\xi\rangle=\|\sum_{i=1}^mp_i\xi p_i\|^2<\frac{\varepsilon^2}{100}.
\end{equation}
Similarly, by using (t) we get that 
\begin{equation}\label{psi}
\|\sum_{i=1}^m\Psi(p_i)p_i^{\text{op}}\|_2^2=\sum_{i=1}^m\tau(\Psi(p_i)^2p_i^{\text{op}})\leq\sum_{i=1}^m\tau(\Psi(p_i)p_i^{\text{op}})=\|\sum_{i=1}^mp_i\eta p_i^{\text{op}}\|^2<\frac{\varepsilon^2}{100}.
\end{equation}

Considering the completely positive map $\text{E}_N:M\overline{\otimes}N\rightarrow N$ and the notation from Subsection \ref{bimm}, $\text{E}_N^*\circ \text{E}_N:M\overline{\otimes}N\rightarrow M\overline{\otimes}N$ is the conditional expectation onto $N$, hence $\mathcal H_{\text{E}_N^*\circ \text{E}_N}=\text{L}^2(\langle M\overline{\otimes}N,e_N\rangle)$ and $\xi_{\text{E}_N^*\circ \text{E}_N}=e_N$.
By combining \eqref{phi}, (y) and Lemma \ref{randomization}(2), we get that

\begin{equation}\label{alpha}
\int_{\mathcal U(M)}\alpha(u)\;\text{d}\nu(u)\leq \|\sum_{i=1}^m\Phi(p_i)p_i\|_2^2+2\|\sum_{i=1}^mp_ie_Np_i\|_{2,\text{Tr}}<\frac{3\varepsilon^2}{100}.\end{equation}

Since $\|\text{E}_N(\Phi(x))\|_2^2=\tau(\Phi'(x)x^*)$, for every $x\in M$, using Lemma \ref{randomization}(1) and (z) we get

\begin{equation}\label{beta}
\int_{\mathcal U(M)}\beta(u)\;\text{d}\nu(u)=\sum_{i=1}^m\|\text{E}_N(\Phi(p_i))\|_2^2=\sum_{i=1}^m\tau(\Phi'(p_i)p_i)=\|\sum_{i=1}^mp_i\xi_{\Phi'} p_i\|_2^2<\frac{\varepsilon^2}{100}.
\end{equation}

Next, by (u) we have $\|\sum_{i=1}^mp_i^{\text{op}}e_Pp_i^{\text{op}}\|_{2,\text{Tr}}=\|\sum_{i=1}^mp_ie_{P^{\text{op}}}p_i\|_{2,\text{Tr}}<\frac{\varepsilon^2}{100}$. By using \eqref{psi}, this fact, that the map $x\mapsto x^{\text{op}}$ is linear and positive, and applying Lemma \ref{randomization}(2), we get that

\begin{equation}\label{gamma}
\int_{\mathcal U(M)}\gamma(u)\;\text{d}\nu(u)\leq \|\sum_{i=1}^m\Psi(p_i)p_i^{\text{op}}\|_2^2+2\|\sum_{i=1}^mp_i^{\text{op}}e_Pp_i^{\text{op}}\|_{2,\text{Tr}}<\frac{3\varepsilon^2}{100}.
\end{equation}

Since $\|\text{E}_P(\Psi(x))\|_2^2=\tau(\Psi'(x)x^*)$, for every $x\in M$, using Lemma \ref{randomization}(1) and (v) we get that
\begin{equation}\label{delta'}
\int_{\mathcal U(M)}\delta(u)\;\text{d}\nu(u)=\sum_{i=1}^m\|\text{E}_P(\Psi(p_i))\|_2^2=\sum_{i=1}^m\tau(\Psi'(p_i)p_i)=\|\sum_{i=1}^mp_i\xi_{\Psi'} p_i\|^2<\frac{\varepsilon^2}{100}.
\end{equation}

Finally, since $\sigma(u)=\langle u\widetilde\zeta u^*,1\otimes 1\rangle$, for every $u\in\mathcal U(M)$, by using (w) we get that
\begin{equation}\label{sigma}
\int_{\mathcal U(M)}\sigma(u)\;\text{d}\nu(u)=\int_{\mathcal U(M)}\langle u\widetilde\zeta u^*,1\otimes 1\rangle\;\text{d}\nu(u)=\langle\sum_{i=1}^mp_i\widetilde\zeta u_i^*,1\otimes 1\rangle<\frac{\varepsilon^8}{10^7(12+C)^4}.
\end{equation}

For $u\in\mathcal U(M)$, let $$\varepsilon_0(u)=\alpha(u)+\beta(u)+\gamma(u)+\delta(u)+\frac{10^5(12+C)^4}{\varepsilon^6}\sigma(u).$$ By combining \eqref{alpha}-\eqref{sigma} we get that 

\begin{equation}\label{varepsilon_0}
\int_{\mathcal U(M)}\varepsilon_0(u)\;\text{d}\nu(u)<\frac{\varepsilon^2}{100}+\frac{3\varepsilon^2}{100}+\frac{\varepsilon^2}{100}+\frac{3\varepsilon^2}{100}+\frac{10^5(12+C)^4}{\varepsilon^6}\cdot\frac{\varepsilon^8}{10^7(12+C)^4}<\frac{\varepsilon^2}{10}.
\end{equation}
Thus, we can find $\lambda_1,\dots,\lambda_m\in\{-1,1\}$ such that $v=\sum_{i=1}^m\lambda_ip_i\in\mathcal U(M)$ satisfies
$\varepsilon_0(v)<\frac{\varepsilon^2}{10}$. Let $A=\{1\leq i\leq m\mid \lambda_i=1\}$ and $B=\{1\leq i\leq m\mid \lambda_i=-1\}$ so that $v=\sum_{i\in A}p_i-\sum_{i\in B}p_i$. After replacing $v$ with $-v$, we may assume that $|A|\geq |B|$. Since $m=|A|+|B|$ is even,  $|A|-|B|$ is also even. Let $F\subset A$ be a subset with $|F|=\frac{|A|-|B|}{2}$. Define $u=\sum_{i\in A\setminus F}p_i-\sum_{i\in B\cup F}p_i$. Then $u\in\mathcal U(M)$, $\tau(u)=0$ and $u^2=1$. 
Moreover, $v-u=2\sum_{i\in F}p_i$ and thus 
\begin{equation}\label{distance}\|v-u\|_2=2\sqrt{\frac{|F|}{m}}=2\sqrt{\frac{|A|-|B|}{2m}}=\sqrt{2\tau(v)}.\end{equation}
Since $\varepsilon_0(v)<\frac{\varepsilon^2}{10}$, we get that $\frac{10^5(12+C)^4}{\varepsilon^6}\sigma(v)<\frac{\varepsilon^2}{10}$ and hence $\sigma(v)<\frac{\varepsilon^8}{10^6(12+C)^4}$. Since $1\in S$, we have $|\tau(v)|\leq\sqrt{\sigma(v)}$, which further entails that $\tau(v)<\frac{\varepsilon^4}{10^3(12+C)^2}$. In combination with \eqref{distance} we conclude that $\|u-v\|_2=\sqrt{2\tau(v)}<\frac{\varepsilon^2}{15(12+C)}$. Hence, using \eqref{varr} we get $$|\varepsilon(u)-\varepsilon(v)|\leq (12+C)\|u-v\|_2<\frac{\varepsilon^2}{15}.$$ Since $\varepsilon(v)\leq\varepsilon_0(v)<\frac{\varepsilon^2}{10}$, we get $\varepsilon(u)<\frac{\varepsilon^2}{15}+\frac{\varepsilon^2}{10}<\frac{\varepsilon^2}{4}$. This finishes the proof of both the claim and the theorem.
\end{proof}

\begin{remark}\label{unitaryleftaction} In order to obtain item (3) in the previous lemma, we used neither the fact that $X$ is orthogonal to $\mathcal H^M$ 
nor the fact that the $M$-$N$-bimodule $_M\mathcal H_{1\overline{\otimes}N}$ is left weakly mixing. 
The only property of $M$ that we used was diffuseness (see the items (iii)--(e)--(w)--\eqref{sigma} in the proof of Lemma \ref{one_unitary}). Hence, for any $M$-$M\overline{\otimes}N$-bimodule $\mathcal{H}$ 
we have that for every $\epsilon>0$ and every finite set $X\subset\mathcal{H}$, there is $u\in \mathcal U(M)$ such that $u^2=1$, $\tau(u)=0$ and $\sup\{|\langle\xi_1(u\otimes y),\xi_2\rangle|\mid\xi_1,\xi_2\in X, y\in (N)_1\}\leq\epsilon$.

\end{remark}

\section{Coarse hyperfinite subfactors}\label{Sec:hyperfinite}

In this section we combine Lemma \ref{one_unitary} with ideas from Popa's iterative method from \cite[Theorem 4.2]{Pop18b} to build an increasing sequence of copies of $2\times 2$ matrix algebras inside $M$. This yields a hyperfinite subfactor $R\subset M$ that is coarse relative to the given bimodules and preserves the required weak mixing properties. The main goal of this section is to prove the following result, which is a more general version of Theorem \ref{popa_extension}.

\begin{theorem}\label{bimodulenonintertwining} Let $M$ be a separable II$_1$ factor and $(N,\tau)$, $(P,\tau)$, $(Q,\tau)$, $(S,\tau)$ be separable tracial von Neumann algebras. 
Assume that $\mathcal H,\mathcal K,\mathcal L$ are separable Hilbert spaces such that
\begin{itemize}
\item[(1)] $\mathcal H$ is an $M$-$M\overline{\otimes}N$-bimodule  such that the $M$-$N$-bimodule $_{M}\mathcal H_{1\overline{\otimes} N}$ is left weakly mixing.
\item[(2)] $\mathcal K$ is an $M$-$M^{\emph{op}}\overline{\otimes}P$-bimodule such that the $M$-$P$-bimodule $_{M}\mathcal K_{1\overline{\otimes} P}$ is left weakly mixing. 
\item[(3)]  $\mathcal L$ is an $Q$-$M\overline{\otimes}S$-bimodule such that the $Q$-$S$ bimodule $_{Q}\mathcal L_{1\overline{\otimes} S}$ is left weakly mixing.
\end{itemize}

Then there exists a hyperfinite subfactor $R\subset M$ satisfying the following properties:
\begin{enumerate}
\item [(a)] The $R$-$R\overline{\otimes}N$-bimodule $\mathcal H\ominus\overline{\emph{sp}}(R\mathcal H^M)$ is coarse.
\item [(b)] The $R$-$R^{\emph{op}}\overline{\otimes}P$-bimodule $\mathcal K$ is coarse.
\item [(c)] The $Q$-$R\overline{\otimes}S$-bimodule $\mathcal L$ is  left weakly mixing.
\end{enumerate}

\end{theorem}
We note that the formulation of (a) implicitly uses the fact that $\overline{\text{sp}}(R\mathcal H^M)$ is an $R$-$R\overline{\otimes}N$-bimodule.

\begin{proof} 

We will define $R=(\otimes_{i=1}^\infty A_i)''$, where $(A_i)_{i=1}^\infty$ is a well-chosen sequence of pairwise commuting copies of $\mathbb M_2(\mathbb C)$ inside $M$. We start by fixing some notation:

\begin{itemize}
\item Let $\{u_j\}_{0\leq j\leq 3}$ be unitaries spanning $\mathbb M_2(\mathbb C)$ with $u_0=u_1^2=u_2^2=1$ and $u_3=u_1u_2=-u_2u_1$. 
Specifically, we can take $u_1=\begin{pmatrix}1&0\\0&-1 \end{pmatrix}$, $u_2=\begin{pmatrix} 0&1\\1&0 \end{pmatrix}$ and $u_3=\begin{pmatrix}0&1\\-1&0 \end{pmatrix}$.

\item For $i\geq 1$, let $\Delta_i:=\{u_{i,j}\}_{0\leq j\leq 3}$ be a copy of $\{u_j\}_{0\leq j\leq 3}$ inside $A_i$. 

\item For  $n\geq 1$, let $R_n=\otimes_{i=1}^nA_i$ and $U_n=\Delta_1\cdots\Delta_n\subset\mathcal U(R_n)$. Put $R_0=\mathbb C1$ and $U_0=\{1\}$.

\item For a Hilbert subspace $\mathcal H_0\subset\mathcal H$, let $\mathbb P_{\mathcal H_0}$ denote the orthogonal projection onto $\mathcal H_0$.

\item Let $\{\xi_n\}_{n\geq 1}\subset \mathcal{H}$, $\{\eta_n\}_{n\geq 1}\subset\mathcal{K}$ and $\{\zeta_n\}_{n\geq 1}\subset\mathcal L$ be dense sequences. 

\item For $n\geq 1$, let $F_n=\{\xi_k\}_{1\leq k\leq n}$,  $\widetilde F_n=U_{n-1}\mathbb P_{\mathcal H\ominus\mathcal H^{R_{n-1}'\cap M}}(F_n)U_{n-1}$, $\widehat{F}_n=\mathbb P_{\overline{\text{sp}}(M\mathcal H^M)}(F_n)$,  $G_n=\{\eta_k\}_{1\leq k\leq n}$, $\widetilde G_n=U_{n-1}G_nU_{n-1}$ and $H_n=\{\zeta_k\}_{1\leq k\leq n}$.

\end{itemize}

{\bf Part I.} In the first part of the proof, we construct recursively the sequence $(A_i)_{i=1}^\infty$  and unitaries  $(v_i)_{i=1}^\infty\subset\mathcal U(Q)$ such that, in the above notation,  the following conditions  hold for every $n\geq 1$:

\begin{enumerate}

\item[(1n)] $\sup_{x\in (N)_1}|\langle y_1\xi (y_2\otimes x),\xi'\rangle|\leq 2^{-5n}$, for every $(y_1,y_2)\in (\Delta_n\times \Delta_n)\setminus \{(1,1)\}$ and $\xi,\xi'\in \widetilde F_n $. 
\item[(2n)] $\|\mathbb P_{\mathcal H^{R_n'\cap M}}(\sigma)-\mathbb P_{\mathcal H^{R_{n-1}'\cap M}}(\sigma)\|\leq 2^{-5n}$, for every $\sigma\in F_n$.
\item[(3n)] $\sup_{z\in (P)_1}|\langle y_1\eta(y_2^{\text{op}}\otimes z),\eta'\rangle|\leq 2^{-5n}$, for every $(y_1,y_2)\in (\Delta_n\times \Delta_n)\setminus \{(1,1)\}$ and $\eta,\eta'\in \widetilde G_n$.
\item[(4n)] $\sup_{t\in (S)_1}|\langle \zeta(y\otimes t),\zeta'\rangle|\leq 2^{-3n}$, for every $y\in \Delta_n\setminus\{1\}$ and $\zeta,\zeta'\in\widetilde H_n:=H_n\cup(\cup_{i=1}^n(v_iH_nU_{n-1}))$.
\item[(5n)]  $\sup_{u\in U_{n-1},t\in (S)_1}|\langle v_n\zeta(u\otimes t),\zeta'\rangle|\leq 2^{-3n}$, for every $\zeta,\zeta'\in H_n$.
\end{enumerate}

Suppose that for some $n\geq 1$ we have constructed $(A_i)_{i=1}^{n-1}$  and $(v_i)_{i=1}^{n-1}$ such that  $(1k)$-$(4k)$ are satisfied for $k=1,\cdots,n-1$. Since the  $Q$-$S$ bimodule $_Q\mathcal L_{1\otimes S}$ is left weakly mixing, there exists $v_{n}\in\mathcal U(Q)$ such that  (5n) holds. Our goal is to construct a von Neumann subalgebra $A_n\subset R_{n-1}'\cap M$  such that $A_n\cong\mathbb M_2(\mathbb C)$  and conditions (1n)-(4n) hold.

Towards this goal, we first prove:

\begin{claim}\label{f}
There exists $f:\overline{\emph{sp}}(M\mathcal H^M)\rightarrow \emph{L}^2(M)\otimes\emph{L}^2(M)$ such that 
\begin{equation}\label{map}\text{$\|\mathbb P_{\emph{sp}(wR_{n-1}\mathcal H^M)}(\sigma)\|^2=\langle wf(\sigma) w^*,1\otimes 1\rangle$, for every $\sigma\in \overline{\emph{sp}}(M\mathcal H^M)$ and $w\in\mathcal U(M)$.}\end{equation}
\end{claim}

\begin{proof}[Proof of Claim \ref{f}.]
Lemma \ref{central_space} gives an $M$-bimodular unitary $T:\overline{\text{sp}}(M\mathcal H^M)\rightarrow\text{L}^2(M)\otimes\ell^2(I)$, for some set $I$, with $T(\mathcal H^M)=\mathbb C1\otimes\ell^2(I)$.
Let $\sigma\in\overline{\text{sp}}(M\mathcal H^M)$. Write $T(\sigma)=\oplus_{i\in I}\sigma_i$, with $\sigma_i\in\text{L}^2(M)$, for $i\in I$.
As $\|\sigma_i^*\otimes \sigma_i\|_2=\|\sigma_i\|_2^2$, we have $\sum_{i\in I}\|\sigma_i^*\otimes \sigma_i\|_2=\sum_{i\in I}\|\sigma_i\|_2^2=\|T(\sigma)\|_2^2<\infty$.
Thus, we can define $f(\sigma)\in\text{L}^2(M)\otimes\text{L}^2(M)$ by $f(\sigma)=\sum_{i\in I,u\in U_{n-1}}u^*\sigma_i^*\otimes \sigma_iu$.
Let $w\in\mathcal U(M)$. Since $T$ is $M$-modular, $T(\text{sp}(wR_{n-1}\mathcal H^M))=\text{sp}(wR_{n-1}T(\mathcal H^M))=wR_{n-1}\otimes\ell^2(I)$.
Since $wU_{n-1}$ is an orthonormal basis of $wR_{n-1}$ and  $|\tau(x)|^2=\langle x^*\otimes x,1\otimes 1\rangle$, for every $x\in M$, we get that
\begin{align*}
\|\mathbb P_{\text{sp}(wR_{n-1}\mathcal H^M)}(\sigma)\|^2&=\|\mathbb P_{wR_{n-1}\otimes\ell^2(I)}(T(\sigma))\|^2\\&=\sum_{i\in I}\|\mathbb P_{wR_{n-1}}(\sigma_i)\|^2\\&=\sum_{i\in I,u\in U_{n-1}}|\tau(\sigma_iu^*w^*)|^2
\\&=\sum_{i\in I,u\in U_{n-1}}\langle wu\sigma_i^*\otimes\sigma_iu^*w^*,1\otimes 1\rangle=\langle wf(\sigma) w^*,1\otimes 1\rangle,
\end{align*}
which proves \eqref{map}.
\end{proof}

Since the bimodules $_M\mathcal{H}_{1\overline{\otimes} N}$ and $_{M}\mathcal{K}_{1\overline{\otimes} P}$ are left weakly mixing and $R_{n-1}'\cap M\subset M$ has finite index, the bimodules $_{R_{n-1}'\cap M}\mathcal{H}_{1\overline{\otimes} N}$ and $_{R_{n-1}'\cap M}\mathcal{K}_{1\overline{\otimes} P}$ are left weakly mixing. By construction, since $U_{n-1}\subset R_{n-1}$, we have $\widetilde{F}_{n}\subset U_{n-1}(\mathcal H\ominus\mathcal H^{R_{n-1}'\cap M})U_{n-1}= \mathcal H\ominus\mathcal H^{R_{n-1}'\cap M}$. Since $\mathcal L$ is a right $M\overline{\otimes}S$-module, we can view it a submodule of $\text{L}^2(M\overline{\otimes}S)\otimes\ell^2(J)$, for some set $J$. We endow  $\text{L}^2(M\overline{\otimes}S)\otimes\ell^2(J)$ with the usual left $M$-module structure, and notice that it is a left weakly mixing  $M$-$1\overline{\otimes}S$-bimodule. Finally, note that the coarse $M$-bimodule $\text{L}^2(M)\otimes\text{L}^2(M)$ is left weakly mixing.
Lemma \ref{one_unitary} and Remark \ref{unitaryleftaction} provide $a\in\mathcal{U}(R_{n-1}'\cap M)$ such that $a^2=1$, $\tau(a)=0$, and the following hold:
\begin{enumerate}
    \item[(a1)] $\sup_{x\in (N)_1}|\langle y_1\xi(y_2\otimes x),\xi'\rangle|\leq 2^{-5n}$, for every $y_1,y_2\in\{1,a\}$, $(y_1,y_2)\not=(1,1)$, and $\xi,\xi'\in \widetilde{F}_{n}$.
      \item [(a2)] $\langle af(\sigma)a,1\otimes 1\rangle\leq \frac{1}{3}2^{-10n}$, for every $\sigma\in\widehat{F}_n$.
    \item[(a3)] $\sup_{z\in (P)_1}|\langle y_1\eta(y_2^{\text{op}}\otimes z),\eta'\rangle|\leq 2^{-5n}$, for every $y_1,y_2\in\{1,a\}$, $(y_1,y_2)\not=(1,1)$, and  $\eta,\eta'\in \widetilde{G}_{n}$.
    \item[(a4)] $\sup_{t\in (S)_1}|\langle\zeta (a\otimes t),\zeta'\rangle|\leq 2^{-3n}$, for every $\zeta,\zeta'\in \widetilde{H}_{n}$.

\end{enumerate}

Next, since $R_{n-1}'\cap M$ is a II$_1$ factor, $a^2=1$ and $\tau(a)=0$, we can find $b_0\in \mathcal U(R_{n-1}'\cap M)$ such that $b_0^2=1$, $\tau(b_0)=0$, $ab_0=-b_0a$ and therefore  $\{ a,b_0\}''\cong \mathbb M_{2}(\mathbb C)\subset R_{n-1}'\cap M$. 
Denote $M_0=\{a,b_0\}'\cap (R_{n-1}'\cap M)$.
Since the inclusion  $M_0\subset M$ has finite index, $\mathcal{H}$ is left weakly mixing as an $M_0$-$(1\overline{\otimes} N)$-bimodule, and $\mathcal K$ is left weakly mixing as an $M_0$-$(1\overline{\otimes} P)$-bimodule. 

Let $T=\{1,a,b_0,ab_0\}$ and $T^{\text{op}}=\{x^{\text{op}}\mid x\in T\}$. Put $X=T\widetilde{F}_{n}T$, $Y=T\widetilde{G}_{n}T^{\text{op}}$ and $Z=\widetilde{H}_{n}T$.
Note that $X\subset\{a,b_0\}''\widetilde{F}_n\{a,b_0\}''\subset\{a,b_0\}''  (\mathcal H\ominus\mathcal H^{R_{n-1}'\cap M})\{a,b_0\}''\subset\mathcal H\ominus\mathcal H^{M_0}.$
By using this fact and applying Lemma \ref{one_unitary} in the same manner as above we find $b_1\in \mathcal{U}(M_0)$ such that $b_1^2=1$, $\tau(b_1)=0$, and the following hold:
\begin{enumerate}
  \item[(b1)] $\sup_{x\in (N)_1}|\langle y_1\xi(y_2\otimes x),\xi'\rangle|\leq 2^{-5n}$, for every $y_1,y_2\in\{1,b_1\}$, $(y_1,y_2)\not=(1,1)$ and $\xi,\xi'\in X$.
    \item [(b2)] $\langle b_1(b_0f(\sigma) b_0)b_1,1\otimes 1\rangle\leq \frac{1}{3}2^{-10n}$ and $\langle b_1(b_0af(\sigma) ab_0)b_1,1\otimes 1\rangle\leq \frac{1}{3}2^{-10n}$, for every $\sigma\in\widehat{F}_n$.
      \item[(b3)] $\sup_{z\in (P)_1}|\langle y_1\eta(y_2^{\text{op}}\otimes z),\eta'\rangle|\leq 2^{-5n}$, for every $y_1,y_2\in\{1,b_1\}$, $(y_1,y_2)\not=(1,1)$ and $\eta,\eta'\in Y$.
    \item[(b4)] $\sup_{t\in (S)_1}|\langle\zeta (b_1\otimes t),\zeta'\rangle|\leq 2^{-3n}$, for every $\eta,\eta'\in Z$.
\end{enumerate}
Put $b=b_0 b_1\in\mathcal{U}(R_{n-1}'\cap M)$. Since $\tau(b)=0$, $b^2=1$ and $ab=ab_0b_1=-b_0ab_1=-b_0b_1a=-ba$, we derive that $\{a,b\}''\cong \mathbb M_2(\mathbb C)\subset R_{n-1}'\cap M_0$. We will show that $A_n=\{a,b\}''$ and  $\Delta_n=\{1,a,b,ab\}$ satisfy inequalities (1n)-(4n).

It is immediate that (a1) and (b1) imply (1n), that (a3) and (b3) imply (3n), and that (a4) and (b4) imply (4n). 
Therefore, it remains to verify (2n). Let $\sigma\in F_n$ and define $\sigma_0=\mathbb P_{\overline{\text{sp}}(M\mathcal H^M)}(\sigma)\in\widehat{F}_n$.
Lemma \ref{fixed_point_proj}(2) gives that $\mathcal H^{R_{n-1}'\cap M}=\text{sp}(R_{n-1}\mathcal H^M)$ and $\mathcal H^{R_n'\cap M}=\text{sp}(R_n\mathcal H^M)$. Since $M$ is a II$_1$ factor, $x\mathcal H^M\perp\mathcal H^M$, for every $x\in M$ with $\tau(x)=0$. Also,  $R_n\ominus R_{n-1}=aR_{n-1}\oplus bR_{n-1}\oplus abR_{n-1}$.
These three facts together imply that 
\begin{align*}
\mathcal H^{R_n'\cap M}\ominus\mathcal H^{R_{n-1}'\cap M}&=\text{sp}(R_n\mathcal H^M)\ominus\text{sp}(R_{n-1}\mathcal H^M)\\&=\text{sp}(aR_{n-1}\mathcal H^M)\oplus\text{sp}(bR_{n-1}\mathcal H^M)\oplus\text{sp}(abR_{n-1}\mathcal H^M).
\end{align*}
Since $\mathbb P_{\mathcal H^{R_n'\cap M}}(\sigma)-\mathbb P_{\mathcal H^{R_{n-1}'\cap M}}(\sigma)=\mathbb P_{\mathcal H^{R_n'\cap M}}(\sigma_0)-\mathbb P_{\mathcal H^{R_{n-1}'\cap M}}(\sigma_0)$, we deduce that
$$\|\mathbb P_{\mathcal H^{R_n'\cap M}}(\sigma)-\mathbb P_{\mathcal H^{R_{n-1}'\cap M}}(\sigma)\|^2=\|\mathbb P_{\text{sp}(aR_n\mathcal H^M)}(\sigma_0)\|^2+\|\mathbb P_{\text{sp}(bR_n\mathcal H^M)}(\sigma_0)\|^2+\|\mathbb P_{\text{sp}(abR_n\mathcal H^M)}(\sigma_0)\|^2.$$
By combining this identity with \eqref{map}, (a2) and (b2), we get that 
\begin{align*}\|\mathbb P_{\mathcal H^{R_n'\cap M}}(\sigma)-\mathbb P_{\mathcal H^{R_{n-1}'\cap M}}(\sigma)\|^2=\langle af(\sigma_0)a,1\otimes 1\rangle+\langle bf(\sigma_0)b,1\otimes 1\rangle+\langle abf(\sigma_0)ba, 1\otimes 1\rangle\leq 2^{-10n}.
\end{align*}
Hence, $\|\mathbb P_{\mathcal H^{R_n'\cap M}}(\sigma)-\mathbb P_{\mathcal H^{R_{n-1}'\cap M}}(\sigma)\|\leq 2^{-5n}$, for every $\sigma\in F_n$, which proves (2n). This finishes the first part of the proof.

{\bf Part II.} Define $R=(\otimes_{i=1}^\infty A_i)''$. Then $R\subset M$ is a  hyperfinite II$_1$ subfactor. 
In the second part of the proof we prove that $R$ satisfies conditions (a)-(c) from the hypothesis. 

Let $U=\cup_{n\geq 1}U_n\subset\mathcal U(R)$ and note that $U$ is an orthonormal basis for $\text{L}^2(R)$.

First, we prove (a). To this end, fix $\sigma=\xi_k$, for some $k\geq 1$. If $n\geq k$, we have
  $U_{n-1}\Delta_n=U_n$ and $\sigma\in F_n$. Applying (1n) to $\xi'=\xi=u_1(\sigma-\mathbb P_{\mathcal H^{R_{n-1}'\cap M}}(\sigma))u_2\in\widetilde{F}_n$ for $u_1,u_2\in U_{n-1}$  gives that for every  $(y_1,y_2)\in (U_n\times U_n)\setminus (U_{n-1}\times U_{n-1})$ we have
\begin{equation}\label{sigma0}\text{$\sup_{x\in (N)_1}|\langle y_1(\sigma-\mathbb P_{\mathcal H^{R_{n-1}'\cap M}}(\sigma))(y_2\otimes x),(\sigma-\mathbb P_{\mathcal H^{R_{n-1}'\cap M}}(\sigma))\rangle|\leq 2^{-5n}$, for every $n\geq k$}.\end{equation}
Since  $\sigma\in F_n$, applying (2n) to $\sigma$ also gives that 
\begin{equation}\label{sigma1}\text{$\|\mathbb P_{\mathcal H^{R_n'\cap M}}(\sigma)-\mathbb P_{\mathcal H^{R_{n-1}'\cap M}}(\sigma)\|\leq 2^{-5n}$, for every $n\geq k$.}\end{equation}
By Lemma \ref{fixed_point_proj} for every  $m\geq 1$ we have $\mathcal H^{R_m'\cap M}=\text{sp}(R_m\mathcal H^M)$. Since $\cup_{m\geq 1}\text{sp}(R_m\mathcal H^M)$ is dense in $\overline{\text{sp}}(R\mathcal H^M)$, we conclude that
 $\mathbb P_{\mathcal H^{R_m'\cap M}}=\mathbb P_{\text{sp}(R_m\mathcal H^M)}\rightarrow\mathbb P_{\overline{\text{sp}}(R\mathcal H^M)}$ in the  SOT, as  $m\rightarrow\infty$. Hence,  using \eqref{sigma1} we derive that for every $n\geq k$ we have
\begin{equation}\label{sigma2}
\text{$\|\mathbb P_{\mathcal H^{R_{n-1}'\cap M}}(\sigma)-\mathbb P_{\overline{\text{sp}}(R\mathcal H^M)}(\sigma)\|\leq\sum_{m=n}^\infty 2^{-5m}<2^{-5n+1}$.}
\end{equation}
Combining \eqref{sigma0} with \eqref{sigma2} gives that  for every $n\geq k$ and $(y_1,y_2)\in (U_n\times U_n)\setminus (U_{n-1}\times U_{n-1})$ 
\begin{equation}\label{U}\text{$c_k(y_1,y_2):=\sup_{x\in (N)_1}|\langle y_1(\sigma-\mathbb P_{\overline{\text{sp}}(R\mathcal H^M)}(\sigma))(y_2\otimes x),(\sigma-\mathbb P_{\overline{\text{sp}}(R\mathcal H^M)}(\sigma))\rangle|\leq 2^{-5n+3}$.}\end{equation}
Since $U_{n-1}\subset U_n$ and $|U_n|=2^{2n}$, for every $n\geq 1$, and $U=\cup_{n\geq 1}U_n$, by using \eqref{U} we get that
\begin{align*}\sum_{(y_1,y_2)\in (U\times U)\setminus (U_{k-1}\times U_{k-1})}c_k(y_1,y_2)&\leq\sum_{n=k}^\infty 2^{-5n+3}|(U_n\times U_n)\setminus (U_{n-1}\times U_{n-1})|\\&\leq\sum_{n=k}^\infty 2^{4n}2^{-5n+3}=2^{-k+4}.\end{align*}
This inequality allows us to conclude that $c_k\in\ell^1(U\times U)$. Since $U$ is an orthonormal basis for $\text{L}^2(R)$ and the sequence $\{\xi_k-\mathbb P_{\overline{\text{sp}}(R\mathcal H^M)}(\xi_k)\}_{k\geq 1}$ is dense in $\mathcal H\ominus\overline{\text{sp}}(R\mathcal H^M)$, Lemma \ref{coarse-criterion} implies (a). 

Similarly, combining (3n) with Lemma \ref{coarse-criterion} implies (b). We leave the details to the reader.

To check (c), let $n\geq 1$ and $\zeta,\zeta'\in H_n$. Then (5n) gives  $\sup_{u\in U_{n-1},t\in (S)_1}|\langle v_n\zeta (u\otimes t),\zeta'\rangle|\leq 2^{-3n}$. If $m\geq n$, then since $H_n\subset H_m$, we get $\zeta,\zeta'\in H_m$. Thus, $v_n\zeta (U_{m-1}\otimes 1)\subset \widetilde H_m$ and $\zeta'\in\widetilde H_m$.
Since $U_{m-1}(\Delta_m\setminus\{1\})=U_m\setminus U_{m-1}$, 
condition (4m) gives $\sup_{u\in U_m\setminus U_{m-1},t\in (S)_1}|\langle v_n\zeta (u\otimes t),\zeta'\rangle|\leq 2^{-3m}$. 

Let $x\in (R\overline{\otimes}S)_1$. Since $U$ is an orthonormal basis of $\text{L}^2(R)$, we get that $x=\sum_{u\in U}u\otimes t_u$, where $t_u=\text{E}_{S}(u^*x)$ satisfies $\|t_u\|\leq \|x\|\leq 1$, for every $u\in U$. Since $U=U_{n-1}\cup(\bigcup_{m\geq n}(U_{m}\setminus U_{m-1}))$, and we have $|U_{n-1}|=2^{2(n-1)}$, $|U_m\setminus U_{m-1}|=3\cdot 2^{2(m-1)}$, using (5n) we get that 
\begin{align*}|\langle v_n\zeta x,\zeta'\rangle|&\leq\sum_{u\in U}|\langle v_n\zeta (u\otimes t_u),\zeta'\rangle| \\ &\leq 2^{2(n-1)}\sup_{u\in U_{n-1},t\in (S)_1}|\langle v_n\zeta(u\otimes t),\zeta'\rangle|+\sum_{m\geq n}3\cdot 2^{2(m-1)}\sup_{u\in U_m\setminus U_{m-1},t\in (S)_1}|\langle v_n\zeta (u\otimes t),\zeta'\rangle|  \\&\leq 2^{2(n-1)}2^{-3n}+\sum_{m\geq n} 3\cdot 2^{2(m-1)} 2^{-3m}=2^{-n-2}+3\cdot 2^{-n-1}<2^{-n+1}.\end{align*}
Hence, we get that $\sup_{x\in (R\overline{\otimes}S)_1}|\langle v_n\zeta x,\zeta'\rangle|<2^{-n+1}$, for every $n\geq 1$ and $\zeta,\zeta'\in H_n$. This gives that $\sup_{x\in (R\overline{\otimes}S)_1}|\langle v_n\zeta x,\zeta'\rangle|\rightarrow 0$, for every $\zeta,\zeta'\in\mathcal L$, which implies (c).
\end{proof}

The following consequence of Theorem \ref{bimodulenonintertwining} is a more general version of Corollary \ref{corG} that will be needed in the proof of Theorem \ref{embeddings}.

\begin{theorem}\label{homomorphisms}

Let $M$ be a separable II$_1$ factor and $M_0\subset M$ be an irreducible subfactor. Let $(\widetilde N_j,\tau_j)_{j\geq 1}, (N_{j})_{j\geq N}, (P_k,\tau_k)_{k\geq 1},(Q_\ell,\tau_\ell)_{\ell\geq 1}, (S_\ell,\tau_\ell')_{\ell\geq 1}$ be tracial von Neumann algebras such that $N_{j}\subset \widetilde N_j$, for every $j\geq 1$.
Assume that we are given the following data:

\begin{enumerate}

\item[(1)] For every $j\geq 1$, a $*$-homomorphism $\alpha_j:M\rightarrow M\overline{\otimes}\widetilde N_j$ such that $\alpha_j(M_0)\nprec_{M\overline{\otimes}\widetilde N_j}1\overline{\otimes} N_{j}$ and there is no nonzero $v\in M\overline{\otimes}\widetilde N_j$ such that $\alpha_j(x)v=v(x\otimes 1)$, for every $x\in M_0$.

\item[(2)] For every $k\geq 1$, a $*$-homomorphism $\beta_k:M\rightarrow M^{\emph{op}}\overline{\otimes}P_k$ such that $\beta_k(M_0)\nprec_{M^{\emph{op}}\overline{\otimes}P_k}1\overline{\otimes} P_k$.
\item[(3)] For every $\ell\geq 1$, a $*$-homomorphism $\gamma_\ell:Q_\ell\rightarrow M\overline{\otimes}S_\ell$ such that $\gamma_l(Q_\ell)\nprec_{M\overline{\otimes}S_\ell}1\overline{\otimes} S_\ell$.

\end{enumerate}

Then there exists a hyperfinite subfactor $R\subset M_0$ 
such that the following conditions hold:
\begin{enumerate}
\item[(a)]  $_{\alpha_j(R)}\emph{L}^2(M\overline{\otimes}\widetilde N_j)_{R\overline{\otimes} N_j}$ is a coarse bimodule, so $\alpha_j(R)\nprec_{M\overline{\otimes}\widetilde N_j}R\overline{\otimes}N_{j}$, for every $j\geq 1$.
\item[(b)]  $_{\beta_k(R)}\emph{L}^2(M^{\emph{op}}\overline{\otimes}P_k)_{R^{\emph{op}}\overline{\otimes}P_k}$ is a coarse bimodule, so $\beta_k(R)\nprec_{M^{\emph{op}}\overline{\otimes}P_k}R^{\emph{op}}\overline{\otimes}P_k$, for every $k\geq 1$.
\item[(c)] $\gamma_\ell(Q_\ell)\nprec_{M\overline{\otimes}S_\ell}R\overline{\otimes}S_\ell$, for every $\ell\geq 1$.

\item[(d)] $R$ is coarse in $M$, i.e., the $R$-bimodule $\emph{L}^2(M)\ominus\emph{L}^2(R)$ is coarse.
\end{enumerate}

\end{theorem}

\begin{proof} Let $\widetilde N_0=N_0=\mathbb C1$ with its unique trace $\tau_0$.
Define the following tracial von Neumann algebras: $(\widetilde N,\tau)=(\oplus_{j\geq 0}\widetilde N_j,\sum_{j\geq 0}2^{-j-1}\tau_j)$,  $N=\oplus_{j\geq 0}N_{j}$, $(P,\tau)=(\oplus_{k\geq 1}P_k,\sum_{k\geq 1}2^{-k}\tau_k)$, $(Q,\tau)=(\oplus_{\ell\geq 1}Q_\ell,\sum_{\ell\geq 1}2^{-\ell}\tau_\ell)$ and $(S,\tau)=(\oplus_{\ell\geq 1}S_\ell,\sum_{\ell\geq 1}2^{-\ell}\tau_{\ell}')$. Let $\alpha_0:M\rightarrow M=M\overline{\otimes}\widetilde N_0$ be the identity.
Then $M\overline{\otimes}\widetilde N=\oplus_{j\geq 0}(M\overline{\otimes}\widetilde N_j), M^{\text{op}}\overline{\otimes}P=\oplus_{k\geq 1}(M^{\text{op}}\overline{\otimes}P_k)$ and $M\overline{\otimes}S=\oplus_{\ell\geq 1}(M\overline{\otimes}S_\ell)$.
Define the $*$-homomorphisms $\alpha=\oplus_{j\geq 0}\alpha_j:M\rightarrow M\overline{\otimes}\widetilde N$,  $\beta=\oplus_{k\geq 1}\beta_k:M\rightarrow M^{\text{op}}\overline{\otimes}P$ and
 $\gamma=\oplus_{\ell\geq 1}\gamma_\ell:Q\rightarrow M\overline{\otimes}S$.
Let
\begin{itemize}
\item $\mathcal H=\text{L}^2(M\overline{\otimes}\widetilde N)$ with the $M_0$-$M_0\overline{\otimes}N$-bimodule structure given by $x\xi y=\alpha(x)\xi y$,
\item $\mathcal K=\text{L}^2(M^{\text{op}}\overline{\otimes}P)$ with the $M_0$-$M_0^{\text{op}}\overline{\otimes}P$-bimodule  structure given by $x\xi y=\beta(x)\xi y$, and
\item $\mathcal L=\text{L}^2(M\overline{\otimes}S)$ with the $Q$-$M_0\overline{\otimes}S$-bimodule structure given by $x\xi y=\gamma(x)\xi y$.
\end{itemize}

Since $M_0$ is diffuse, the $M_0$-$\mathbb C 1$-bimodule $_{M_0}\text{L}^2(M)_{\mathbb C1}$ is left weakly mixing. If $j\geq 1$, then since 
$\alpha_j(M_0)\nprec_{M\overline{\otimes}\widetilde N_j}1\overline{\otimes} N_{j}$ by (1), the $M_0$-$N_j$-bimodule $_{\alpha_j(M_0)}\text{L}^2(M\overline{\otimes}\widetilde N_j)_{1\overline{\otimes} N_j}$ is left weakly mixing. Hence, the  $M_0$-$N$-bimodule $_{M_0}\mathcal H_{1\overline{\otimes} N}$ is left weakly mixing.
Similarly, the $M_0$-$P$ bimodule $_{M_0}\mathcal K_{1\overline{\otimes} P}$ and the $Q$-$S$-bimodule $_Q\mathcal L_{1\overline{\otimes} S}$ are left weakly mixing. 
Since $\mathcal H=\text{L}^2(M)\oplus(\oplus_{j\geq 1}\text{L}^2(M\overline{\otimes}\widetilde N_j))$, by using (1) again we derive that $\mathcal H^{M_0}=\mathbb C1\oplus(\oplus_{j\geq 1}\{0\})$.

By applying Theorem \ref{bimodulenonintertwining} we find a hyperfinite subfactor $R\subset M_0$ satisfying the following properties:
\begin{enumerate}
\item [(i)] The $R$-$R\overline{\otimes}N$-bimodule $\mathcal H\ominus\overline{\text{sp}}(R\mathcal H^M_0)=(\text{L}^2(M)\ominus\text{L}^2(R))\oplus(\oplus_{j\geq 1}\text{L}^2(M\overline{\otimes}\widetilde N_j))$ is coarse.
\item [(ii)] The $R$-$R^{\text{op}}\overline{\otimes}P$-bimodule $\mathcal K$ is a multiple of the coarse $R$-$R^{\text{op}}\overline{\otimes}P$-bimodule.
\item [(iii)] The $Q$-$R\overline{\otimes}S$-bimodule $\mathcal L$ is  left weakly mixing.
\end{enumerate}
Condition (i)  implies (d). It also gives that the $R$-$R\overline{\otimes}N_j$-bimodule $_{\alpha_j(R)}\text{L}^2(M\overline{\otimes}\widetilde N_j)_{R\overline{\otimes}N_j}$ is coarse and so (a) holds, for every $j\geq 1$. Similarly, (b) and (c) follow from (ii) and (iii), respectively. 
\end{proof}

\section{Rigidity of embeddings}\label{Sec:embeddings}
This section is devoted to proving the following rigidity theorem, from which the main results of the paper, Theorems \ref{no_crossed1} and \ref{no_comult1}, follow.

\begin{theorem}\label{embeddings} Let $M_1$ and $M_2$ be II$_1$ factors with property (T). Then there exist coarse embeddings of the hyperfinite II$_1$ factor $\alpha_1:R\rightarrow M_1$ and $\alpha_2:R\rightarrow M_2$ such that
 $M=M_1*_RM_2$ is a II$_1$ factor which satisfies the following properties.
 
 \begin{enumerate}
 \item
Let $\theta:M\rightarrow\mathcal M$ be any embedding, where $\mathcal M=M^{\mathbb N}\overline{\otimes}(M^{\emph{op}})^{\mathbb N}$. 
For $n\in\mathbb N$,  identify $\mathcal M=M^{\{n\}}\overline{\otimes}(M^{\mathbb N\setminus\{n\}}\overline{\otimes}(M^{\emph{op}})^{\mathbb N})$ and let $\Delta_n:M\rightarrow\mathcal M$ be the embedding $\Delta_n(x)=x\otimes 1$.

\vspace{2mm}

\noindent Then there exist projections $(p_n)_{n\in\mathbb N}\subset\mathcal Z(\theta(M)'\cap\mathcal M)$ and $(q_n)_{n\in\mathbb N}\subset\mathcal M$, and unitaries $(u_n)_{n\in\mathbb N}\subset\mathcal M$ such that we have $\sum_{n\in\mathbb N}p_n=1$, 
$q_n\in\Delta_n(M)'\cap\mathcal M$, for every $n\in\mathbb N$, and $\theta(x)p_n=u_n\Delta_n(x)q_nu_n^*$, for every $x\in M$. 

\vspace{2mm}

\item For $\rho\in\emph{Aut}(R)$, denote $M(\rho)=M_1*_RM_2$ the amalgamated free product associated with the embeddings $\alpha_1\circ\rho:R\rightarrow M_1$ and $\alpha_2:R\rightarrow M_2$. 
Then
\begin{itemize} 
\item[(a)] The property in (1) holds if we replace $M$ by $M(\rho)$, for every $\rho\in\emph{Aut}(R)$.
\item[(b)] For every $\rho,\sigma\in\emph{Aut}(R)$, the following  are equivalent:  
\begin{itemize}
\item[($i$)] $M(\rho)$ is isomorphic to $M(\sigma)$.
\item[($ii$)] $M(\rho)$ embeds into $M(\sigma)$.
\item[($iii$)] $M(\rho)$ stably embeds into $M(\sigma)$.
\item[($iv$)]$M(\rho)$ embeds into $M(\sigma)^{\mathbb N}\overline{\otimes}(M(\sigma)^{\emph{op}})^{\mathbb N}$.
\item[($v$)] $\rho^{-1}\sigma\in\emph{Inn}(R)$.
\end{itemize}
\item[(c)] Let $\rho\in\emph{Aut}(R)$ and $(\rho_n)_{n\in\mathbb N}\subset\emph{Aut}(R)$. Then $M(\rho)$ embeds into $\overline{\otimes}_{n\in\mathbb N}M(\rho_n)$ if and only if $\rho^{-1}\rho_n\in\emph{Inn}(R)$ (equivalently, by (b), $M(\rho)\cong M(\rho_n)$), for some $n\in\mathbb N$.
\end{itemize}

\end{enumerate}
\end{theorem}

In preparation for the proof of Theorem \ref{embeddings}, we introduce two new notions for embeddings and prove a key technical result, see Theorem \ref{embedd2}.

\begin{definition}\label{pouter}
Let $(M,\tau)$ be a tracial von Neumann algebra and $p\in\mathbb M_n(\mathbb C)\overline{\otimes}M$ be a projection, for some $n\in\mathbb N$. We say that an embedding $\theta:M\rightarrow p(\mathbb M_n(\mathbb C)\overline{\otimes}M)p$ is {\it properly outer} if there exists no nonzero $v\in p(\mathbb M_n(\mathbb C)\overline{\otimes}M)$ satisfying $\theta(x)v=v(1\otimes x)$, for every $x\in M$.
\end{definition}

This notion generalizes the usual notion of proper outerness for automorphisms $\theta:M\rightarrow M$ which requires that the restriction of $\theta$ to $Mz$ is outer, for every central projection $z\in M$ with $\theta(z)=z$.

\begin{definition}
Let $M, N_1,\cdots,N_k$ be tracial von Neumann algebras. Let $\mathcal M=\mathbb M_n(\mathbb C)\overline{\otimes}N_1\overline{\otimes}\cdots\overline{\otimes}N_k$ and $p\in\mathcal M$ be a projection, for $n\in\mathbb N$. We say that an embedding $\theta:M\rightarrow p\mathcal M p$ is  {\it non-degenerate} if it satisfies that $$\text{$\theta(M)\nprec_{\mathcal M}\mathbb M_n(\mathbb C)\overline{\otimes}N_1\overline{\otimes}\cdots\overline{\otimes} N_{j-1}\overline{\otimes}1\overline{\otimes}N_{j+1}\overline{\otimes}\cdots\overline{\otimes}N_k$, for every $1\leq j\leq k$.}$$

Note that if $k=1$ and $M$ is diffuse, then any embedding $\theta:M\rightarrow p(\mathbb M_n(\mathbb C)\overline{\otimes}N_1)p$ is non-degenerate.
\end{definition}

The proof of Theorem \ref{embeddings} relies on the following technical result.

\begin{theorem}\label{embedd2}
Let $M_1$ and $M_2$ be II$_1$ factors with property (T).  Put $\mathcal A=\{M_1,M_1^{\emph{op}},M_2,M_2^{\emph{op}}\}$. Then there exist coarse copies $R_1\subset M_1$ and $R_2\subset M_2$ of the hyperfinite II$_1$ factor such that denoting $f(M_1)=R_1,f(M_1^{\emph{op}})=R_1^{\emph{op}},f(M_2)=R_2$, $f(M_2^{\emph{op}})=R_2^{\emph{op}}$, the following conditions hold:
\begin{enumerate}

\item Let $\theta:M_i\rightarrow p\mathcal Mp$ be a non-degenerate embedding, where $i\in\{1,2\}$, 
  $\mathcal M=\mathbb M_n(\mathbb C)\overline{\otimes}N_1\overline{\otimes}\cdots\overline{\otimes}N_k$ and $p\in\mathcal M$ is a projection, for some $n,k\in\mathbb N$ and $N_1,\dots,N_k\in\mathcal A$.
Assume that either \\ (1) $k\geq 2$, (2) $k=1$ and $N_1\not=M_i$, or (3) $k=1$, $N_1=M_i$ and $\theta$ is properly outer.

\vspace{2mm}

\noindent Then $\theta(R_i)\nprec_{\mathcal M}\mathbb M_n(\mathbb C)\overline{\otimes} N_1\overline{\otimes}\cdots\overline{\otimes} N_{j-1}\overline{\otimes}f(N_j)\overline{\otimes}N_{j+1}\overline{\otimes}\cdots\overline{\otimes}N_k,$ for every $1\leq j\leq k$.

\vspace{2mm}

\item Let $\sigma_1:M_1\rightarrow p_1\mathcal Mp_1$ and $\sigma_2:M_2\rightarrow p_2\mathcal Mp_2$ be embeddings, where  $\mathcal M=\mathbb M_n(\mathbb C)\overline{\otimes}N_1\overline{\otimes}\cdots\overline{\otimes}N_k$ and $p\in\mathcal M$ is a projection, for some $n,k\in\mathbb N$ and $N_1,\dots,N_k\in\mathcal A$.

\vspace{2mm}

\noindent Then $\sigma_1(R_1)\nprec_{\mathcal M}\sigma_2(R_2)$ and $\sigma_2(R_2)\nprec_{\mathcal M}\sigma_1(R_1)$.

\end{enumerate}

\end{theorem}

\begin{proof}
Let $\mathcal S$ be the set of tuples $F=(N_1,\dots,N_k,n)$, where $k,n\in\mathbb N$ and $N_1,\cdots,N_k\in\mathcal A$. For $F=(N_1,\dots,N_k,n)\in\mathcal S$, put $\mathcal M_F=\mathbb M_n(\mathbb C)\overline{\otimes}N_1\overline{\otimes}\cdots\overline{\otimes}N_k$. 
For $i\in\{1,2\}$, let $\mathcal S_i$ be the subset of $\mathcal S$ consisting of $F=(N_1,\dots,N_k,n)\in\mathcal S$ such that
(1) $k\geq 2$, or (2) $k=1$ and $N_1\not=M_i$.

Fix $i\in\{1,2\}$, $F\in\mathcal S_i$, $G\in\mathcal S$ and $r\in\mathbb N$. 
Since $M_i$ has property (T), it is separable and Lemma \ref{sep_argument} implies that $(\mathcal E(M_i,\mathcal M),\text{d})$ is a separable metric space, for every separable tracial von Neumann algebra $(\mathcal M,\tau)$.
This allows us to find:

\begin{itemize}
\item[(i)] A sequence $(\theta_\ell^{i,F}:M_i\rightarrow\mathcal M_F)_{\ell\in\mathbb N}$ which is dense in the set of all non-degenerate, not necessarily unital, embeddings $\theta:M_i\rightarrow\mathcal M_F$.
\item[(ii)] A sequence $(\rho_m^{i,r}:M_i\rightarrow\mathbb M_r(\mathbb C)\overline{\otimes}M_i)_{m\in\mathbb N}$ which is dense in the set of all, not necessarily unital, properly outer embeddings $\rho:M_i\rightarrow\mathbb M_r(\mathbb C)\overline{\otimes}M_i$.
\item[(iii)] A sequence $(\sigma_s^{i,G}:M_i\rightarrow\mathcal M_G)_{s\in\mathbb N}$ which is dense in the set of all, not necessarily unital, embeddings $\sigma:M_i\rightarrow\mathcal M_G$. 
\end{itemize}

Then for every $F=(N_1,\dots,N_k,n)\in\mathcal S_1$, $1\leq j\leq k$, $F'=(N_1',\dots,N_{k'}',n')\in\mathcal S_2$, $1\leq j'\leq k'$,  and $\ell,m,r\in\mathbb N$ we have \begin{itemize}
\item If $N_j=M_1$, then $k\geq 2$,
$\theta_\ell^{1,F}(M_1)\nprec_{\mathcal M_F}\mathbb M_n(\mathbb C)\overline{\otimes}N_1\overline{\otimes}\cdots\overline{\otimes}N_{j-1}\overline{\otimes}1\overline{\otimes}N_{j+1}\overline{\otimes}\cdots\overline{\otimes}N_k$ and $\theta_\ell^{1,F}(M_1)\nprec_{\mathcal M_F}1\overline{\otimes}\cdots\overline{\otimes}1\overline{\otimes}N_j\overline{\otimes}1\cdots\overline{\otimes}1=1\overline{\otimes}\cdots\overline{\otimes}1\overline{\otimes}M_1\overline{\otimes}1\cdots\overline{\otimes}1$.
\item If $N_j\in\{M_1^{\text{op}},M_2,M_2^{\text{op}}\}$, then $\theta_\ell^{1,F}(M_1)\nprec_{\mathcal M_F}\mathbb M_n(\mathbb C)\overline{\otimes}N_1\overline{\otimes}\cdots\overline{\otimes}N_{j-1}\overline{\otimes}1\overline{\otimes}N_{j+1}\overline{\otimes}\cdots\overline{\otimes}N_k$.
\item $\theta_\ell^{2,F'}(M_2)\nprec_{\mathcal M_{F'}}\mathbb M_{n'}(\mathbb C)\overline{\otimes}N_1'\overline{\otimes}\cdots\overline{\otimes}N'_{j'-1}\overline{\otimes}1\overline{\otimes}N_{j'+1}'\overline{\otimes}\cdots\overline{\otimes}N_{k'}'$.
\item  $\rho_m^{1,r}:M_1\rightarrow\mathbb M_r(\mathbb C)\overline{\otimes}M_1$ is a properly outer embedding.
\end{itemize}

Since the family of $*$-homomorphisms  $\{\theta_\ell^{i,F}\mid\ell\in\mathbb N,1\leq i\leq 2, F\in\mathcal S_i\}\cup\{\rho_m^{1,r}\mid m,r\in\mathbb N\}$ is countable, applying Theorem \ref{homomorphisms} gives a hyperfinite subfactor $R_1\subset M_1$ such that for every $F=(N_1,\dots,N_k,n)\in\mathcal S_1$,  $1\leq j\leq k$, $F'=(N_1',\dots,N_{k'}',n')\in\mathcal S_2$, $1\leq j'\leq k'$, and $\ell,m,r\in\mathbb N$ we have
\begin{enumerate}
\item[(1a)] $\theta_\ell^{1,F}(R_1)\nprec_{\mathcal M_F}\mathbb M_n(\mathbb C)\overline{\otimes}N_1\overline{\otimes}\cdots\overline{\otimes}N_{j-1}\overline{\otimes}f(N_j)\overline{\otimes}N_{j+1}\overline{\otimes}\cdots\overline{\otimes}N_k$, if $N_j\in\{M_1,M_1^{\text{op}}\}$.
\item[(1b)] $\theta_\ell^{1,F}(R_1)\nprec_{\mathcal M_F}\mathbb M_n(\mathbb C)\overline{\otimes}N_1\overline{\otimes}\cdots\overline{\otimes}N_{j-1}\overline{\otimes}1\overline{\otimes}N_{j+1}\overline{\otimes}\cdots\overline{\otimes}N_k$, if $N_j\in\{M_2,M_2^{\text{op}}\}$.
\item[(1c)]  $\theta_\ell^{2,F'}(M_2)\nprec_{\mathcal M_{F'}}\mathbb M_{n'}(\mathbb C)\overline{\otimes}N_1'\overline{\otimes}\cdots\overline{\otimes}N_{j'-1}'\overline{\otimes}f(N_{j'}')\overline{\otimes}N_{j'+1}'\overline{\otimes}\cdots\overline{\otimes}N_{k'}'$ if $N_{j'}'\in\{M_1,M_1^{\text{op}}\}$.
\item[(1d)] $\rho_m^{1,r}(R_1)\nprec_{\mathbb M_r(\mathbb C)\overline{\otimes}M_1}\mathbb M_r(\mathbb C)\overline{\otimes}R_1$.
\end{enumerate}

Next, for every $F=(N_1,\dots,N_k,n)\in\mathcal S_1$, $1\leq j\leq k$,  $F'=(N_1',\dots,N_{k'}',n')\in\mathcal S_2$, $1\leq j'\leq k'$, 
 $\ell,m,r,s,s'\in\mathbb N$ and $G\in\mathcal S$ we have
\begin{itemize}
\item If $N_j\in\{M_2,M_2^{\text{op}}\}$, then $\theta_\ell^{1,F}(R_1)\nprec_{\mathcal M_F}\mathbb M_n(\mathbb C)\overline{\otimes}N_1\overline{\otimes}\cdots\overline{\otimes}N_{j-1}\overline{\otimes}1\overline{\otimes}N_{j+1}\overline{\otimes}\cdots\overline{\otimes}N_k$.
\item If $N_{j'}'=M_2$, then $k'\geq 2$,
$\theta_\ell^{2,F'}(M_2)\nprec_{\mathcal M_{F'}}\mathbb M_{n'}(\mathbb C)\overline{\otimes}N_1'\overline{\otimes}\cdots\overline{\otimes}N_{j'-1}'\overline{\otimes}1\overline{\otimes}N_{j'+1}'\overline{\otimes}\cdots\overline{\otimes}N_{k'}'$ and $\theta_\ell^{2,F'}(M_2)\nprec_{\mathcal M_{F'}}1\overline{\otimes}\cdots\overline{\otimes}1\overline{\otimes}N_{j'}'\overline{\otimes}1\cdots\overline{\otimes}1=1\overline{\otimes}\cdots\overline{\otimes}1\overline{\otimes}M_2\overline{\otimes}1\cdots\overline{\otimes}1$.
\item If $N_{j'}'=M_2^{\text{op}}$, then $\theta_\ell^{2,F'}(M_2)\nprec_{\mathcal M_{F'}}\mathbb M_{n'}(\mathbb C)\overline{\otimes}N_1'\overline{\otimes}\cdots\overline{\otimes}N_{j'-1}'\overline{\otimes}1\overline{\otimes}N_{j'+1}'\overline{\otimes}\cdots\overline{\otimes}N_{k'}'$.
\item If $N_{j'}'\in\{M_1,M_1^{\text{op}}\}$, then $\theta_\ell^{2,F'}(M_2)\nprec_{\mathcal M_{F'}}\mathbb M_{n'}(\mathbb C)\overline{\otimes}N_1'\overline{\otimes}\cdots\overline{\otimes}N_{j'-1}'\overline{\otimes}f(N_{j'}')\overline{\otimes}N_{j'+1}'\overline{\otimes}\cdots\overline{\otimes}N_{k'}'$.
\item  $\rho_m^{2,r}:M_2\rightarrow\mathbb M_r(\mathbb C)\overline{\otimes}M_2$ is a properly outer embedding.
\item $\sigma_{s'}^{2,G}(M_2)\nprec_{\mathcal M_G}\sigma_s^{1,G}(R_1)$.
\end{itemize}
 Theorem \ref{homomorphisms} gives a hyperfinite subfactor $R_2\subset M_2$ such that for every 
 $F=(N_1,\dots,N_k,n)\in\mathcal S_1$, $1\leq j\leq k$, $F'=(N_1',\dots,N_{k'}',n')\in\mathcal S_2$, $1\leq j'\leq k'$,  $\ell,m,r,s,s'\in\mathbb N$ and $G\in\mathcal S$ we have

\begin{enumerate}
\item[(2a)] $\theta_\ell^{1,F}(R_1)\nprec_{\mathcal M_F}\mathbb M_n(\mathbb C)\overline{\otimes}N_1\overline{\otimes}\cdots\overline{\otimes}N_{j-1}\overline{\otimes}f(N_j)\overline{\otimes}N_{j+1}\overline{\otimes}\cdots\overline{\otimes}N_k$, if $N_j\in\{M_2,M_2^{\text{op}}\}$.
\item[(2b)]  $\theta_\ell^{2,F'}(R_2)\nprec_{\mathcal M_{F'}}\mathbb M_{n'}(\mathbb C)\overline{\otimes}N_1'\overline{\otimes}\cdots\overline{\otimes}N_{j'-1}'\overline{\otimes}f(N_{j'}')\overline{\otimes}N_{j'+1}'\overline{\otimes}\cdots\overline{\otimes}N_{k'}'$, if $N_{j'}'\in\{M_2,M_2^{\text{op}}\}$.
\item[(2c)] $\theta_\ell^{2,F'}(R_2)\nprec_{\mathcal M_{F'}}\mathbb M_{n'}(\mathbb C)\overline{\otimes}N_1'\overline{\otimes}\cdots\overline{\otimes}N_{j'-1}'\overline{\otimes}f(N_{j'}')\overline{\otimes}N_{j'+1}'\overline{\otimes}\cdots\overline{\otimes}N_{k'}'$, if $N_{j'}'\in\{M_1,M_1^{\text{op}}\}$.
\item[(2d)] $\rho_m^{2,r}(R_2)\nprec_{\mathbb M_r(\mathbb C)\overline{\otimes}M_2}\mathbb M_r(\mathbb C)\overline{\otimes}R_2$.
\item[(2e)] $\sigma_{s'}^{2,G}(R_2)\nprec_{\mathcal M_G}\sigma_s^{1,G}(R_1)$ and $\sigma_s^{1,G}(R_1)\nprec_{\mathcal M_G}\sigma_{s'}^{2,G}(R_2)$.
\end{enumerate}

The combination of (1a), (1d) and (2a)-(2e) implies that  for every $i\in\{1,2\}$, $F=(N_1,\dots,N_k,n)\in\mathcal S_i$, $\ell,m,r,s,s'\in\mathbb N$ and $G\in\mathcal S$ we have \begin{itemize}
\item[(iv)] $\theta_\ell^{i,F}(R_i)\nprec_{\mathcal M_F}\mathbb M_n(\mathbb C)\overline{\otimes}N_1\overline{\otimes}\cdots\overline{\otimes}N_{j-1}\overline{\otimes}f(N_j)\overline{\otimes}N_{j+1}\overline{\otimes}\cdots\overline{\otimes}N_k$, for every $1\leq j\leq k$.

\item[(v)] $\rho_m^{i,r}(R_i)\nprec_{\mathbb M_r(\mathbb C)\overline{\otimes}M_i}\mathbb M_r(\mathbb C)\overline{\otimes}R_i$.
\item[(vi)] $\sigma_{s'}^{2,G}(R_2)\nprec_{\mathcal M_G}\sigma_s^{1,G}(R_1)$ and $\sigma_s^{1,G}(R_1)\nprec_{\mathcal M_G}\sigma_{s'}^{2,G}(R_2)$.
\end{itemize}
Let $1\leq i\leq 2,F=(N_1,\dots,N_k,n)\in\mathcal S_i$, $r\geq 1$ and $G\in\mathcal S$.
By  combining (i)-(iii) with (iv)-(vi),  Lemma \ref{dense_inter}
implies that
\begin{itemize}
\item $\theta(R_i)\nprec_{\mathcal M_F}\mathbb M_n(\mathbb C)\overline{\otimes}N_1\overline{\otimes}\cdots\overline{\otimes}N_{j-1}\overline{\otimes}f(N_j)\overline{\otimes}N_{j+1}\overline{\otimes}\cdots\overline{\otimes}N_k$, for every $1\leq j\leq k$ and every non-degenerate, not necessarily unital, embedding $\theta:M_i\rightarrow\mathcal M_F$.
\item $\rho(R_i)\nprec_{\mathbb M_r(\mathbb C)\overline{\otimes}M_i}\mathbb M_r(\mathbb C)\overline{\otimes}R_i$, for every properly outer, not necessarily unital, embedding $\rho:M_i\rightarrow\mathbb M_r(\mathbb C)\overline{\otimes}M_i$.
\item $\sigma_1(R_1)\nprec_{\mathcal M_G}\sigma_2(R_2)$ and $\sigma_2(R_2)\nprec_{\mathcal M_G}\sigma_1(R_1)$, for any not necessarily unital embeddings $\sigma_1:M_1\rightarrow\mathcal M_G$ and $\sigma_2:M_2\rightarrow\mathcal M_G$.
\end{itemize}
This finishes the proof.
\end{proof}

We are now ready to prove Theorem \ref{embeddings}.
\begin{proof}[\bf Proof of Theorem \ref{embeddings}, assertion (1)]
Let $R_1\subset M_1$ and $R_2\subset M_2$ be coarse copies of the hyperfinite II$_1$ factor, $R$, given by Theorem \ref{embedd2}. 
Identify $R_1=R_2$ via any isomorphism $R_1\cong R_2$, 
denote $R=R_1=R_2$ and define $M=M_1*_{R}M_2$. 
Let $\theta:M\rightarrow\mathcal M$ be any embedding.

We continue by introducing the following notation which we use throughout the proof:
\begin{itemize}

    \item We write $\mathcal M=M^{{\mathbb N}}\overline{\otimes}(M^{\text{op}})^{{\mathbb Z_{-}}}$, where $\mathbb Z_{-}=\mathbb Z\setminus\mathbb N$ is the set of negative integers.
    \item For $n\in\mathbb N$, denote by $\Delta_N:M\rightarrow\mathcal M$ the embedding of $M$ into the $n^{\text{th}}$ copy of $M$.
   \item  For $k\in \mathbb Z$, denote by $\mathcal M_k$ the tensor product obtained from $\mathcal M$ by replacing the $k^{\text{th}}$ tensor factor $M$ with $R$ if $k\in\mathbb N$, and replacing the $k^{\text{th}}$ tensor factor $M^{\mathrm{op}}$ with $R^{\mathrm{op}}$ if $k\in\mathbb Z_-$.
   
    \item Let $\mathcal S$ be the collection of $4$-tuples $F=(F_1,F_2,F_3,F_4)$ of finite 
    sets such that $F_1,F_2\subset\mathbb N$, $F_3,F_4\subset\mathbb Z_{-}$ and $F_1\cap F_2=F_3\cap F_4=\emptyset$.  For $F=(F_1,F_2,F_3,F_4),G=(G_1,G_2,G_3,G_4)\in\mathcal S$, we write $F\leq G$ if $F_j\subset G_j$, for every $1\leq j\leq 4$. We also write $F<G$ if $F\leq G$ and $F\not=G$.
    
  \item   For $F=(F_1,F_2,F_3,F_4)\in\mathcal S$, we define a von Neumann subalgebra $\mathcal M_F\subset\mathcal M$ by letting $$\mathcal M_F=M_1^{F_1}\overline{\otimes}M_2^{F_2}\overline{\otimes}(M_1^{\text{op}})^{F_3}\overline{\otimes}(M_2^{\text{op}})^{F_4}.$$ 
   \item We define $f(M_1)=R_1,f(M_1^{\text{op}})=R_1^{\text{op}},f(M_2)=R_2$ and $f(M_2^{\text{op}})=R_2^{\text{op}}$.
  \item  
  For $F=(F_1,F_2,F_3,F_4)\in\mathcal{S}$, put $\widetilde{F}=F_1\cup F_2\cup F_3\cup F_4$. For $k\in\widetilde{F}$, let $\mathcal{M}_{F,k}$ be the tensor product obtained from $\mathcal{M}_F$ by replacing the tensor factor $N$ in the $k^{\text{th}}$ position with $f(N)$.

\end{itemize}

 Since $R$ is coarse in both $M_1$ and $M_2$, it is also coarse in $M$. In particular, we have that $R'\cap M=\mathbb C1$. 
Let $i\in\{1,2\}$ and $n\in\mathbb N$. 
Since $M_i'\cap M=\mathbb C1$, we have $\Delta_n(M_i)'\cap \mathcal M=M^{\mathbb N\setminus\{n\}}\overline{\otimes}(M^{\text{op}})^{{\mathbb Z_{-}}}$,  hence $\mathcal Z(\Delta_n(M_i)'\cap \mathcal M)=\mathbb C1$.
By Lemma \ref{intertwiners} we can find a projection $p_{i,n}\in\mathcal Z(\theta(M_i)'\cap \mathcal M)$ and a partial isometry $v_{i,n}\in\mathcal M$ such that $v_{i,n}v_{i,n}^*=p_{i,n}$ and $\theta(x)v_{i,n}=v_{i,n}\Delta_n(x)$, for every $x\in M_i$, and if $w\in\mathcal M$ satisfies $\theta(x)w=w\Delta_n(x)$, for every $x\in M_i$, then $(1-p_{i,n})w=0$.

Let $i\in\{1,2\}$. If $m,n\in\mathbb N$ and $m\not=n$,  then $$\text{$v_{i,m}^*v_{i,n}\Delta_n(x)=v_{i,m}^*\theta(x)v_{i,n}=\Delta_m(x)v_{i,m}^*v_{i,n}$, for every $x\in M_i$.}$$ Since $M_i$ is diffuse being a II$_1$ factor, we get that $v_{i,m}^*v_{i,n}=0$ and thus $p_{i,m}p_{i,n}=0$. Therefore, $\{p_{i,n}\}_{n\in\mathbb N}$ are pairwise orthogonal projections. Hence, $p_i=\sum_{n\in\mathbb N}p_{i,n}\in\mathcal Z(\theta(M_i)'\cap \mathcal M)$ is a projection such that if $w\in\mathcal M$ and $n\in\mathbb N$ satisfy $\theta(x)w=w\Delta_n(x)$, for every $x\in M_i$, then $(1-p_{i})w=0$. We define $q_i=1-p_i\in\mathcal Z(\theta(M_i)'\cap\mathcal M)$.

The rest of the proof relies on four claims, Claims \eqref{max_proj}-\eqref{b}.

\begin{claim}\label{max_proj} Let $i\in\{1,2\}$. Then there exist projections   $(q_{i,F})_{F\in\mathcal S}\subset\mathcal Z(\theta(M_i)'\cap \mathcal M)q_i$ such that 
\begin{enumerate}
\item $\theta(M_i)q_{i,F}\prec^{{s}}_{\mathcal M}\mathcal M_F$ and $\theta(M_i)q_{i,F}\nprec_{\mathcal M}\mathcal M_G$, for every $F,G\in\mathcal S$ with $G<F$. 
\item $\sum_{F\in\mathcal S}q_{i,F}=q_i$.
\end{enumerate}
\end{claim}

\begin{proof}[Proof of Claim \ref{max_proj}] 
Let $i\in\{1,2\}$. If $F\in\mathcal S$, then Lemma \ref{elementary_facts}(3) implies that there exists a projection
 $r_{i,F}\in\mathcal Z(\theta(M_i)'\cap \mathcal M)q_i$ such that $\theta(M_i)r_{i,F}\prec^{s}_{\mathcal M}\mathcal M_F$ and $\theta(M_i)(q_i-r_{i,F})\nprec_{\mathcal M}\mathcal M_F$. Since $M_i$ has property (T), Theorem \ref{amalgamated} and Lemma \ref{elementary_facts}(3) readily imply that $\bigvee_{F\in\mathcal S}r_{i,F}=q_i$. It is clear that $r_{i,F}\leq r_{i,G}$ for every $F,G\in\mathcal S$ with $F\leq G$.  Then defining $q_{i,F}=r_{i,F}-\bigvee_{G<F}r_{i,G}$, for every $F\in\mathcal S$, the conclusion holds.
\end{proof}

\begin{claim}\label{non-inter}
Let $i\in\{1,2\}$ and $F=(F_1,F_2,F_3,F_4)\in\mathcal S$. 
Then $\theta(R)q_{i,F}\nprec_{\mathcal M}\mathcal M_{F,k}$, for every $k\in \widetilde F$.
\end{claim}

\begin{proof}[Proof of Claim \ref{non-inter}]
Assume by contradiction that $\theta(R)q_{i,F}\prec_{\mathcal M}\mathcal M_{F,k}$, for $k\in\widetilde F$. By Lemma \ref{elementary_facts}(3) we can find a nonzero projection $r\in\theta(R)'\cap \mathcal M$ such that $r\leq q_{i,F}$ and $\theta(R)r\prec_{\mathcal M}^s\mathcal M_{F,k}$. Since $\theta(M_i)q_{i,F}\prec_{\mathcal M}\mathcal M_F$ there exist $\ell\in\mathbb N$, a $*$-homomorphism $\varphi:\theta(M_i)q_{i,F}\rightarrow p(\mathbb M_\ell(\mathbb C)\overline{\otimes}\mathcal M_F)p$ and a nonzero partial isometry $v\in p(\mathbb M_{\ell,1}(\mathbb C)\overline{\otimes}\mathcal M)q_{i,F}$ such that $\varphi(x)v=vx$, for every $x\in \theta(M_i)q_{i,F}$. 

Since $\theta(M_i)q_{i,F}\prec_{\mathcal M}^{s}\mathcal M_F$, by Lemma \ref{elementary_facts}(1) we may take $v^*v$ arbitrarily close to $q_{i,F}$. Hence, we may assume that  $vr\not=0$. 
Moreover, if $e\in\varphi(\theta(M_i)q_{i,F})'\cap p(\mathbb M_\ell(\mathbb C)\overline{\otimes}\mathcal M_F)p$ is the support projection of $\text{E}_{p(\mathbb M_\ell(\mathbb C)\overline{\otimes}\mathcal M_F)p}(vv^*)$, then after replacing $\varphi(\cdot)$ with $\varphi(\cdot)e$, we may assume that $p=e$. Since  $\theta(M_i)q_{i,F}\nprec_{\mathcal M}\mathcal M_G$, it follows that $\varphi(\theta(M_i)q_{i,F})\nprec_{\mathbb M_\ell(\mathbb C)\overline{\otimes}\mathcal M_F}\mathbb M_\ell(\mathbb C)\overline{\otimes}\mathcal M_G$, for every $G\in\mathcal S$ with $G<F$ (see \cite[Remark 3.8]{Vae07} or the proof of Lemma \ref{elementary_facts}(2)). 

Define an embedding $\delta:M_i\rightarrow p(\mathbb M_\ell(\mathbb C)\overline{\otimes}\mathcal M_F)p$ by letting $\delta(x)=\varphi(\theta(x)q_{i,F})$.
Then we have \begin{equation}\label{star}\text{$\delta(M_i)\nprec_{\mathbb M_\ell(\mathbb C)\overline{\otimes}\mathcal M_F}\mathbb M_\ell(\mathbb C)\overline{\otimes}\mathcal M_G$,\;\; for every $G\in\mathcal S$ with $G<F$.}\end{equation}

Moreover, assume that $|F_i|=1$ and $F_j=\emptyset$, for every $j\not=i$. In other words, $\widetilde F=F_i=\{k\}$, for some $k\in \mathbb N$, and $\mathcal M_F=M_i^{\{k\}}$ can be identified with $M_i$.
In this case, we claim that 
\begin{equation}\label{starstar}\text{$\delta:M_i\rightarrow p(\mathbb M_\ell(\mathbb C)\overline{\otimes}M_i)p$ \;\; is properly outer.}\end{equation} To this end, let $w\in p(\mathbb M_{\ell,1}(\mathbb C)\overline{\otimes}M_i)$ such that $\delta(x)w=w(1\otimes x)$, for every $x\in M_i$.  Then $v^*w\in q_{1,F}\mathcal M$ and for every $x\in M_i$ we have
\begin{align*}\theta(x)v^*w=\theta(x)q_{i,F}v^*w&=(v\theta(x^*)q_{i,F})^*w\\&=(\varphi(\theta(x^*)q_{i,F})v)^*w=(\delta(x^*)v)^*w\\&=v^*\delta(x)w=(v^*w)(1\otimes x)=(v^*w)\Delta_k(x).\end{align*}

The  definition of $p_{i,k}$ implies that $(1-p_{i,k})(v^*w)=0$. Since $q_{i,F}\leq 1-p_{i,k}$ and $vq_{i,F}=v$, we conclude that $v^*w=q_{i,F}(v^*w)=0$. 
Hence, $\text{E}_{p(\mathbb M_\ell(\mathbb C)\overline{\otimes}\mathcal M_F)p}(vv^*)w=\text{E}_{p(\mathbb M_\ell(\mathbb C)\overline{\otimes}\mathcal M_F)p}(vv^*w)=0$. Since the support projection of $\text{E}_{p(\mathbb M_\ell(\mathbb C)\overline{\otimes}\mathcal M_F)p}(vv^*)$ is equal to $p$, we further get that $pw=0$ and so $w=0$. This proves that $\delta$ is properly outer, as claimed in \eqref{starstar}. 

Next, \eqref{star} and \eqref{starstar} say that  $\delta:M_i\rightarrow p(\mathbb M_\ell(\mathbb C)\overline{\otimes}\mathcal M_F)p$ satisfies the hypothesis of Theorem \ref{embedd2}(1), which gives that
$\delta(R)=\delta(R_i)\nprec_{\mathbb M_\ell(\mathbb C)\overline{\otimes}\mathcal M_F}\mathbb M_\ell(\mathbb C)\overline{\otimes}\mathcal M_{F,k}$, for every $k\in \widetilde F$.  We claim that
\begin{equation}\label{delta}\text{$\delta(R)\nprec_{\mathbb M_\ell(\mathbb C)\overline{\otimes}\mathcal M}\mathbb M_\ell(\mathbb C)\overline{\otimes}\mathcal M_{F,k}$, for every $k\in\widetilde F$.}\end{equation}
 To prove \eqref{delta}, let $k\in\widetilde F$. 
Write $\mathcal M_F=\overline{\otimes}_{k'\in \widetilde F}N_{k'}^{\{k'\}}$, where $(N_{k'})_{k'\in\widetilde F}\subset\{M_1,M_1^{\text{op}},M_2,M_2^{\text{op}}\}$. 
If $k\in F_1\cup F_2\subset\mathbb N$, then $\mathbb M_\ell(\mathbb C)\overline{\otimes}\mathcal M=(\mathbb M_\ell(\mathbb C)\overline{\otimes}M^{\{k\}})\overline{\otimes}\mathcal N$, where 
$\mathcal N=M^{{\mathbb N\setminus\{k\}}}\overline{\otimes}(M^{\text{op}})^{{\mathbb Z_{-}}}$.
If $k\in F_3\cup F_4\subset\mathbb Z_{-}$, then $\mathbb M_\ell(\mathbb C)\overline{\otimes}\mathcal M=(\mathbb M_\ell(\mathbb C)\overline{\otimes}(M^{\text{op}})^{\{k\}})\overline{\otimes}\mathcal N$, where 
$\mathcal N=M^{{\mathbb N}}\overline{\otimes}(M^{\text{op}})^{\mathbb Z_{-}\setminus\{k\}}$.
Also, we have
 $\mathbb M_\ell(\mathbb C)\overline{\otimes}\mathcal M_F=(\mathbb M_\ell(\mathbb C)\overline{\otimes}N_k^{\{k\}})\overline{\otimes}\mathcal N_0$ and $\mathbb M_\ell(\mathbb C)\overline{\otimes}\mathcal M_{F,k}=(\mathbb M_\ell(\mathbb C)\overline{\otimes}f(N_k)^{\{k\}})\overline{\otimes}\mathcal N_0$, where
$\mathcal N_0\subset\mathcal N$ is given by $\mathcal N_0=\overline{\otimes}_{k'\in\widetilde F\setminus\{k\}}N_{k'}^{\{k'\}}$. 

Note that $M=N_k*_{f(N_k)}P$, for some $P\in\{M_1,M_1^{\text{op}},M_2,M_2^{\text{op}}\}$,  and that $\delta(R)\subset p(\mathbb M_\ell(\mathbb C)\overline{\otimes}\mathcal M_F)p$ satisfies  $\delta(R)\nprec_{\mathbb M_\ell(\mathbb C)\overline{\otimes}\mathcal M_F}\mathbb M_\ell(\mathbb C)\overline{\otimes}\mathcal M_{F,k}$ by \eqref{delta}. By using the previous paragraph, Lemma \ref{tensor_products}(2) gives that $\delta(R)\nprec_{\mathbb M_\ell(\mathbb C)\overline{\otimes}\mathcal M}\mathbb M_\ell(\mathbb C)\overline{\otimes}\mathcal M_{F,k}$. This finishes the proof of \eqref{delta}.

Finally, since $\delta(R)=\varphi(\theta(R)q_{i,F})$, we get $\delta(R)vv^*=\varphi(\theta(R)q_{i,F})vv^*=v\theta(R)q_{i,F}v^*=v\theta(R)v^*$. Hence, \eqref{delta} implies that 
$\theta(R)v^*v\nprec_{\mathcal M}\mathcal M_{F,k}$, for every $k\in\widetilde F$.
However, this contradicts the fact that $vr\not=0$ and $\theta(R)r\prec_{\mathcal M}^s\mathcal M_{F,k}$.
\end{proof}

\begin{claim}\label{a}
For every $n\in\mathbb N$ and $F,F'\in\mathcal S$, $F\not=F'$, we have $q_{1,F}p_{2,n}=q_{2,F}p_{1,n}=q_{1,F}q_{2,F'}=0$.
\end{claim}

\begin{proof}[Proof of Claim \ref{a}]
Let $n\in\mathbb N$ and $F=(F_1,F_2,F_3,F_4)\in\mathcal S$. 
We first prove that $q_{1,F}p_{2,n}=0$. Assume by contradiction that $q_{1,F}p_{2,n}\not=0$. Note that $q_{1,F}\in\theta(R)'\cap\mathcal M$ and let $z\in\mathcal Z(\theta(R)'\cap\mathcal M)$ be the central support of $q_{1,F}$ in $\theta(R)'\cap\mathcal M$.  Since $\theta(M_1)q_{1,F}\prec^s_{\mathcal M}\mathcal M_F$, we get $\theta(R)q_{1,F}\prec^{{s}}_{\mathcal M}\mathcal M_F$ and Lemma \ref{elementary_facts}(3) gives that $\theta(R)z\prec^{{s}}_{\mathcal M}\mathcal M_F$.
 Since $\theta(x)v_{2,n}=v_{2,n}\Delta_n(x)$ we deduce that \begin{equation}\label{v_{2,n}}\text{$\theta(x)(zv_{2,n})=(zv_{2,n})\Delta_n(x)$, for every $x\in R$.}\end{equation} As $q_{1,F}p_{2,n}\not=0$, we have $zp_{2,n}\not=0$. Since $v_{2,n}v_{2,n}^*=p_{2,n}\in\theta(R)'\cap\mathcal M$ and $z\in\mathcal Z(\theta(R)'\cap\mathcal M)$, we get that $\zeta=zv_{2,n}$ is a nonzero partial isometry. Put $f=\zeta\zeta^*\in (\theta(R)'\cap \mathcal M)z$ and $g=\zeta^*\zeta\in\Delta_n(R)'\cap\mathcal M$. Then \eqref{v_{2,n}} implies that  $\theta(R)f=\zeta(\Delta_n(R)g)\zeta^*$. Since $f\leq z$ and $\theta(R)z\prec^{{s}}_{\mathcal M}\mathcal M_F$, we get that $\theta(R)f\prec_{\mathcal M}\mathcal M_F$, hence $\Delta_n(R)g\prec_{\mathcal M}\mathcal M_F$. Since $\mathcal M_F\subset M^{F_1\cup F_2}\overline{\otimes}(M^{\text{op}})^{F_3\cup F_4}$, it is immediate that $n\in F_1\cup F_2$ and therefore $\Delta_n(R)\subset\mathcal M_{F,n}$. Using that $\theta(R)f=\zeta(\Delta_n(R)g)\zeta^*$, we derive that $\theta(R)f\prec_{\mathcal M}\mathcal M_{F,n}$. 
 Since $f\leq z$, it follows that $\theta(R)z\prec_{\mathcal M}\mathcal M_{F,n}$. 
 Since $z$ is the central support of $q_{1,F}$ in $\theta(R)'\cap\mathcal M$, we get that $\theta(R)q_{1,F}\prec_{\mathcal M}\mathcal M_{F,n}$, which contradicts Claim \ref{non-inter}.

By symmetry, it also follows that $q_{2,F}p_{1,n}=0$. 

Thus, it remains to prove that $q_{1,F}q_{2,F'}=0$, for every $F,F'\in\mathcal S$ with $F\not=F'$. 
Assume by contradiction that $q_{1,F}q_{2,F'}\not=0$, for some $F=(F_1,F_2,F_3,F_4),F'=(F_1',F_2',F_3',F_4')\in\mathcal S$ with $F\not=F'$.  
Let $i\in\{1,2,3,4\}$ such that $F_i\not=F_i'$. Without loss of generality, we may assume that $F_i\setminus F_i'\not=\emptyset$. Let $k\in F_i\setminus F_i'$.
Let $z_1,z_2\in\mathcal Z(\theta(R)'\cap\mathcal M)$ be the central supports of $q_{1,F}$ and $q_{2,F'}$ in $\theta(R)'\cap\mathcal M$, respectively. By using Lemma \ref{elementary_facts}(3) and reasoning as above we get that $\theta(R)z_1\prec_{\mathcal M}^s\mathcal M_F$ and $\theta(R)z_2\prec_{\mathcal M}^s\mathcal M_{F'}$.
Then $\eta=z_1z_2\not=0$, $\theta(R)\eta\prec_{\mathcal M}^s\mathcal M_{F}$ and $\theta(R)\eta\prec_{\mathcal M}^s\mathcal M_{F'}$. 

By using this fact, Lemma \ref{tensor_products}(3) and that $M=M_1*_RM_2$, we conclude that $\theta(R)\eta\prec_{\mathcal M}^s\mathcal M_{F,k}$.
Since $\eta\leq z_1$ we get that $\theta(R)z_1\prec_{\mathcal M}\mathcal M_{F,k}$ and further that $\theta(R)q_{1,F}\prec_{\mathcal M}\mathcal M_{F,k}$.
This, however, contradicts Claim \ref{non-inter}.
\end{proof}

If $F\in\mathcal S$, then $q_{1,F}p_{2,n}=q_{1,F}q_{2,F'}$, for every $n\in\mathbb N$ and $F'\in\mathcal S\setminus\{F\}$ by the previous claim. Since $(\sum_{n\in\mathbb N}p_{2,n})+(\sum_{F'\in\mathcal S}q_{2,F'})=1$, we deduce that $q_{1,F}\leq q_{2,F}$. By symmetry, we also get that $q_{2,F}\leq q_{1,F}$ and so we conclude that $q_{1,F}=q_{2,F}$, for every $F\in\mathcal S$.

\begin{claim}\label{b}
For every $F\in\mathcal S$, we have $q_{1,F}=q_{2,F}=0$.
\end{claim}

\begin{proof}[Proof of Claim \ref{b}]
Assume by contradiction that $q:=q_{1,F}=q_{2,F}\not=0$, for some $F\in\mathcal S$. Since $\theta(M_1)q\prec_{\mathcal M}\mathcal M_F$ by Claim \ref{max_proj}, we can find $\ell_1\in\mathbb N$, a projection $p_1\in\mathbb M_{\ell_1}(\mathbb C)\overline{\otimes}\mathcal M_F$, a nonzero partial isometry $v_1\in p_1(\mathbb M_{\ell_1}(\mathbb C)\overline{\otimes}\mathcal M)q$ and a $*$-homomorphism $\varphi_1:\theta(M_1)q\rightarrow p_1(\mathbb M_{\ell_1}(\mathbb C)\overline{\otimes}\mathcal M_F)p_1$ such that $\varphi_1(x)v_1=v_1x$, for every $x\in\theta(M_1)q$. Since by Claim \ref{non-inter} we have $\theta(R)q\nprec_{\mathcal M}\mathcal M_{F,k}$, Lemma \ref{elementary_facts}(2) gives that we may also assume that
\begin{equation}\label{F,k1}\text{$\varphi_1(\theta(R)q)\nprec_{\mathbb M_{\ell_1}(\mathbb C)\overline{\otimes}\mathcal M_F}\mathbb M_{\ell_1}(\mathbb C)\overline{\otimes}\mathcal M_{F,k}$, for every $k\in\widetilde F$.}
\end{equation}

By using that $M=M_1*_RM_2$ and \eqref{F,k1}, Lemma \ref{tensor_products}(2) implies that
\begin{equation}\label{F,k2}\text{$\varphi_1(\theta(R)q)\nprec_{\mathbb M_{\ell_1}(\mathbb C)\overline{\otimes}\mathcal M}\mathbb M_{\ell_1}(\mathbb C)\overline{\otimes}\mathcal M_{k}$, for every $k\in\widetilde F$.}
\end{equation}
Indeed, Lemma \ref{tensor_products}(2) applies since there exist tracial von Neumann algebras  $\mathcal N_0\subset\mathcal N$ such that either (1)  $k\in F_1\cup F_2$, $\mathcal M=M\overline{\otimes}\mathcal N$, $\mathcal M_k=R\overline{\otimes}\mathcal N$ and $\mathcal M_{F,k}=R\overline{\otimes}\mathcal N_0$, or (2) $k\in F_3\cup F_4$, $\mathcal M=M^{\text{op}}\overline{\otimes}\mathcal N$, $\mathcal M_k=R^{\text{op}}\overline{\otimes}\mathcal N$ and $\mathcal M_{F,k}=R^{\text{op}}\overline{\otimes}\mathcal N_0$.

Since $\theta(M_2)q\prec_{\mathcal M}\mathcal M_F$ by Claim \ref{max_proj}, we can find $\ell_2\in\mathbb N$, a projection $p_2\in\mathbb M_{\ell_2}(\mathbb C)\overline{\otimes}\mathcal M_F$, a nonzero partial isometry $v_2\in p_2(\mathbb M_{\ell_2}(\mathbb C)\overline{\otimes}\mathcal M)q$ and a $*$-homomorphism $\varphi_2:\theta(M_2)q\rightarrow p_2(\mathbb M_{\ell_2}(\mathbb C)\overline{\otimes}\mathcal M_F)p_2$ such that $\varphi_2(x)v_2=v_2x$, for every $x\in\theta(M_2)q$. Moreover, since  $\theta(M_2)q\prec_{\mathcal M}^s\mathcal M_F$,  Lemma \ref{elementary_facts}(1) implies that we may take  $v_2^*v_2$ to be arbitrarily close to $q$. In particular, we may assume that $v:=v_1v_2^*\in p_1(\mathbb M_{\ell_1,\ell_2}(\mathbb C)\overline{\otimes}\mathcal M)p_2$ is nonzero.

 Next, we have 
 \begin{equation}
\label{v12}\text{$\varphi_1(x)v=\varphi_1(x)v_1v_2^*=v_1xv_2^*=v_1v_2^*\varphi_2(x)=v\varphi_2(x)$,  for every $x\in\theta(R)q$}. \end{equation} 
Consequently, we get that
\begin{equation}\label{interw}
    \varphi_1(\theta(R)q)v=v\varphi_2(\theta(R)q).
\end{equation}
For $k\in \mathbb Z$ and $j\in\{1,2\}$, denote by $\mathcal M_{k,j}$ the tensor product obtained from $\mathcal M$ by replacing the $k^{\text{th}}$ tensor factor $M$ with $M_j$ if $k\in\mathbb N$, and replacing the $k^{\text{th}}$ tensor factor $M^{\mathrm{op}}$ with $M_j^{\mathrm{op}}$ if $k\in\mathbb Z_-$.

Fix $k\in\widetilde F$ and note that $\mathcal M_F\subset\mathcal M_{k,j_k}$, for some $j_k\in\{1,2\}$. Specifically, we take $j_k=1$ if $k\in F_1\cup F_3$ and  $j_k=2$ if $k\in F_2\cup F_4$. Thus, we have $\varphi_i(\theta(R)q)\subset p_i(\mathbb M_{\ell_i}(\mathbb C)\overline{\otimes}\mathcal M_F)p_i\subset p_i(\mathbb M_{\ell_i}(\mathbb C)\overline{\otimes}\mathcal M_{k,j_k})p_i$.
Additionally, we have $\mathcal M=\mathcal M_{k,1}*_{\mathcal M_k}\mathcal M_{k,2}$.
By combining the last two facts with \eqref{F,k2} and \eqref{interw}, and applying \cite[Theorem 1.1]{IPP05}, we derive that $v\in  p_1(\mathbb M_{\ell_1,\ell_2}(\mathbb C)\overline{\otimes} \mathcal M_{k,j_k})p_2$. Hence, we get that \begin{equation}\label{vv}v\in\bigcap_{k\in\widetilde F}p_1(\mathbb M_{\ell_1,\ell_2}(\mathbb C)\overline{\otimes} \mathcal M_{k,j_k})p_2=p_1(\mathbb M_{\ell_1,\ell_2}(\mathbb C)\overline{\otimes}\mathcal M_F)p_2\overline{\otimes}\mathcal P,\end{equation} where $\mathcal P=M^{\mathbb N\setminus (F_1\cup F_2)}\overline{\otimes}(M^{\text{op}})^{\mathbb Z_{-}\setminus (F_3\cup F_4)}$. Let $(\eta_m)_{m\geq 1}\subset\mathcal P$ be an orthonormal basis for $\text{L}^2(\mathcal P)$, and write $v=\sum_{m\geq 1}\xi_m\otimes\eta_m$, where $(\xi_m)_{m\geq 1}\subset p_1(\mathbb M_{\ell_1,\ell_2}(\mathbb C)\overline{\otimes}\mathcal M_F)p_2$. As $v\not=0$, there is $m\geq 1$ with $\xi_m\not=0$. 
Since $\varphi_i(\theta(R)q)\subset p_i(\mathbb M_{\ell_i}(\mathbb C)\overline{\otimes}\mathcal M_F)p_i$, for $i\in\{1,2\}$,  \eqref{interw} implies that $\varphi_1(\theta(R)q)\xi_m=\xi_m\varphi_2(\theta(R)q)$.
Since $\xi_m\in p_1(\mathbb M_{\ell_1,\ell_2}(\mathbb C)\overline{\otimes}\mathcal M_F)p_2$, this entails $\varphi_1(\theta(R)q)\prec_{\mathbb M_\ell(\mathbb C)\overline{\otimes}\mathcal M_F}\varphi_2(\theta(R)q)$. 

For $i\in\{1,2\}$, let $\alpha_i:M_i\rightarrow p_i(\mathbb M_\ell(\mathbb C)\overline{\otimes}\mathcal M_F)p_i$ be the embedding given by $\alpha_i(x)=\varphi_i(\theta(x)q)$. Then $\alpha_1(R)\prec_{\mathbb M_\ell(\mathbb C)\overline{\otimes}\mathcal M_F}\alpha_2(R)$, or, equivalently, $\alpha_1(R_1)\prec_{\mathbb M_\ell(\mathbb C)\overline{\otimes}\mathcal M_F}\alpha_2(R_2)$, contradicting the choice of $R_1$ and $R_2$.
\end{proof}

We can now finish the proof of assertion (1).
Claim \ref{b} gives that $q_{1,F}=q_{2,F}=0$, for every $F\in\mathcal S$.
Hence, we get that $\sum_{n\in\mathbb N}p_{1,n}=\sum_{n\in\mathbb N}p_{2,n}=1$.

Let $m,n\in\mathbb N$. Then for every $x\in R$, we have $v_{2,n}^*v_{1,m}\Delta_m(x)=v_{2,n}^*\theta(x)v_{1,m}=\Delta_n(x)v_{2,n}^*v_{1,m}$. Since $R$ is diffuse, it follows that if $m\not=n$, then $v_{2,n}^*v_{1,m}=0$, so $p_{2,n}p_{1,m}=v_{2,n}v_{2,n}^*v_{1,m}v_{1,m}^*=0$. This implies that $p_{1,n}=p_{2,n}$, for every $n\in\mathbb N$. 

Let $n\in\mathbb N$ and put $p_n=p_{1,n}=p_{2,n}$. Then $p_n\in\mathcal Z(\theta(M_1)'\cap\mathcal M)\cap\mathcal Z(\theta(M_2)'\cap \mathcal M)\subset\mathcal Z(\theta(M)'\cap\mathcal M)$. 
Moreover,  $v_{2,n}^*v_{1,n}\in\Delta_n(R)'\cap \mathcal M$. 
Since $R'\cap M=\mathbb C1$, we have that $\Delta_n(R)'\cap\mathcal M=\Delta_n(M)'\cap \mathcal M$. Hence, $v_{2,n}^*v_{1,n}\in\Delta_n(M)'\cap \mathcal M$. Therefore, for every $x\in M_2$, we have that 
\begin{align*}
\theta(x)v_{1,n}=\theta(x)p_nv_{1,n}=\theta(x)v_{2,n}v_{2,n}^*v_{1,n}&=v_{2,n}\Delta_n(x)(v_{2,n}^*v_{1,n})\\&=v_{2,n}(v_{2,n}^*v_{1,n})\Delta_n(x)=p_nv_{1,n}\Delta_n(x)=v_{1,n}\Delta_n(x).
\end{align*}

Since also $\theta(x)v_{1,n}=v_{1,n}\Delta_n(x)$, for every $x\in M_1$, we conclude that $\theta(x)v_{1,n}=v_{1,n}\Delta_n(x)$, for every $x\in M_1\cup M_2$. 

As $\{M_1,M_2\}''=M$, we get  $\theta(x)v_{1,n}=v_{1,n}\Delta_n(x)$, for every $x\in M$. Let $q_n=v_{1,n}^*v_{1,n}\in\Delta_n(M)'\cap \mathcal M$ and $u_n\in\mathcal M$ be any unitary extending $v_{1,n}$. Then $\theta(x)p_n=v_{1,n}\Delta_n(x)v_{1,n}^*=u_n\Delta_n(x)q_nu_n^*$, for every $x\in M$. This finishes the proof of assertion (1) of Theorem \ref{embeddings}.
\end{proof}

\begin{proof}[\bf Proof of Theorem \ref{embeddings}, assertion (2)] 
(a) The proof of assertion (1) for $M=M_1*_{R}M_2$ only uses properties of the inclusions $R\subset M_1$ and $R\subset M_2$ which are the same for $M(\rho)=M_1*_{R}M_2$.

(b) Let $\rho,\sigma\in\text{Aut}(R)$. 
We clearly have that $(i)\Rightarrow(ii)\Rightarrow (iii)$.
Assume that  $M(\rho)$ embeds into $M(\sigma)^t$, for some $t>0$. 
Then $M(\rho)$ embeds into $M(\sigma)\overline{\otimes}N\cong M(\sigma)^t\overline{\otimes}N^{1/t}$, for every II$_1$ factor $N$. In particular $M(\rho)$ embeds into $\mathcal M(\sigma):=M(\sigma)^{\mathbb N}\overline{\otimes}(M(\sigma)^{\text{op}})^{\mathbb N}$. This shows that $(iii)\Rightarrow (iv)$.
Next, assume that $\rho^{-1}\sigma\in\text{Inn}(R)$, and let $u\in\mathcal U(R)$ such that $\sigma=\text{Ad}(u)\circ\rho$.
Then there exists a $*$-isomorphism $\Phi:M(\rho)\rightarrow M(\sigma)$ such that $\Phi_{|M_1}=\text{Ad}(\alpha_1(u))$ and $\Phi_{|M_2}=\text{Id}$. To see that $\Phi$ is well defined we note that for every $x\in R$, we have that $\alpha_1(\rho(x))=\alpha_2(x)$ in $M(\rho)$, and $\Phi(\alpha_1(\rho(x)))=\text{Ad}(\alpha_1(u))(\alpha_1(\rho(x)))=\alpha_1(\text{Ad}(u)(\rho(x)))=\alpha_1(\sigma(x))=\alpha_2(x)=\Phi(\alpha_2(x))$ in $M(\sigma)$.
 This shows that $(v)\Rightarrow(i)$.

Therefore, it remains to show that $(iv)\Rightarrow(v)$. To this end, assume that $(iv)$ holds, and there exists an embedding $\theta:M(\rho)\rightarrow \mathcal M(\sigma)$. For  every $n\in\mathbb N$ and $i\in\{1,2\}$, we denote by $\Delta_{i,n}:M_i\rightarrow\mathcal M(\sigma)$ be the natural embedding of $M_i$ into the $n^{\text{th}}$ tensor copy of $M(\sigma)$ in $\mathcal M(\sigma)$.
 
Let $i\in\{1,2\}$. An immediate adaptation of the proof of (1) gives a partition of unity $(p_{i,n})_{n\in\mathbb N}$ in $\mathcal Z(\theta(M_i)'\cap\mathcal M(\sigma))$ and partial isometries $(v_{i,n})_{n\in\mathbb N}\subset \mathcal M(\sigma)$ such that we have $v_{i,n}v_{i,n}^*=p_{i,n}$ and $\theta(x)v_{i,n}=v_{i,n}\Delta_{i,n}(x)$, for every $n\in\mathbb N$ and $x\in M_i$.

Let $m,n\in\mathbb N$ such that $p_{2,m}p_{1,n}\not=0$. If $x\in R$, then we have $\theta(\alpha_1(\rho(x)))v_{1,n}=v_{1,n}\Delta_{1,n}(\alpha_1(\rho(x)))$ and $\theta(\alpha_2(x))v_{2,m}=v_{2,m}\Delta_{2,m}(\alpha_2(x))$. Since $\alpha_1(\rho(x))=\alpha_2(x)$ in $M(\rho)$, for every $x\in R$ we have
\begin{equation}\label{rho}v_{2,m}^*v_{1,n}\Delta_{1,n}(\alpha_1(\rho(x)))=v_{2,m}^*\theta(\alpha_1(\rho(x))v_{1,n}=v_{2,m}^*\theta(\alpha_2(x))v_{1,n}=\Delta_{2,m}(\alpha_2(x))v_{2,m}^*v_{1,n}.\end{equation} 
Since $v_{2,m}^*v_{1,n}\not=0$, \eqref{rho} implies that $\Delta_{1,n}(\alpha_1(R))\prec_{\mathcal M(\sigma)}\Delta_{2,m}(\alpha_2(R))$, which forces that $n=m$.
 By decomposing $\mathcal M(\sigma)=M(\sigma)^{\{n\}}\overline\otimes(M(\sigma)^{\mathbb N\setminus\{n\}}\overline{\otimes}(M^{\text{op}})^{\mathbb N})$, using again \eqref{rho} and that $v_{2,m}^*v_{1,n}\not=0$, we find $v\in M(\sigma)$ such that $v\not=0$ and the following equality holds in $M(\sigma)$
 \begin{equation}\label{rrho}
 \text{$v\alpha_1(\rho(x))=\alpha_2(x)v$, for every $x\in R$.}
 \end{equation}
On the other hand, in $M(\sigma)$ we also have that $\alpha_1(\sigma(x))=\alpha_2(x)$, for every $x\in R$. In combination with \eqref{rrho} we deduce that $v\alpha_1(\rho(x))=\alpha_1(\sigma(x))v$, for every $x\in R$. In particular, $v\alpha_1(R)=\alpha_1(R)v$. Since $\alpha_1(R)$ is coarse in both $M_1$ and $M_2$, it is also coarse in $M(\sigma)=M_1*_{\alpha_1(R)}M_2$. Hence, we deduce that $v\in\alpha_1(R)$. Letting $w\in R$ such that $v=\alpha_1(w)$, we obtain that $w\rho(x)=\sigma(x)w$, for every $x\in R$. Since $w\not=0$ and $R$ is a factor, we conclude that $\rho^{-1}\sigma$ is an inner automorphism of $R$.
  This proves that $(v)$ holds, which finishes the proof of (b).
  
  (c) Proving this assertion is similar to the proof of (b), and so we leave the details to the reader.
\end{proof}

\section{Proofs of main results}\label{sec:proofmain}
This section is devoted to the proofs of our main results. We will first derive Theorem \ref{no_comult1} from Theorem \ref{embeddings}, and then use Theorem \ref{no_comult1} to prove Theorem \ref{no_crossed1} and Corollaries \ref{C}--\ref{applicationsC}.

\subsection{Proof of Theorem \ref{no_comult1}}
Let $M_1$, $M_2$ be II$_1$ factors with property (T).
 Theorem \ref{embeddings}(1) implies the existence of coarse embeddings $R\subset M_i$, for $i=1,2$, such that  $M=M_1*_RM_2$ is a II$_1$ factor satisfying Theorem \ref{no_comult1}. Choose continuum many automorphisms $(\rho_x)_{x\in 2^{\mathbb N}}\subset\text{Aut}(R)$
 which have distinct images in
$\mathrm{Out}(R)$. Then applying Theorem \ref{embeddings}(2) gives  II$_1$ factors $(M(\rho_x))_{x\in 2^{\mathbb N}}$  which satisfy Theorem \ref{no_comult1} and are pairwise not stably embeddable. \hfill$\square$

\begin{remark}
     Theorem \ref{embeddings}(2) moreover implies that if $x\in 2^{\mathbb N}$ and $(x_n)_{n\in\mathbb N}\subset 2^{\mathbb N}$ is a sequence then $M(\rho_x)$ embeds into $\overline{\otimes}_{n\in\mathbb N}M(\rho_{x_n})$ if and only if $x\in\{x_n\mid n\in\mathbb N\}$. Define $M_S=\overline{\otimes}_{x\in S}M(\rho_x)$, for every subset $S\subset 2^{\mathbb N}$. It follows that if $S,T\subset 2^{\mathbb N}$, then $M_S$ embeds into $M_T$ if and only if $S\subset T$ (cf. non-embeddability results from \cite{PV21}).
\end{remark}

\subsection{Proofs of Theorem \ref{no_crossed1} and Corollary \ref{D}} We next establish the following result which implies Theorem \ref{no_crossed1} and Corollary  \ref{D}.

\begin{corollary}\label{DD}
The following holds for any II$_1$ factor $M$ satisfying the conclusion of Theorem \ref{no_comult1}:
\begin{enumerate}
\item $M^t\not\cong B\rtimes_{\sigma} G$, for any  $t>0$ and any trace preserving action $G\curvearrowright^\sigma(B,\tau)$ of a countably infinite group $G$ on a tracial von Neumann algebra $(B,\tau)$.
\item $M^t\not\cong B\rtimes_{\sigma,c} G$, for any  $t>0$ and any trace preserving cocycle action $(\sigma,c)$ of a countably infinite group $G$ on a tracial von Neumann algebra $(B,\tau)$ with $c(G)''$  finite dimensional.

\item If $\Delta:M\rightarrow M^t$ is any embedding, for some $t>0$, then $t\in\mathbb N$ and there exists a unitary $u\in M^t=M\overline{\otimes}\mathbb M_t(\mathbb C)$ such that $\Delta(x)=u(x\otimes 1)u^*$, for every $x\in M$.
\item If $\Delta:M\rightarrow M\overline{\otimes}M$ is any embedding, then there exist projections $p_1,p_2\in M$  and unitaries $u_1,u_2\in M\overline{\otimes}M$ such that  $\Delta(x)=u_1(x\otimes p_1)u_1^*+u_2(p_2\otimes x)u_2^*$, for every $x\in M$. 
 \item If $\Delta:M\rightarrow M\overline{\otimes}M$ is any comultiplication, then there exists a completely atomic von Neumann subalgebra $A\subset M$ such that exactly one of the following conditions holds 
 \begin{itemize}
 \item[(a)] There exists a unitary $u\in A\overline{\otimes}M$ such that $\Delta(x)=u(x\otimes 1)u^*$, for every $x\in M$.  \\ In particular, $\Delta(x)=x\otimes 1$, for every $x\in A'\cap M$.
 \item[(b)]  There exists a unitary $u\in M\overline{\otimes}A$ such that $\Delta(x)=u(1\otimes x)u^*$, for every $x\in M$. \\ In particular, $\Delta(x)=1\otimes x$, for every $x\in A'\cap M$.
 \end{itemize}
 \item  Let $I,J,K,L$ be any (possibly infinite) sets. Then $M^I\overline{\otimes}(M^{\emph{op}})^J\hookrightarrow M^K\overline{\otimes}(M^{\emph{op}})^L$ if and only if    $|I|\leq |K|$ and $|J|\leq |L|$.
In particular, $M^{\emph{op}}$ does not embed into $M^{\mathbb N}$.
\end{enumerate}
\end{corollary}

\begin{proof}
Let $M$ be any II$_1$ factor satisfying the conclusion of Theorem \ref{no_comult1}.

(1) 
Assume that $M^t=B\rtimes_\sigma G$, for some $t>0$ and a trace preserving action of a countably infinite group $G$. Let $\theta:M^t\rightarrow M^t\overline{\otimes}M^t$ be the comultiplication given by $\theta(bu_g)=bu_g\otimes u_g$, for every $b\in B$ and $g\in G$. Since $G$ is infinite, it is immediate that $\theta(M^t)\nprec_{M^t\overline{\otimes}M^t}M^t\overline{\otimes}1$ and 
$\theta(M^t)\nprec_{M^t\overline{\otimes}M^t}1\overline{\otimes}M^t$ (see, e.g., \cite[Lemma 9.2]{Ioa10}). Next, we amplify $\theta$ to get a $*$-homomorphism $\theta_0:=\theta^{1/t}:M\rightarrow M\overline{\otimes}M^t$. We further define  $\mathcal M=M^{\mathbb N}\overline{\otimes}(M^{\text{op}})^{\mathbb N}$ and $\theta_1:M\rightarrow M\overline{\otimes}M^t\overline{\otimes}\mathcal M$ by letting $\theta_1(x)=\theta_0(x)\otimes 1$.
Considering a natural isomorphism $M\overline{\otimes}M^t\overline{\otimes}\mathcal M\cong \mathcal M$, we can view $\theta_1$ as an embedding $\theta_1:M\rightarrow\mathcal M$. Since $\theta_1$  satisfies that $\theta_1(M)\nprec\Delta_n(M)$,  where $\Delta_n:M\rightarrow\mathcal M$  is the embedding into the $n^{\text{th}}$ tensor copy of $M$, for every $n\in\mathbb N$, we get a contradiction with Theorem \ref{no_comult1}.

(2)  Assume that $M^t=B\rtimes_{\sigma,c} G$, for some $t>0$ and a trace preserving cocycle action  $(\sigma,c)$ of a countably infinite group $G$ on a tracial von Neumann algebra with $c(G)''$ finite dimensional. Then Lemma \ref{cocyclefindim} provides an embedding $\theta:M\rightarrow (M\overline{\otimes}M\overline{\otimes}M^{\text{op}})^s$, for some $s>0$, such that $\theta(M)\nprec_{(M\overline{\otimes}M\overline{\otimes}M^{\text{op}})^s}S$, for every $S\in \{M^s\overline{\otimes}1\overline{\otimes}1,1\overline{\otimes}M^s\overline{\otimes}1,1\overline{\otimes}1\overline{\otimes}(M^{\text{op}})^s\}$. It is easy to derive a contradiction by arguing similarly to the proof of (1).

(3) Let $\Delta:M\rightarrow M^t$ be an embedding, for some $t>0$. 
Identify $M^t=p(M\overline{\otimes}\mathbb M_m(\mathbb C))p$, where $m\geq t$ is an integer and $p\in M\overline{\otimes}\mathbb M_m(\mathbb C)$ a projection of normalized trace $\frac{t}{m}$. Let $n\geq \frac{1}{t}$ be an integer and  $q\in\mathcal M\overline{\otimes}\mathbb M_n(\mathbb C)$ be a projection of normalized trace $\frac{1}{tn}$. 

Put $\mathcal N=(M\overline{\otimes}\mathbb M_m(\mathbb C))\overline{\otimes}(\mathcal M\overline{\otimes}\mathbb M_n(\mathbb C))$.
Then $p\otimes q\in \mathcal N$ is a projection of normalized trace $\frac{1}{mn}$ so it is equivalent to $(1\otimes e_{1,1})\otimes (1\otimes e_{1,1})$. Let $u\in \mathcal N$ be a unitary such that $u(p\otimes q)u^*=(1\otimes e_{1,1})\otimes (1\otimes e_{1,1})$. Define  $\Delta_0:M\rightarrow\mathcal N$
by letting $\Delta_0(x)=u(\Delta(x)\otimes q)u^*$. Then $\Delta_0(1)=(1\otimes e_{1,1})\otimes (1\otimes e_{1,1})$, and thus we can view $\Delta_0$ as an embedding of $M$ into $M\overline{\otimes}\mathcal M$, for which we have $\Delta_0(M)\nprec_{M\overline{\otimes}\mathcal M}1\overline{\otimes}\mathcal M$.

Using that $M\overline{\otimes}\mathcal M\cong\mathcal M$ and applying Theorem \ref{no_comult1}, we find a unitary $v\in (M\overline{\otimes}e_{1,1})\overline{\otimes}(\mathcal M\overline{\otimes}e_{1,1})$ such that $\Delta_0(x)=v((x\otimes e_{1,1})\otimes (1\otimes e_{1,1}))v^*$, for every $x\in M$. Hence, for every $x\in M$, we have 
$(\Delta(x)\otimes q)(u^*v)=(u^*v)((x\otimes e_{1,1})\otimes (1\otimes e_{1,1}))$. Therefore, $\zeta=(p\otimes q)(u^*v)\in\mathcal N$ is a partial isometry such that $\zeta\zeta^*=p\otimes q,\zeta^*\zeta=(1\otimes e_{1,1})\otimes (1\otimes e_{1,1})$ and
\begin{equation}\label{deltaa}\text{$(\Delta(x)\otimes q)\zeta=\zeta((x\otimes e_{1,1})\otimes (1\otimes e_{1,1}))$, for every $x\in M$.}\end{equation} Let $(\eta_i)_{i\in I}\subset q(\mathcal M\overline{\otimes}\mathbb M_n(\mathbb C))(1\otimes e_{1,1})$ be an orthonormal basis for
$\text{L}^2(q(\mathcal M\overline{\otimes}\mathbb M_n(\mathbb C))(1\otimes e_{1,1}))$.
Then we can   write $\zeta=\sum_{i\in I}\zeta_i\otimes\eta_i$, where $(\zeta_i)_{i\in I}\subset p(M\overline{\otimes}\mathbb M_m(\mathbb C))(1\otimes e_{1,1})$. Using \eqref{deltaa} we get that $\Delta(x)\zeta_i=\zeta_i(x\otimes e_{1,1})$, for every $i\in I$ and $x\in M$. 

Let $i\in I$ with $\zeta_i\not=0$. Then $\zeta_i^*\zeta_i\in\mathcal Z(M\overline{\otimes} e_{1,1})=\mathbb C\overline{\otimes} e_{1,1}$,  thus $\zeta_i^*\zeta_i=\lambda_i(1\otimes e_{1,1})$, for some $\lambda_i>0$. Therefore, $\xi_i=\lambda_i^{-\frac{1}{2}}\zeta_i$ is a partial isometry such that $\xi_i^*\xi_i=1\otimes e_{1,1}$ and $\Delta(x)\xi_i=\xi_i(x\otimes e_{1,1})$, for every  $x\in M$. 
Letting $p_i=\xi_i\xi_i^*\in\Delta(M)'\cap p(M\overline{\otimes}\mathbb M_m(\mathbb C))p$, we have that $\Delta(x)p_i=\xi_i(x\otimes e_{1,1})\xi_i^*$, for every $x\in M$. Hence, $\Delta(M)p_i=\xi_i(M\overline{\otimes} e_{1,1})\xi_i^*$, and taking relative commutants gives that $p_i(\Delta(M)'\cap p(M\overline{\otimes}\mathbb M_m(\mathbb C))p)p_i=\mathbb Cp_i$, that is, $p_i$ is a minimal projection of $\Delta(M)'\cap p(M\overline{\otimes}\mathbb M_m(\mathbb C))p$.

Let $(r_k)_{k=1}^K\subset \Delta(M)'\cap p(M\overline{\otimes}\mathbb M_m(\mathbb C))p$ be a maximal (necessarily finite, of cardinality $K\leq m$) family of pairwise orthogonal projections such that 
for every $1\leq k\leq K$, there is a partial isometry $\delta_k\in p(M\overline{\otimes}\mathbb M_m(\mathbb C))$ with $\delta_k\delta_k^*=r_k$, $\delta_k^*\delta_k=1\otimes e_{1,1}$ and $\Delta(x)r_k=\delta_k(x\otimes e_{1,1})\delta_k^*$, for every $x\in M$. 

We claim that $\sum_{k=1}^Kr_k=p$. Otherwise, assume that $r:=p-\sum_{k=1}^Kr_k\not=0$.  Since we have $(r\otimes 1)\zeta\zeta^*=(r\otimes 1)(p\otimes q)=r\otimes q\not=0$ and $(r\otimes 1)\zeta=\sum_{i\in I}r\zeta_i\otimes\eta_i$, we find $i\in I$ such that $r\zeta_i\not=0$. Thus, $r\xi_i\not=0$ and so $rp_i\not=0$.
Since $p_i$ is a minimal projection of $\Delta(M)'\cap p(M\overline{\otimes}\mathbb M_m(\mathbb C))p$, this provides a projection $s\in \Delta(M)'\cap p(M\overline{\otimes}\mathbb M_m(\mathbb C))p$ such that $s\leq r$ and $s$ is equivalent to $p_i$ in $\Delta(M)'\cap p(M\overline{\otimes}\mathbb M_m(\mathbb C))p$. This contradicts the maximality of the 
family $(r_k)_{k=1}^K$.

Finally, let $1\leq k\leq K$, and put $\delta_k'=\delta_k(1\otimes e_{1,k})$. Then $\Delta(x)r_k=\delta_k(x\otimes e_{1,1})\delta_k^*=\delta_k'(x\otimes e_{k,k})\delta_k'^*$, for every $x\in M$. Hence, if we let $\delta=\sum_{k=1}^K\delta_k'$ and $e=\sum_{k=1}^Ke_{k,k}$, then $\delta\delta^*=p,\delta^*\delta=1\otimes e$ and $$\text{$\Delta(x)=\sum_{k=1}^K\Delta(x)r_k=\sum_{k=1}^K\delta_k'(x\otimes e_{k,k})\delta_k'^*=\delta(x\otimes e)\delta^*$, for every $x\in M$.}$$
This finishes the proof of statement (3).

(4) The proof of this assertion is similar to the proof of (3), so we leave the details to the reader.

(5) Let $\Delta:M\rightarrow M\overline{\otimes}M$ be any comultiplication. 
Since $M$ is a II$_1$ factor, the uniqueness of its trace implies that $(\tau\otimes\tau)\Delta=\tau$.
By part (4), there exist projections $p_1,p_2\in M$  and unitaries $u_1,u_2\in M\overline{\otimes}M$ such that $\tau(p_1)+\tau(p_2)=1$ and $\Delta(x)=u_1(x\otimes p_1)u_1^*+u_2(p_2\otimes x)u_2^*$, for every $x\in M$. 
Then there exist unitaries $v_1,v_2,v_3,w_1,w_2,w_3\in M\overline{\otimes}M\overline{\otimes}M$ such that for every $x\in M$, $$(\Delta\otimes\text{Id})(\Delta(x))=\text{Ad}(v_1)(x\otimes p_1\otimes p_1)+\text{Ad}(v_2)(p_2\otimes x\otimes p_1)+\text{Ad}(v_3)(\Delta(p_2)\otimes x)$$ and 
$$\text{$(\text{Id}\otimes\Delta)(\Delta(x))=\text{Ad}(w_1)(x\otimes \Delta(p_1))+\text{Ad}(w_2)(p_2\otimes x\otimes p_1)+\text{Ad}(w_3)(p_2\otimes p_2\otimes x)$.}$$
Let $(u_n)\subset\mathcal U(M)$ be any sequence converging weakly to $0$. Using the last two formulas and the fact that $\|\Delta(p_1)\|_2=\sqrt{(\tau\otimes\tau)(\Delta(p_1))}=\sqrt{\tau(p_1)}$, it is easy to see that $$\sup_{z\in\mathcal U(M\overline{\otimes}M\overline{\otimes}M)}\Big(\limsup_{n\rightarrow\infty}\|\text{E}_{M\overline{\otimes}1\overline{\otimes}1}(\text{Ad}(z)(\Delta\otimes\text{Id})(\Delta(u_n)))\|_2\Big)=\|p_1\otimes p_1\|_2=\tau(p_1)$$ and $$\sup_{z\in\mathcal U(M\overline{\otimes}M\overline{\otimes}M)}\Big(\limsup_{n\rightarrow\infty}\|\text{E}_{M\overline{\otimes}1\overline{\otimes}1}(\text{Ad}(z)(\text{Id}\otimes\Delta)(\Delta(u_n)))\|_2\Big)=\|\Delta(p_1)\|_2=\sqrt{\tau(p_1)}.$$
Since  $\Delta$ is co-associative, 
we get that $\tau(p_1)=\sqrt{\tau(p_1)}$, hence $\tau(p_1)\in\{0,1\}$ and $p_1\in\{0,1\}$.

We assume first that $p_1=1$. Thus, $p_2=0$, so that $u=u_1\in M\overline{\otimes}M$ satisfies $\Delta(x)=u(x\otimes 1)u^*$, for every $x\in M$. 
Define $S=\{(\text{Id}\otimes\tau)(\Delta(x))\mid x\in M\}$ and $N=\{x\in M\mid \Delta(x)=x\otimes 1\}$.
Since $\Delta$ is co-associative and  $(\tau\otimes\tau)\Delta=\tau$, the proof of \cite[Lemma 2.7]{DV25} implies that $S\subset N$.

Let $(\xi_i)_{i\in\mathbb N}\subset M$ be an orthonormal basis for $\text{L}^2(M)$ and write $u=\sum_{i\in\mathbb N}v_i\otimes\xi_i$, where $(v_i)_{i\in\mathbb N}\subset M$. 
Since $u$ is a unitary, we get that $\sum_{i\in\mathbb N}v_iv_i^*=\sum_{i\in\mathbb N}v_i^*v_i=1$. 
Since $\Delta(x)=x\otimes 1=u(x\otimes 1)u^*$, we get that $x$ commutes with $(v_i)_{i\in \mathbb N}$, for every $x\in N$. In other words, $(v_i)_{i\in \mathbb N}\subset N'$.
Further, we have \begin{equation}\label{EN}\text{$(\text{Id}\otimes\tau)(\Delta(x))=(\text{Id}\otimes\tau)(u(x\otimes 1)u^*)=\sum_{i\in\mathbb N}v_ixv_i^*\in N$, for every $x\in M$.}\end{equation}
If $i\in \mathbb N$, then since $v_i\in N'$ we have that $\text{E}_N(v_iy)=\text{E}_N(yv_i)$, for every $y\in M$. By combining this fact with \eqref{EN} we derive that
\begin{equation}\label{ENN}\text{$\text{E}_N(x)=\sum_{i\in\mathbb N}\text{E}_N(xv_iv_i^*)=\sum_{i\in\mathbb N}\text{E}_N(v_ixv_i^*)=\sum_{i\in\mathbb N}v_ixv_i^*$, for every $x\in M$.}\end{equation}

We claim that $M\prec_MNz$, for every nonzero projection $z\in N'\cap M$. For $i\in\mathbb N$, denote $w_i=zv_iz$. Then by using \eqref{ENN} we get that 
\begin{equation}\label{ENz}
\text{$\text{E}_{Nz}(x)=\text{E}_N(xz)z=z\text{E}_N(zxz)z=\sum_{i\in\mathbb N}w_ixw_i^*$, for every $x\in M$.}
\end{equation}
To prove the claim, let $i\in\mathbb N$ with $w_i\not=0$. Let $q\in M$ be a nonzero projection such that $qw_i^*w_iq\geq cq$, for some $c>0$. Then for every projection $p\in qMq$ we have   $pw_i^*w_ip\geq cp$ and thus $\|w_ipw_i^*\|_2=\|pw_i^*w_ip\|_2\geq c\|p\|_2$. Together with \eqref{ENz} we get $\|\text{E}_{Nz}(p)\|_2\geq \|w_ipw_i^*\|_2\geq c\|p\|_2$, for every projection $p\in qMq$.  Lemma \ref{small_projections} gives that $qMq\prec_M Nz$, thus $M\prec_M Nz$, as claimed.

Next, if $z\in \mathcal Z(N'\cap M)$ is a nonzero projection then since $M\prec_M Nz$, \cite[Lemma 3.5]{Vae07} gives that $(N'\cap M)z=(Nz)'\cap zMz\prec_M \mathbb C1$. This implies that  $(N'\cap M)z$ contains a minimal projection. Since this holds for every $z\in \mathcal Z(N'\cap M)$, we conclude that $N'\cap M$ is completely atomic. 

Thus, $\mathcal Z(N)$ is completely atomic. Write  $\mathcal Z(N)=\oplus_{j\in J}\mathbb Cz_j$, for a partition of unity $(z_j)_{j\in J}$ in $\mathcal Z(N)$.
If $j\in J$, then $\mathcal Z(N)z_j=\mathbb Cz_j$, hence $Nz_j$ is a factor.
Since $z_jMz_j\prec_{z_jMz_j}Nz_j$, the inclusion $Nz_j\subset z_jMz_j$ has finite index. Since $M$ satisfies (1), by arguing as in the proof of Lemma \ref{finite_cross}, we get that the inclusion $Nz_j\subset z_jMz_j$  is isomorphic to $Nz_j\overline{\otimes}\mathbb C1\subset Nz_j\overline{\otimes}\mathbb M_{n_j}(\mathbb C)$, for some $n_j\in\mathbb N$.

Write $A=\oplus_{j\in J}(\mathbb C z_j\overline{\otimes}\mathbb M_{n_j}(\mathbb C))$. Then $A\subset M$ is a completely atomic von Neumann subalgebra with $A=N'\cap M$. Since $u\in (N\overline{\otimes}1)'\cap M\overline{\otimes}M$, we deduce that $u\in A\overline{\otimes}M$. This proves that (a) holds. Similarly, if we assume that $p_2=1$, then it follows that (b) holds.

(6) We only justify that  $M^{\text{op}}\not\hookrightarrow M^{\mathbb N}$, leaving the main assertion to the reader.   If  $M^{\text{op}}\hookrightarrow M^{\mathbb N}$, then  $M\hookrightarrow (M^{\text{op}})^{\mathbb N}$. Let $\theta:M\rightarrow (M^{\text{op}})^{\mathbb N}$ be any embedding and define $\Delta:M\rightarrow \mathcal M=M^{\mathbb N}\overline{\otimes}(M^{\text{op}})^{\mathbb N}$ by letting  $\Delta(x)=1\otimes\theta(x)$. Then $\Delta(M)\nprec_{\mathcal M}M^{\mathbb N}\overline{\otimes}1$, which contradicts Theorem \ref{no_comult1}.
\end{proof}

\subsection{Proof of Corollary \ref{C}} Corollary \ref{C} is a consequence of the following result:

\begin{corollary}\label{CC}
Let $M$ be a II$_1$ factor satisfying the conclusion of Theorem \ref{no_comult1}. Assume additionally that $M=M_1*_RM_2$, where $M_1,M_2$ are nonamenable II$_1$ factors and $R\subset M_1,R\subset M_2$ are coarse embeddings of the hyperfinite II$_1$ factor $R$. Let $t>0$. Then we have
\begin{enumerate}
\item  $M^t\not\cong\mathcal B\rtimes_{\sigma,c}\mathcal G$, for 
any trace preserving action $\mathcal G\curvearrowright^\sigma\mathcal B$ of an infinite discrete pmp groupoid $\mathcal G$ on a measurable field $\mathcal B=(B_x,\tau_x)_{x\in\mathcal G^{(0)}}$ of tracial von Neumann algebras and  $c\in Z^2(\mathcal G,\mathbb T)$.
\item If $M^t$ admits a regular  von Neumann subalgebra which is infinite dimensional and of infinite index, then $M^s$ admits a regular, irreducible subfactor of infinite index, for some $s\in (0,t]$.
\end{enumerate}
\end{corollary}

\begin{proof} The proof relies on the following claim:
\begin{claim}\label{regular}
    Let $B\subset M^t$ be a regular von Neumann subalgebra. If $B$ is amenable, then $B$ is finite dimensional. If $B$ is diffuse, then $B'\cap M$ is finite dimensional.
\end{claim}
    
Assume that $B$ is amenable.
As $M_1,M_2$ are nonamenable,  $M$ is nonamenable and $M^t\nprec_{M^t}M_i^t$, for all $i\in\{1,2\}$.
Using that $M=M_1*_RM_2$ and applying \cite[Theorem A]{Vae13} we get that $B\prec_{M^t}^sR^t$. 

We claim that $B$ is completely atomic.
Suppose by contradiction that there exists a projection $z\in\mathcal Z(B)$ such that $Bz$ is diffuse. Since $B$ is regular in the II$_1$ factor $M^t$, it follows that $B$ is diffuse.
Since the inclusions $R\subset M_1$ and $R\subset M_2$ are coarse, so are the inclusions $R\subset M$ and $R^t\subset M^t$.
Hence, the inclusion $R^t\subset M^t$ is mixing (see \cite[Definition 9.2]{Ioa12}). Since $B\prec_{M^t}R^t$, applying \cite[Lemma 9.4]{Ioa12}  gives that $M^t=\mathcal N_{M^t}(B)''\prec_{M^t}R^t$,  contradicting that $M^t$ is nonamenable. This proves the claim.  Since $B$ is regular in the II$_1$ factor $M^t$, it must be finite dimensional.

Second, assume that $B$ is diffuse. Suppose by contradiction that $B'\cap M^t$ is not finite dimensional. 
Then by the above, $B$ and $B'\cap M^t$ cannot be amenable. Since $B$ and $B'\cap M^t$ are regular in the II$_1$ factor $M^t$, we get that they have no amenable direct summands. 
 Since $B$ has no amenable direct summand, by applying \cite[Theorem 1.1]{CH08} we find $i\in\{1,2\}$ such that $B'\cap M^t \prec_{M^t}M_i^t$. Since $B'\cap M^t$ has no amenable direct summand, we also have that $B'\cap M^t\nprec_{M^t}R^t$. Thus, \cite[Theorem 1.1]{IPP05} implies that $M^t=\mathcal N_{M^t}(B'\cap M^t)''\prec_{M^t}M_i^t$, which is a contradiction.

We are now ready to prove assertions (1) and (2) of the corollary.

(1) Assume that $M^t=\mathcal B\rtimes_{\sigma,c}\mathcal G$, for some $t>0$, a trace preserving action $\sigma$ of a discrete pmp groupoid $\mathcal G$ on a measurable field $\mathcal B=(B_x,\tau_x)_{x\in X}$ of separable tracial von Neumann algebras indexed over the space of units $(X,\mu)$ of $\mathcal G$, and a $2$-cocycle $c\in Z^2(\mathcal G,\mathbb T)$. 
Then $B=\int^{\oplus}_{X}B_x\;\text{d}\mu(x)$ is a regular von Neumann subalgebra of $M^t$. Hence, $\mathcal Z(B)$ is a regular abelian subalgebra of $M^t$. 

By Claim \ref{regular}, $\mathcal Z(B)$ must be finite dimensional. 
Thus, $\text{L}^\infty(X)\subset\mathcal Z(B)$ is finite dimensional, and therefore $(X,\mu)$ is finite and atomic.
For $x\in X$, let $p_x={\bf 1}_{\{x\}}\in \text{L}^{\infty}(X)$ and $G_x$ be the isotropy group of $x$. Consider the action $G_x\curvearrowright^{\rho}(B_x,\tau_x)$ and $2$-cocycle $d\in Z^2(G_x,\mathbb T)$ obtained by restricting $\sigma$ and $c$ to $G_x$. Then $p_x(\mathcal B\rtimes_{\sigma,c}\mathcal G)p_x\cong B_x\rtimes_{\rho,d}G_x$.
Hence, $M^{t\tau(p_x)}\cong p_xM^tp_x\cong B_x\rtimes_{\rho,d}G_x$. By Corollary \ref{DD}(2), we derive that $G_x$ is a finite group. 

Since $X$ is a finite set and the isotropy groups $(G_x)_{x\in X}$ are finite, the groupoid $\mathcal G$ must be finite.

(2) Let $B\subset M^t$ be a regular von Neumann subalgebra which is infinite dimensional and of infinite index in the sense that $M^t\nprec_{M^t}B$. Claim \ref{regular} implies that $B$ has no amenable direct summand, and further that $B'\cap M^t$ is finite dimensional. Let $z\in\mathcal Z(B'\cap M)$ be a minimal projection.
Since $\mathcal Z(B)\subset\mathcal Z(B'\cap M)$, we get that $Bz$ is a II$_1$ factor.
If $u\in\mathcal N_{M^t}(B)$, then $uzu^*\in\mathcal Z(B'\cap M^t)$ and therefore either $uzu^*=z$ or $uzu^*z=0$. 
This implies that $Bz$ is a regular subfactor of $zM^tz$. Moreover, $A=(Bz)'\cap zM^tz$ is a finite dimensional factor. 

Let $N=A'\cap zM^tz$. Then $Bz\subset N$ and $(Bz)'\cap N=\mathcal Z((Bz)'\cap zM^tz)=\mathbb Cz$. Hence, $Bz\subset N$ is an irreducible subfactor. We claim that it is also regular. To this end, let $v\in\mathcal N_{zM^tz}(Bz)$. Then $vAv^*=A$, and since $A$ is a finite dimensional factor, there exists $w_v\in\mathcal U(A)$ such that $vxv^*=w_vxw_v^*$, for every $x\in A$. Hence, $w_v^*v\in N$ and so $w_v^*v\in\mathcal N_{N}(Bz)$. This implies that $v\in\mathcal N_N(Bz)\mathcal U(A)$, and hence $\mathcal N_{zM^tz}(Bz)\subset \mathcal U(A)\mathcal N_N(Bz)$. Since $A$ is a finite dimensional factor, we also have $zM^tz=A\overline{\otimes}N$. Using the last two facts and that $Bz$ is regular in $zM^tz$, we conclude that $zM^tz\subset A\overline{\otimes}\mathcal N_N(Bz)''$ and further that $\mathcal N_N(Bz)''=N$. 

Altogether, we have shown that $Bz\subset N$ is a regular, irreducible subfactor.
Since $M^t\nprec_{M^t}B$, we also have $N\nprec_NBz$, and thus the inclusion $Bz\subset N$ has infinite index.
Let $n\geq 1$ such that $A\cong\mathbb M_n(\mathbb C)$. Then $N=A'\cap zM^tz\cong M^s$, where $s=\frac{\tau(z)t}{n}\in (0,t]$,  finishing the proof of (2).
\end{proof}

\subsection{Proof of Corollary \ref{applicationsC}} Let $M$ be a separable II$_1$ factor satisfying Theorem \ref{no_comult1} and denote by $\tau$ its trace. Then \cite[Proposition 2.1]{Phi02} provides a simple,   separable $C^*$-algebra $A\subset M$ such that $A''=M$ and $\tau_{|A}$ is the unique tracial state of $A$. Moreover, $A$ can be taken to be stably finite, with stable rank one and real rank zero.
Below we will use the following corollary of $\tau_{|A}$ being the unique tracial state of $A$:  every unital $*$-homomorphism $\pi:A\rightarrow N$, where $(N,\varphi)$ is a tracial von Neumann algebra, satisfies $\varphi\circ\pi=\tau$ and thus extends to a $*$-homomorphism $\pi:M\rightarrow N$.

(1) Assume that $A^{\text{op}}$ embeds into $\otimes_{n\in\mathbb N}A$. Then viewing $\otimes_{n\in\mathbb N}A$ as a subalgebra of $M^{\mathbb N}$ and applying the above fact, we would deduce that $M^{\text{op}}$ embeds $M^{\mathbb N}$, which contradicts Corollary \ref{DD}(4).

(2) Assume  that $A=B\rtimes_{\sigma,c,r}G$, for some action $G\curvearrowright^\sigma B$ of a countably infinite group $G$ on a $C^*$-algebra $B$ and a $2$-cocycle $c\in Z^2(G,\mathbb T)$. Denote by $(u_g)_{g\in G}\subset \mathcal U(A)$ the canonical unitaries, by $(\lambda_g)_{g\in G}\subset\mathcal U(\text{L}(G))$ the left regular representation of $G$ and by $\varphi:\text{L}(G)\rightarrow\mathbb C$ the usual trace. 
Consider the unital $*$-homomorphism $\Delta_0:A\rightarrow M\overline{\otimes}\text{L}(G)$ given by $\Delta_0(bu_g)=bu_g\otimes \lambda_g$, for every $b\in B$ and $g\in G$. The above fact implies that $(\tau\otimes\varphi)\Delta_0(a)=\tau(a)$, for every $a\in A$. In particular, $\tau(bu_g)=\tau(bu_g)\varphi(\lambda_g)=\delta_{g,e}\tau(b)$, for every $b\in B$ and $g\in G$. This implies that $M=A''=B''\rtimes_{\sigma'',c}G$, where $G\curvearrowright^{\sigma''}B''$ denotes the extension of $\sigma$ to $B''$, which contradicts Corollary \ref{C}.

(3) For this part, we take $M$ as in Theorem \ref{embeddings}. Thus, in addition to satisfying the conclusion of Theorem \ref{no_comult1}, $M=M_1*_RM_2$, where $R\subset M_1$ and $R\subset M_2$ are coarse embeddings of the hyperfinite II$_1$ factor $R$ into property (T) II$_1$ factors $M_1$ and $M_2$. As in the proof of Corollary \ref{DD}(3), this guarantees that any regular abelian von Neumann subalgebra $P\subset M$ must be finite dimensional.

Assume that $A=C^*_r(\mathcal G,\Sigma)$, for a locally compact Hausdorff  \'{e}tale groupoid $\mathcal G$ and a twist $\Sigma$ over $\mathcal G$.  Let $X=\mathcal G^{(0)}$ be the unit space of $\mathcal G$. Then $A_0=C_0(X)$ is a regular abelian $C^*$-subalgebra of $A$. Hence, $A_0''$ is a regular abelian von Neumann subalgebra of $M=A''$. By the above paragraph, $A_0''$ must be finite dimensional. Therefore, $A_0$ is finite dimensional, hence $X$ is finite and $\mathcal G$ is discrete. For $x\in X$, denote by $G_x$ the isotropy group of $x$ and $\Sigma_x=\Sigma_{|G_x}\in Z^2(G_x,\mathbb T)$. 
Since $x\in X$ is an isolated point, $p_x={\bf 1}_{\{x\}}\in A_0$ is a nonzero projection. 

Fix $x_0\in X$. Since $M$ is a factor it follows that for every $x\in X$ there exists $g\in\mathcal G$ such that $s(g)=x_0$ and $r(g)=x$. Hence, $p_x$ and $p_{x_0}$ are equivalent in $A$, for every $x\in X$. Therefore, denoting $n=|X|$ we have that $A\cong \mathbb M_n(p_{x_0}Ap_{x_0})$.  Since $A$ has a unique tracial state, we deduce that $p_{x_0}Ap_{x_0}$ has a unique tracial state.  
Moreover, \cite[Lemma 4.2]{BGP26} implies that $p_xAp_x\cong C^*_{r,\Sigma_x}(G_x)$.  Arguing as in (2), we get that $p_xMp_x=(p_xAp_x)''\cong \text{L}_{\Sigma_x}(G_x)$. This contradicts Corollary \ref{DD}(2). \hfill$\square$

\end{document}